\documentclass[10pt]{article}
\usepackage{paper,macros,graphicx}
\numberwithin{equation}{section} 

\usepackage{epstopdf}
\usepackage{url}
\usepackage{xcolor}
\usepackage{subcaption} 
\usepackage[colorlinks=true]{hyperref}
\usepackage{framed}
\usepackage{bbm}
\usepackage{float}

\def\R{\mathbb{R}}

\def\E{\mathbb{E}}

\usepackage{tikz}
\usepackage{pifont}
\usepackage{array}
\usepackage{arydshln}
\usepackage[normalem]{ulem} 

\usepackage{enumitem}

\newlength{\leftstackrelawd}
\newlength{\leftstackrelbwd}
\def\leftstackrel#1#2{\settowidth{\leftstackrelawd}%
{${{}^{#1}}$}\settowidth{\leftstackrelbwd}{$#2$}%
\addtolength{\leftstackrelawd}{-\leftstackrelbwd}%
\leavevmode\ifthenelse{\lengthtest{\leftstackrelawd>0pt}}%
{\kern-.5\leftstackrelawd}{}\mathrel{\mathop{#2}\limits^{#1}}}
\usepackage[colorlinks=true]{hyperref}
\usepackage[numbers, sort&compress ]{natbib}

\title{First- and Second-Order High Probability Complexity Bounds for Trust-Region Methods with Noisy Oracles}
\author{L. Cao\footnotemark[1]\ \footnotemark[4]
   \and  A. S. Berahas\footnotemark[2]
   \and K. Scheinberg\footnotemark[3]}

\begin{document}

\renewcommand{\thefootnote}{\fnsymbol{footnote}}
\footnotetext[1]{Beijing International Center for Mathematical Research, Peking University, Beijing, China;  \href{mailto:caoliyuan@bicmr.pku.edu.cn}{caoliyuan@bicmr.pku.edu.cn}}
\footnotetext[2]{Department of Industrial and Operations Engineering, University of Michigan, Ann Arbor, MI, USA; \href{mailto:albertberahas@gmail.com}{albertberahas@gmail.com}}
\footnotetext[3]{School of Operations Research and Information Engineering, Cornell University, Ithaca, NY, USA;  \href{mailto:katyas@cornell.edu}{katyas@cornell.edu}}
\footnotetext[4]{Corresponding author}
\renewcommand{\thefootnote}{\arabic{footnote}}
\maketitle

\begin{abstract}
  
In this paper, we present convergence guarantees for a modified trust-region method designed for minimizing objective functions whose value and gradient and Hessian estimates are computed with noise. 
These estimates are produced by generic stochastic oracles, which are not assumed to be unbiased or consistent. 
We introduce these oracles and show that they are more general and have more relaxed assumptions than the stochastic oracles used in prior literature on stochastic trust-region methods. 
Our method utilizes  a relaxed step acceptance criterion and a cautious trust-region radius updating strategy which allows us to derive exponentially decaying tail bounds on the iteration complexity for convergence to points that satisfy approximate first- and second-order optimality conditions. 
Finally, we present two sets of numerical results. We first explore the tightness of our theoretical results on an example with adversarial zeroth- and first-order oracles. We then investigate the performance of the modified trust-region algorithm on standard noisy derivative-free optimization problems.
\end{abstract}

\section{Introduction}
The trust-region (TR) methods form a well-established class of iterative numerical methods for optimizing nonlinear continuous functions. 
In each iteration, TR methods minimize an approximation model of the objective function, often a quadratic model, within a trust-region.
The book \cite{TRbook} contains an exhaustive coverage of these methods up to the time of its publication, and \cite{yuan2015recent} offers a more recent survey. 
In this article, we consider the behavior of TR methods on unconstrained continuous optimization problems
\begin{align} \label{eq.prob}
	\min_{x \in \mathbb{R}^n} \phi(x), 
\end{align}
where $\phi:\mathbb{R}^n \rightarrow \mathbb{R}$ is (twice) differentiable with Lipschitz continuous derivatives, but neither the objective function nor its associated derivatives are assumed to be computable accurately.

Under the condition that the function $\phi$ and its derivatives can be evaluated accurately, the complexity analysis of TR methods (for solving problem \eqref{eq.prob}) is well-established \cite{TRbook}. 
However, this condition is not always satisfied in the real-world problems. 
For example, in derivative-free optimization (DFO), the function $\phi$ is often a mapping between the input and the output of a computer program, such as a simulation of the physical world or the training of a machine learning model. 
As a result, there can be noise in the evaluation of $\phi$ due to limited numerical precision or randomness. 
Furthermore, the derivatives of $\phi$ cannot be computed directly and, when needed, can only be estimated using zeroth-order information. 
Another example is empirical risk minimization (ERM), where the objective function is the average of many functions, i.e., $\phi(x) = \frac{1}{N} \sum_{i=1}^N l(x,d_i)$. 
In settings where the number of functions $N$ is large (or even infinite, in which case the average becomes an expectation), it is typical to use the average of a subset of the functions, i.e., $\phi_{\mathcal{B}}(x) = \frac{1}{|\mathcal{B}|} \sum_{d_i\in\mathcal{B}}l(x,d_i)$, where $\mathcal{B} \subseteq \{d_1,\dots, d_N\}$, instead of the true objective function $\phi$ to reduce the computational effort, at the cost of introducing errors in the evaluations.

When dealing with such problems, algorithms utilize various, usually stochastic, approximations of the objective function and derivative information ($\phi(x)$, $\nabla \phi(x)$ and $\nabla^2 \phi(x)$) in lieu of their exact counterparts. 
These approximations vary in terms of quality and reliability, and a variety of algorithms, not just TR methods, have been proposed and analyzed under different assumptions on the approximations employed.  
In order to give an overview of existing works and to clearly describe the contributions of this paper, we find it convenient to first propose and define general oracles that compute approximations of $\phi(x)$, $\nabla \phi(x)$ and $\nabla^2 \phi(x)$,  and then discuss the different assumptions on these oracles made in prior literature as compared to those made in this paper. Definition~\ref{def:oracle} presents the general form of the stochastic oracles considered.
\begin{definition}[\bf Stochastic $j$th-order oracle over a set $\mathcal{S}_j$] \label{def:oracle}
We say that the $j$th-order oracle $\varphi_j$ is implementable over a set $\mathcal{S}_j \subseteq [0,+\infty) \times [0,1]$ if it is capable of producing an estimate of $\nabla^j\phi$ that satisfies 
\begin{equation} \label{eq:oracle}
	\mathbbm{P}_{\xi^{(j)}}\left\{\|\varphi_j(x,\xi^{(j)})- \nabla^j \phi(x)\| \leq A_j\right\} \geq p_j 
\end{equation} 
for each $(A_j,p_j)\in \mathcal{S}_j$ and any $x\in\mathbb{R}^n$, 
where $\xi^{(j)}$ is a random variable defined on some probability space whose distribution depends on $(A_j,p_j)$ and $x$, and $\mathbbm{P}_{\xi^{(j)}}$ denotes probability with respect to that distribution. 
\end{definition}

Here $\nabla^j \phi(x)$ denotes the $j$-th derivative of $\phi(x)$, where $\nabla^0 \phi(x)$ reduces to $\phi(x)$. 
Clearly, if $(A_j, p_j)\in \mathcal{S}_j$ for some  $(A_j, p_j)$, then $[A_j, +\infty) \times [0,p_j] \subseteq \mathcal{S}_j$. 
The cost of implementing the oracles depends on the application and will not be considered in this paper. 
In most applications, the actual cost depends on $(A_j,p_j)$ and monotonically decreases as $A_j$ increases and $p_j$ decreases. 
We use DFO and ERM as examples to discuss how to implement oracles satisfying \eqref{eq:oracle} for a given $(A_j, p_j)$ in Section~\ref{sec:oracles}. 
That being said, the focus of this paper is on the analysis of TR methods under weak  assumptions on the oracles in terms of the sets $S_j$. The results will thus apply whenever an oracle is implementable over the corresponding sets. 

There are a variety of TR and line search algorithms in the literature that rely on stochastic oracles of this general form, although they are typically posed in a different way, specialized for each paper. 
For example, \cite{bandeira2014convergence,gratton2018complexity} analyze the convergence of TR methods under the assumption that the zeroth-order oracle gives exact function values, and, at each iteration $k$, the first-order oracle can produce, with sufficiently high probability,  a gradient estimate whose error is no more than the TR radius ($\delta_k$) multiplied by a constant. Since $\delta_k$ is not guaranteed to be bounded away from zero, these works effectively assume that the first-order oracle can produce gradient estimates with arbitrarily high precision, albeit only with a sufficiently high probability. 
Using our definition of stochastic oracles, \cite{bandeira2014convergence,gratton2018complexity} assume their zeroth-order oracles are implementable over the (largest possible) set $\mathcal{S}_0$ that contains $(0,1)$, and their first-order oracles are implementable over $(0,\infty) \times [0, \bar p_1]$, where $\bar p_1$ is sufficiently large (e.g., $\bar p_1>\frac{1}{2}$ in  \cite{bandeira2014convergence}). 
In contrast, \cite{STORM,blanchet2019convergence} analyze a first-order TR method under the assumption that both zeroth- and first-order oracles are implementable over $(0,\infty)\times [0,\bar p_j]$, $j = 0,1$, for sufficiently large $\bar p_0$ and $\bar p_1$.  
In \cite{blanchet2019convergence}, a second-order TR method is also analyzed under the assumption that $\mathcal{S}_j = (0,\infty)\times [0,\bar p_j]$, $j=0,1,2$, for sufficiently large $\bar p_0$, $\bar p_1$ and $\bar p_2$, with an additional bound on ${\mathbb E_{\xi_0}}[|\varphi_0(x,\xi_0)-\phi(x)|]]$, which requires a larger set $\mathcal{S}_0$. 
Finally, we mention a recent technical report \cite{sun2022noisy} where a first-order TR method with a relaxed step acceptance criterion is analyzed for oracles with deterministically bounded noise, i.e., $\mathcal{S}_0 = [\epsilon_f,+\infty) \times [0,1]$ and $\mathcal{S}_1= [\epsilon_g, +\infty) \times [0,1]$ for some positive constants $\epsilon_f$ and $\epsilon_g$, which represent the irreducible upper bounds on the noise in the oracles.  

Examples of ``line search''-like methods\footnote{These methods change the random  search direction even when reducing step size, thus, should not be strictly called ``line search'' methods. In \cite{jin2021high}, the  term ``step search'' was introduced to distinguish the two classes of methods.} based on stochastic oracles include the following. The authors in \cite{cartis2018global} analyze a stochastic line search method in the same oracle setting as in \cite{bandeira2014convergence,gratton2018complexity}. 
In \cite{paquette2020stochastic}, a modified stochastic line search method is proposed and analyzed in a setting similar to that of \cite{blanchet2019convergence}.
A line search method with relaxed Armijo condition is analyzed in \cite{berahas2019derivative} under deterministic oracle assumptions, i.e., the sets $\mathcal{S}_0$ and $\mathcal{S}_1$ contain pairs $(\epsilon_f,1)$ and $(\epsilon_g,1)$, respectively, for some positive constants $\epsilon_f$ and $\epsilon_g$. 
In \cite{berahas2021global}, the analysis of the line search method with relaxed Armijo condition is extended to less restrictive oracle settings where $\mathcal{S}_0 = [\epsilon_f, +\infty) \times [0,1]$ and $\mathcal{S}_1 = [0,+\infty) \times [0,\bar p_1]$ for some positive $\epsilon_f$ and sufficiently large $\bar p_1$. 
 
In a recent paper \cite{jin2021high}, the same line search as in \cite{berahas2021global,berahas2019derivative} was analyzed under weaker oracle conditions. In particular, $\mathcal{S}_1 = [\epsilon_g,+\infty) \times [0, \bar{p}_1]$, where $\epsilon_g$ (some positive constant) is the best possible accuracy that the first-order oracle can guarantee and $\bar{p}_1$ is some constant larger than $1/2$. 
The zeroth-order oracle is a bit more complex: specifically, it is assumed to be implementable over $\mathcal{S}_0 = \{(A_0, p_0):~ A_0\geq \epsilon_f \text{ and } p_0 \le 1-\exp(a(\epsilon_f-A_0))\}$\footnote{This is a simplified version of the actual zeroth-order oracle in \cite{jin2021high}.}, for some $a>0$ and $\epsilon_f \ge 0$. 
This means that the errors in the function value estimates may not be bounded by $\epsilon_f$, but the distribution of the estimates is such that the probability of this error being larger than $\epsilon_f$, while positive, decays exponentially. 
A more extensive discussion and comparison of different oracles is presented in Section~\ref{sec:oracles}. 

While most of the prior papers provide the analysis of expected iteration complexity of the corresponding algorithms, in \cite{jin2021high}, as well as \cite{gratton2018complexity}, high probability (with exponentially decaying tail) bounds on the iteration complexity are derived. 
In this paper we draw inspiration from \cite{jin2021high} and propose a TR method with a relaxed step acceptance criterion, and provide convergence guarantees under similar oracle conditions. In summary, our improvement upon the latest literature is as follows:
\begin{itemize}
\item In comparison to \cite{blanchet2019convergence}, we allow irreducible noise in the zeroth-, first-, and second-order oracles, so our analysis applies to problems where evaluation noise cannot be reduced below certain level. 
Our zeroth-order oracle is both more relaxed because it allows irreducible noise and yet somewhat stronger because it assumes a light-tailed distribution of the noise.   This allows us to derive high probability complexity bounds for both first- and second-order versions of the algorithm as compared to bounds in expectation. 

\item In comparison to  \cite{gratton2018complexity}, we use more relaxed first- and second-order oracles, by allowing irreducible noise, and also a significantly more relaxed zeroth-order oracle, as it is assumed to be exact in \cite{gratton2018complexity}. 

\item In comparison to  \cite{jin2021high}, which analyzed a first-order step search method, we analyze first- and second-order TR methods. The difference between the complexity analyses for the two types of algorithms is significant, particularly in the presence of irreducible noise. 
New analytical techniques are employed to derive our results,  e.g., the step size threshold $\bar\alpha$ in \cite{jin2021high} is a constant whereas the TR radius threshold $\bar\Delta$ in this paper is a random variable.  Second order analysis presented here is the first such analysis for irreducible noise. 
\end{itemize}
In terms of iteration complexity, we obtain similar bounds to those in \cite{gratton2018complexity,blanchet2019convergence,jin2021high}.  It is important to note that in this paper, as in many others that we discuss above,  we analyze iteration complexity under specific assumptions on the oracles, but without directly accounting for the oracle costs.  All the algorithms mentioned above are designed to update oracle accuracy (and thus their costs) adaptively, which makes the algorithms more practical, but harder to analyze in terms of the total oracle cost (or work). 

\textbf{Organization} \quad
The paper is organized as follows. 
In Section~\ref{sec:oracles}, we introduce the assumptions and oracles, as well as motivate the oracles with examples from the literature. 
In Section~\ref{sec:algorithm}, we introduce our modified TR algorithms and present some preliminary technical results and describe the stochastic process used to analyze the algorithms. 
The high probability tail bounds on the iteration complexity of the first- and second-order algorithms are presented in Sections~\ref{sec:analysis1} and \ref{sec:analysis2}, respectively. 
In Section~\ref{sec:dual}, we present  synthetic numerical experiments that simulate the worst-case behavior allowed under our oracle assumptions to support our theoretical findings. 
Finally, we test a practical TR algorithm with our proposed modification numerically in Section~\ref{sec:r}. 
We summarize the article in Section~\ref{sec:conclusion}.

\section{Assumptions and oracles}
\label{sec:oracles}
We consider the unconstrained optimization problem \eqref{eq.prob} with the following assumptions on $\phi$. Let $\langle \cdot,\cdot \rangle$ denote the sum of entry-wise products and $\|\cdot\|$ the 2-norm.
 \begin{assumption}[\textbf{Lipschitz-smoothness}]	\label{assum:lip_cont}  
The function $\phi$ is continuously differentiable, and the gradient of $\phi$ is $L_1$-Lipschitz continuous on $\mathbb{R}^n$, i.e., $\|\nabla\phi(y) - \nabla\phi(x)\| \le L_1 \|y-x\|$ for all $(y,x) \in \mathbb{R}^n\times\mathbb{R}^n$.
\end{assumption}
\begin{assumption}[\textbf{Lipschitz continuous Hessian}]	\label{ass:lip_cont2} 
The function $\phi$ is twice continuously differentiable, and the Hessian of $\phi$ is $L_2$-Lipschitz continuous on $\mathbb{R}^n$, i.e., $\|\nabla^2\phi(y) - \nabla^2\phi(x)\| \le L_2 \|y-x\|$ for all $(y,x) \in \mathbb{R}^n\times\mathbb{R}^n$. 
\end{assumption}
\begin{assumption}[\textbf{Lower bound on $\pmb{\phi}$}] \label{assum:low_bound} 
The function $\phi$ is bounded below by a scalar $\hat{\phi}$ on $\mathbb{R}^n$. 	
\end{assumption}

Our algorithms utilize approximations of $\phi$, $\nabla\phi$, and $\nabla^2\phi$ obtained via stochastic oracles. 
We assume our zeroth-, first- and second-oracles have the following properties in terms of accuracy and reliability. 

\begin{customoracle}{0}[\textbf{Stochastic zeroth-order oracle}]\label{oracle.zero}
Given a point $x \in \mathbb{R}^n$,  the oracle computes $f(x,\xi^{(0)})$, a (random) estimate of the function value $\phi(x)$, where $\xi^{(0)}$ is a random variable whose distribution may depend on $x$. 
Let $e(x,\xi^{(0)}) = f(x,\xi^{(0)}) - \phi(x)$. 
For any $x \in \mathbb{R}^n$, $e(x,\xi^{(0)})$ satisfies at least one of the two conditions:
\begin{enumerate}
    \item \label{ass:bounded noise} 
    {\bf (Deterministically bounded noise: zeroth-order oracle implementable over $\mathcal{S}_0 = [\epsilon_f,\infty) \times [0,1]]$)} There is a constant $\epsilon_f \geq 0$ such that $|e(x,\xi^{(0)})| \le \epsilon_f$  for all realizations of $\xi^{(0)}$. 
    \item \label{ass:subexponential noise} 
    {\bf (Independent subexponential noise: zeroth-order stochastic oracle implementable over $\mathcal{S}_0 = \{(A_0,p_0):~ A_0\ge\epsilon_f \text{ and } p_0\le1-\exp(a(\epsilon_f-A_0))\}$)} There are constants $\epsilon_f \ge 0$ and $a > 0$ such that $\mathbbm{P}_{\xi^{(0)}} \left\{ |e(x,\xi^{(0)})| > t \right\} \le \exp (a(\epsilon_f - t))$ for all $t \ge 0$. 
\end{enumerate}
\end{customoracle}
\begin{remark}
The two oracles introduced above will be referred to as Oracle~\ref{oracle.zero}.\ref{ass:bounded noise} and Oracle~\ref{oracle.zero}.\ref{ass:subexponential noise}, respectively. 
When Oracle~\ref{oracle.zero}.\ref{ass:bounded noise} is implemented for any $x\in\mathbb{R}^n$, it returns an estimate of $\phi(x)$ with bounded noise. 
This case includes deterministic or even adversarial noise, as long as it is bounded by $\epsilon_f$. 
Otherwise, when Oracle~\ref{oracle.zero}.\ref{ass:subexponential noise} is implemented, the cumulative distribution of the noise has a subexponential tail whose rate of decay is governed by $a$, but is unrestricted on the interval $[-\epsilon_f, \epsilon_f]$ as the right-hand side $\exp (a(\epsilon_f - t)) \ge 1$ when $t \le \epsilon_f$. The constants $\epsilon_f$ and $a$ are considered to be intrinsic to the oracle. 
\end{remark}

\begin{customoracle}{1}[\textbf{Stochastic first-order oracle implementable over ${\cal S}_1 = (\epsilon_g, \infty) \times [0,p_1]$}]\label{oracle.first}
Given $\delta^{(1)}>0$, a probability $p_1\in[0.5,1]$, and a point $x \in \mathbb{R}^n$, the oracle computes $g(x,\xi^{(1)})$, a (random) estimate of the gradient $\nabla \phi(x)$ that satisfies 
\begin{equation}\label{eq:first_order}
{\mathbb P}_{\xi^{(1)}}\left\{\|g(x,\xi^{(1)})-\nabla \phi(x)\| \leq \epsilon_g+\kappa_{\rm eg} \delta^{(1)} \right\} \geq p_1,
\end{equation}
where $\xi^{(1)}$ is a random variable (whose distribution may depend on the input $x$ and $\delta^{(1)}$). 
\end{customoracle}

\begin{customoracle}{2}[\textbf{Stochastic second-order oracle implementable over ${\cal S}_2 = (\epsilon_H, \infty)\times [0,p_2]$}]\label{oracle.second}
Given $\delta^{(2)}>0$, a probability $p_2\in[0.5,1]$  and a point $x \in \mathbb{R}^n$, the oracle computes $H(x,\xi^{(2)})$, a (random) estimate of the Hessian $\nabla^2 \phi(x)$, such that 
\begin{equation}\label{eq:second_order}
{\mathbb P}_{\xi^{(2)}}\left\{ \|H(x,\xi^{(2)})-\nabla^2\phi(x)\|\leq \epsilon_H+\kappa_{\rm eh} \delta^{(2)} \right\} \geq p_2,
\end{equation}
where $\xi^{(2)}$ is a random variable (whose distribution may depend on the input $x$ and $\delta^{(2)}$). 
\end{customoracle}

\begin{remark} 
The inputs to the first- and second-order oracles are the triplets $(\delta^{(j)},p_j,x)$, $j=1,2$. 
The constants $\epsilon_g$, $\epsilon_H$, $\kappa_{\rm eg}$ and $\kappa_{\rm eh}$ are nonnegative and are intrinsic to the oracles\footnote{``eg'' and ``eh'' stand for ``\underline{e}rror in the \underline{g}radient'' and ``\underline{e}rror in the \underline{H}essian'', respectively.}.  
Note that $\epsilon_g$ and $\epsilon_H$ limit the achievable accuracy, thus allowing the oracle to have error up to $\epsilon_g$ and $\epsilon_H$, respectively, with probability up to 1. 
The positive values $\delta^{(1)}$ and $\delta^{(2)}$ will be chosen dynamically by the algorithm, according to the TR radius. 
The probabilities $p_1$ and $p_2$ will be chosen to be constant and will need to satisfy certain bounds which will be derived in the theoretical analyses of our algorithms.
\end{remark}

These oracle definitions are special cases of the oracle defined in the introduction (Definition~\ref{def:oracle}). Henceforth, we will use  $f(x,\xi^{(0)})$, $g(x,\xi^{(1)})$ and  $H(x,\xi^{(2)})$ to denote the outputs of the oracles (instead of $\varphi_j(x,\xi^{(j)})$, $j=0,1,2$, in Definition~\ref{def:oracle}). 
The expressions $g(x,\xi^{(1)})$ and $H(x,\xi^{(2)})$ will be further abbreviated to $g(x)$ and $H(x)$ in the rest of Section~\ref{sec:oracles}, and their realizations will be denoted by $g_k$ and $H_k$ once they are put in the context of algorithms (Section~\ref{sec:algorithm} and beyond), where $k$ is the iteration index. 
Similarly, $f(x,\xi^{(0)})$ will be abbreviated as $f(x)$. 
The realizations of $f(x_k,\xi^{(0)})$ and $e(x_k,\xi^{(0)})$ will be denoted by $f_k$ and $e_k$, respectively. 

\subsection{Stochastic oracles used in prior literature}
We now discuss different stochastic oracles used in prior literature for unconstrained optimization and  show how our stochastic oracle definition (Definition~\ref{def:oracle}) relates to them by comparing their respective sets ${\cal S}_j$, $j=0,1,2$ to those used by Oracles \ref{oracle.zero}, \ref{oracle.first} and \ref{oracle.second}. 
A visual comparison of various such sets considered in the literature is provided in Figure~\ref{fig:phase1}. Each subfigure of Figure~\ref{fig:phase1} shows a different set ${\cal S}_j$ on the $A_j$-$p_j$ plane.

There is an important difference between the sets in Figure~\ref{fig:oracle0pclosed} and Figure~\ref{fig:oracle0popen}. 
When a stochastic oracle is implementable over the first set (Figure~\ref{fig:oracle0pclosed}), it means that it is possible to evaluate the function or its derivative exactly, with probability at least $\bar{p}_j$. 
In contrast, when a stochastic oracle is implementable over the second set (Figure~\ref{fig:oracle0popen}), it is assumed that the upper bound on the error (that holds with probability $\bar{p}_j$) can be made arbitrarily small but never zero. 
An oracle of this second type appears naturally when, for example, $\phi$ is an expectation over a distribution from which one can obtain arbitrarily large number of samples.  Figures~\ref{fig:oraclee1} and \ref{fig:oraclesubexp} depict the two conditions assumed for Oracle~\ref{oracle.zero}, the zeroth-order oracle used in this paper. 
The conditions assumed for Oracle~\ref{oracle.first} and \ref{oracle.second} are both depicted in Figure~\ref{fig:oracleepopen}, with $\epsilon_1 = \epsilon_g + \kappa_{\rm eg}\delta^{(1)}$ and $\epsilon_2 = \epsilon_H + \kappa_{\rm eh} \delta^{(2)}$, respectively. 
 
\begin{figure}[tbhp]
  \centering
  \subfloat[$\mathcal{S}_j \ni (0,1)$]{\label{fig:oracle01} \resizebox{0.25\linewidth}{!}{
    \begin{tikzpicture}
   \shade[shading angle=60] (0,0) -- (2.5,0) -- (2.5,1) -- (0,1);
   \draw[->] (-0.2, 0) -- (2.5, 0) node[right] {$A_j$};
   \draw[->] (0, -0.2) -- (0, 1.4) node[above] {$p_j$};
   \filldraw[black] (0,1) circle (2pt) node[anchor=south west]{$(0,1)$};
   \draw (0, 1) -- (2.5, 1);
\end{tikzpicture}
}}
  \hfill
  \subfloat[$\mathcal{S}_j=[0,+\infty)\times{[0,\bar{p}_j]}$]{\label{fig:oracle0pclosed}\resizebox{0.24\linewidth}{!}{
    \begin{tikzpicture}
     \shade[shading angle=60] (0,0.02) -- (2.5,0) -- (2.5,0.8) -- (0.02,0.8);
     \draw[->] (-0.2, 0) -- (2.5, 0) node[right] {$A_j$};
     \draw[->] (0, -.2) -- (0, 1.4) node[above] {$p_j$};
     \draw[black] (0,0.8) circle (2pt) node[anchor=north west]{$(0,\bar p_j)$};
     \draw (0, 1) -- (2.5, 1);
 \end{tikzpicture}
}}
  \hfill
  \subfloat[$\mathcal{S}_j=(0,+\infty)\times{[0,\bar{p}_j]}$]{\label{fig:oracle0popen}\resizebox{0.25\linewidth}{!}{
    \begin{tikzpicture}
      \shade[shading angle=60] (0,0.01) -- (2.5,0) -- (2.5,0.8) -- (0.01,0.8);
      \draw[->] (-0.2, 0) -- (2.5, 0) node[right] {$A_j$};
      \draw (0, -0.2) -- (0, 0);
      \draw[dashed] (0, 0) -- (0, 0.8);
      \draw[->] (0, 0.87) -- (0, 1.4) node[above] {$p_j$};
      \draw[black] (0,0.8) circle (2pt) node[anchor=north west]{$(0,\bar p_j)$};
      \draw (0, 1) -- (2.5, 1);
    \end{tikzpicture}
  }}\\
  \subfloat[$\mathcal{S}_j = (\epsilon_j, +\infty)\times{[0,\bar{p}_j]}$]{\label{fig:oracleepopen}\resizebox{0.24\linewidth}{!}{
    \begin{tikzpicture}
     \shade[shading angle=60] (0.52,0.02) -- (2.5,0) -- (2.5,0.8) -- (0.52,0.8);
     \draw[->] (-0.2, 0) -- (2.5, 0) node[right] {$A_j$};
     \draw[->] (0, -0.2) -- (0, 1.4) node[above] {$p_j$};
     \draw[dashed] (0.52, 0) -- (0.52, 0.8);
     \draw[black] (0.5,0.8) circle (2pt) node[anchor=north west]{$(\epsilon_j,\bar p_j)$};
     \draw (0, 1) -- (2.5, 1);
    \end{tikzpicture}
}}
  \hfill
  \subfloat[$\mathcal{S}_j=[\epsilon_j,+\infty)\times{[0,1]}$]{\label{fig:oraclee1} \resizebox{0.25\linewidth}{!}{
    \begin{tikzpicture}
      \shade[shading angle=60] (0.5,0) -- (2.5,0) -- (2.5,1) -- (0.5,1);
      \draw[->] (-0.2, 0) -- (2.5, 0) node[right] {$A_j$};
      \draw[->] (0, -0.2) -- (0, 1.4) node[above] {$p_j$};
      \filldraw[black] (0.5,1) circle (2pt) node[anchor=south west]{$(\epsilon_j,1)$};
      \draw (0, 1) -- (2.5, 1);
    \end{tikzpicture}
  }}
  \hfill
  \subfloat[\raggedright$\mathcal{S}_j=\{(A_j,p_j):A_j\ge\epsilon, p_j \le 1 - \exp(a(\epsilon-A_j))\}$]{\label{fig:oraclesubexp}\resizebox{0.24\linewidth}{!}{
    \begin{tikzpicture}
      \shade[shading angle=60, domain=0.8:2.5, variable=\x] (0.5,0) -- plot ({\x}, {1-2.71828^(1.5*(0.5-\x))}) -- (2.5,0) -- cycle;
      \draw[->] (-0.2, 0) -- (2.5, 0) node[right] {$A_j$};
      \draw[->] (0, -0.2) -- (0, 1.4) node[above] {$p_j$};
      \filldraw[black] (0.5,0) circle (2pt) node[anchor=south]{$(\epsilon_j,0)$};
      \draw (0, 1) -- (2.5, 1);
    \end{tikzpicture}
  }}
 \caption{A visual comparison of various assumptions on the oracle. Each subfigure shows a different set $\mathcal{S}_j$ on the $A_j$-$p_j$ plane.}
 \label{fig:phase1}
\end{figure}
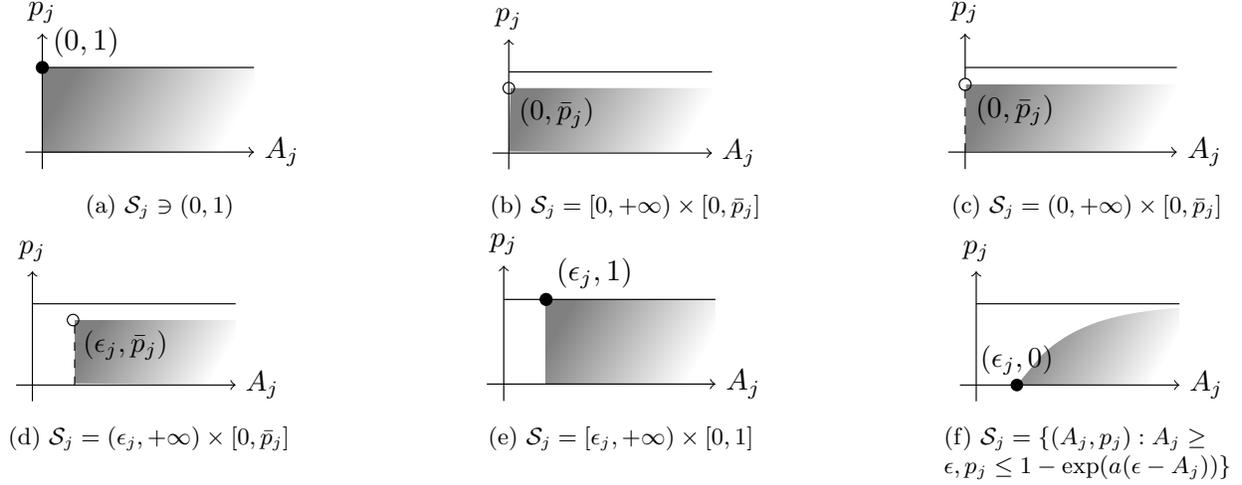

\paragraph{Probabilistic Taylor-like Conditions.} Corresponding to a norm $\|\cdot\|$, let $B(x,\delta)$ denote a ball of radius~$\delta$ centered at $x\in\mathbb{R}^n$.  If the gradient and Hessian estimates, $g(x)$ and $H(x)$, respectively, satisfy
 \begin{equation}\label{eq:fully-linear}
    \|g(x) - \nabla \phi(x)\| \leq \kappa_{\rm eg} \delta\ \ \text{and}\ \  \|H(x) \| \leq \kappa_{\rm bhm}
  \end{equation}
  for some nonnegative scalars $\kappa_{\rm eg}$ and $\kappa_{\rm bhm}$\footnote{``bhm" stands for ``\underline{b}ound on the \underline{H}essian of the \underline{m}odel".}, then the model $m : \mathbb{R}^{n} \to \mathbb{R}{}$ defined by
  \begin{equation*}
    m(y) = \phi(x) + \langle g(x), y - x\rangle + \frac12\langle H(x)(y - x), y - x\rangle
  \end{equation*}
  gives an approximation of $f(y)$ within $B(x,\delta)$ that is comparable to that given by an accurate first-order Taylor series approximation (with error dependent on $\delta$).  Such models, introduced in \cite{DFObook}, are known as {\em fully linear} models on $B(x,\delta)$. Similarly, if
  \begin{equation}\label{eq:fully-quad}
   \|g(x) - \nabla \phi(x)\| \leq \kappa_{\rm eg} \delta^2\ \ \text{and}\ \ \|H(x) - \nabla^2 \phi(x)\| \leq \kappa_{\rm eh}\delta,
  \end{equation}
  for some nonnegative scalars $(\kappa_{\rm eg},\kappa_{\rm eh})$, then $m(y)$ gives an approximation of $f(y)$ that is comparable to that given by an accurate second-order Taylor series approximation.  Such models are known as {\em fully quadratic} models on $B(x,\delta)$. 
  
  The concept of {\em $p$-probabilistically fully-linear} and {\em fully-quadratic models} was introduced in \cite{bandeira2014convergence} and requires conditions \eqref{eq:fully-linear} or \eqref{eq:fully-quad} to hold with  probability $p$ at each iteration of a trust region algorithm, conditioned on the past. 
  Thus, a $p$-probabilistically fully-linear model can be constructed given access to Oracle~\ref{oracle.first} with $\epsilon_g=0$ and input $\delta_1=\delta$ and $p_1=p$, and a $p$-probabilistically fully-quadratic model can be built given access to Oracle~\ref{oracle.first} with   $\epsilon_g=0$ and input $\delta^{(1)}=\delta^2$ and Oracle~\ref{oracle.second} with $\epsilon_H=0$ and input $\delta^{(2)}=\delta$ and $p_1p_2\geq p$. 

  In \cite{bandeira2014convergence,gratton2018complexity}, the convergence and complexity of a trust region method were analyzed under the  assumption that  Oracles \ref{oracle.first} and \ref{oracle.second} are implementable over $\Scal_j=(0,\infty) \times [0, \bar p_j]$, where $\bar p_j$ is sufficiently large, $j=1,2$ (Figure~\ref{fig:oracle0popen}). 
  On the other hand, the function oracle  in \cite{bandeira2014convergence,gratton2018complexity} was assumed to be exact and not stochastic, i.e., Oracle~\ref{oracle.zero}.\ref{ass:bounded noise} with $\epsilon_f=0$ (Figure~\ref{fig:oracle01}).

  In \cite{STORM,blanchet2019convergence}, the conditions on the zeroth-order oracle were significantly relaxed by assuming (in the first-order analysis) that
  \begin{equation*}
    \mathbbm{P}\left\{|f(x) - \phi(x)| \leq \kappa_{\rm ef} \delta^2\right\} \ge \bar p_0
  \end{equation*}
  for a sufficiently large $\bar p_0$, where $\kappa_{\rm ef}$ is some positive constant. 
  That is, 
  $\Scal_0=(0,\infty) \times [0, \bar p_0]$, for some sufficiently large $\bar p_0$ (see Figure~\ref{fig:oracle0popen}).  
  For the second-order analysis, the assumptions are stronger, requiring 
  \begin{equation*}
    \mathbbm{P}\left\{|f(x) - \phi(x)| \leq \kappa_{\rm ef} \delta^3 \right\} \le \bar p_0
  \end{equation*}
  to hold for some sufficiently large $\bar p_0$ and  
  \begin{equation*}
    {\mathbb{E}} \left[|f(x) - \phi(x)| \right] \leq \kappa_{F} \delta^3
  \end{equation*}
  for some $\kappa_{F}$. 
  If $\delta$ is not bounded below by a positive number, the second condition implies a zeroth-order oracle  implementable over $(0,+\infty)\times[0,1)$. 
    However, if $\delta$ is bounded below by a positive number, then these two conditions imply  weaker  conditions on ${\cal S}_0$ than  our assumptions on Oracle \ref{oracle.zero}, 
    allowing heavy tailed distributions of the error $|f(x) - \phi(x)|$. 
  Establishing algorithm complexity under the oracles in  \cite{STORM,blanchet2019convergence}  requires a different type of analysis than the one presented in this paper, which so far has not been extended to deriving a high probability complexity bound and has resulted in worse bounds on $\bar{p}_1$ and $\bar{p}_2$ as well as worse constants in the complexity bound. 

\paragraph{Gradient Norm Condition.} The fully-linear model condition is strongly tied to the TR algorithm by the use of the TR radius on the right-hand side of \eqref{eq:fully-linear} and \eqref{eq:fully-quad}. 
If, instead of the TR method, a line search method based on inexact gradient estimates $\{g(x_k)\}_{k=0,1,\dots}$ is used 
for obtaining a solution $x_*$ with $\|\nabla\phi(x_*)\|\le\epsilon$ for some $\epsilon>0$, then to establish the computational complexity, it is sufficient that the gradient estimate $g(x)$ satisfies the gradient \emph{norm condition} 
\begin{equation}\label{eq:norm-cond}
  \|g(x) - \nabla \phi(x)\| \leq \theta \|\nabla \phi(x)\|,
\end{equation} 
at all the iterates $x\in\{x_k\}_{k=0,1,\dots}$, for some $\theta \in [0,1)$ with sufficiently high probability. 
To satisfy this condition, the first order oracle needs to be implementable over $[\epsilon,+\infty)\times[0,\bar{p}_1]$ for a sufficiently large $\bar{p}_1$.
  
The norm condition \eqref{eq:norm-cond} was first introduced in~\cite{carter1991global} in the context of TR methods. 
This condition, with sufficiently small $\theta$ ensures the convergence of popular methods such as Armijo backtracking line-search.  
Verifying the gradient norm condition \eqref{eq:norm-cond} requires knowledge of $\|\nabla \phi(x_k)\|$, thus it cannot be enforced, even with some fixed probability, in the case of expected risk minimization \cite{byrd2012sample}. 
However, in some settings, where $g(x)$ is a randomized finite difference approximation of the gradient of a (possibly noisy) function \cite{nesterov2017random}, it is possible to ensure \eqref{eq:norm-cond} holds with sufficiently high probability  \cite{berahas2021theoretical}. 
  
\paragraph{Stochastic Gradient Norm Conditions.} When the norm condition cannot be ensured, one again can resort to some Taylor-like conditions. In the case of line search, however,  those have to hold in a region, whose size is a product of the step size parameter $\alpha$ and the norm of the search direction $\|g(x)\|$.   Thus, the following  conditions inspired by the concept of  fully-linear models  have been used in the literature to ensure convergence of line search methods \cite{cartis2018global,paquette2020stochastic},
\begin{equation}\label{eq:ls_grad_cond}
  |f(x) - \phi(x)| \leq \kappa_{\rm ef} \alpha^2 \|g(x)\|^2\ \text{ and } \ \|g(x) - \nabla \phi(x)\| \leq \kappa_{\rm eg} \alpha \|g(x)\|,
\end{equation}
for some nonnegative constants $(\kappa_{\rm ef},\kappa_{\rm eg})$, where $\alpha$ is a step size parameter.  
When \eqref{eq:ls_grad_cond} holds, the linear model $\phi(x) + \langle g(x), y-x\rangle$ gives an approximation of $\phi$ around $y = x - \alpha g(x)$ that is comparable to that given by the first-order Taylor expansion of $\phi$ in $B(x,\alpha\|g(x)\|)$. 
In \cite{cartis2018global,paquette2020stochastic}, complexity analyses of randomized and stochastic line search algorithms were derived under the condition that \eqref{eq:ls_grad_cond} holds with sufficiently high probability. 
A possible procedure of ensuring this conditions was outlined in \cite{cartis2018global} and in essence requires access to a stochastic first-order oracle implementable over $\Scal_1=(0,\infty) \times [0, \bar p_1]$, where $\bar p_1$ is sufficiently large. 
The zeroth-order oracle in \cite{cartis2018global}, as in \cite{bandeira2014convergence,gratton2018complexity}, is exact. 
In \cite{paquette2020stochastic}, as in \cite{blanchet2019convergence}, the zeroth-order oracle is assumed to be implementable over $(0,\infty)\times[0, \bar p_0]\subset \Scal_0$ (Figure~\ref{fig:oracle0popen}). 

Finally, an almost identical Oracle~\ref{oracle.zero} as the one in this paper is proposed in \cite{jin2021high}, while the first-order oracle is similar to our Oracle~\ref{oracle.first} but with $ \|g(x)\|$ appearing the bound on the accuracy, as in other papers on line search methods. There is no second-order analysis in \cite{jin2021high}, thus no second-order oracle is defined. 

In summary, our first- and second-order oracles are more general than those used in prior literature.
Our assumptions on the zeroth-order oracle, while not general enough to include those in \cite{STORM,blanchet2019convergence}, are relevant to practice and yet allow for strong theoretical results. In the next two settings, we discuss two common settings for the stochastic zeroth-, first-, and second-order oracles. 

\subsection{Expected risk minimization}
In this setting, $\phi(x)=\E_{d \sim \mathcal{D}}[l(x,d)]$, where $x$ are the model parameters, $d$ is a data sample following distribution $\mathcal{D}$, and $l(x,d):\mathbb{R}^n \rightarrow \mathbb{R}$ is the loss function of the $d$-th data point parametrized by $x$. 
The zeroth- and first-order oracles estimates are obtained by sample averages of the loss function and its gradient, respectively, over $\mathcal{B}$ (a mini-batch sampled from $\mathcal{D}$), i.e.,  
\begin{equation}\label{eq:erm_oracle}
	f_\mathcal{B}(x) = \frac{1}{|\mathcal{B}|}\sum_{d\in \mathcal{B}}l(x,d)\ \text{ and } \ g_\mathcal{B}(x) = \frac{1}{|\mathcal{B}|}\sum_{d\in \mathcal{B}}\nabla_x l(x,d).
\end{equation}

In \cite{jin2021high} it is shown that the conditions of Oracle~\ref{oracle.zero}.\ref{ass:subexponential noise} are satisfied for any $x$ for which $l(x,d)$ has a subexponential  distribution 
(e.g., when the support of $\mathcal{D}$ is bounded and $l$ is Lipschitz) by selecting an appropriate sample size  $|\mathcal{B}|$. 

 Let us show how  Oracle~\ref{oracle.first}  is easily implemented in this setting and explain the roles of $\epsilon_g$ and $\kappa_{\rm eg}$.
Assume that  the variance of stochastic gradient is bounded, $\E_{d\sim \mathcal{D}} \left[\|\nabla l(x,d)-\nabla\phi(x)\|^2 \right]\le \sigma^2$.
Given input $\delta^{(1)}$  and $p_1$, choose random i.i.d. mini-batch $\mathcal{B}$ whose size is at least  $\min\{N,\left((1-p_1)\delta^{(1)}\right)^{-2}\}$, where $N$ is the maximum possible mini-batch size that the oracle can generate. Then, 
we have
\[ \E_{\mathcal{B}} \left[\|g_{\mathcal{B}}(x)-\nabla\phi(x)\| \right]
\le \sqrt{\E_{\mathcal{B}} \left[\|g_{\mathcal{B}}(x)-\nabla\phi(x)\|^2 \right]}
\le \max \left\{\frac{\sigma}{\sqrt{N}}, \sigma (1-p_1)\delta^{(1)} \right\},
\] 
which by Markov inequality implies  
\begin{equation} \label{eq:ERM Markov} \begin{aligned} 
  &\mathbbm{P} \left\{ \|g_\mathcal{B}(x)-\nabla\phi(x)\| \le \frac{\sigma}{(1-p_1)\sqrt{N}} + \sigma\delta^{(1)} \right\} \\
  &\ge \mathbbm{P} \left\{ \|g_\mathcal{B}(x)-\nabla\phi(x)\| \le \max \left\{ \frac{\sigma}{(1-p_1)\sqrt{N}}, \sigma\delta^{(1)} \right\}\right\} 
  \ge p_1. 
\end{aligned} \end{equation}
Thus, we have Oracle~\ref{oracle.first} with input $\delta^{(1)}$ and $p_1$, with $\epsilon_g =\frac{\sigma}{(1-p_1)\sqrt{N}}$ and $\kappa_{\rm eg} =\sigma$. 
We note that $\epsilon_g$ and $\kappa_{\rm eg}$ need not be known for the execution of Algorithm~\ref{alg:tr}, but Algorithm~\ref{alg:tr2} requires an estimate of $\epsilon_g$.

\subsection{Gradient and Hessian approximation via zeroth-order oracle}

Let us consider the setting in which only the zeroth-order oracle is available for the objective function. 
This zeroth-order oracle can come from a variety of settings, such as simulation-based optimization, machine learning, or solutions for complex systems~\cite{DFObook}. Both cases, Oracle~\ref{oracle.zero}.\ref{ass:bounded noise} and  Oracle~\ref{oracle.zero}.\ref{ass:subexponential noise}, have numerous applications. 
Let us now discuss how Oracle~\ref{oracle.first} and Oracle~\ref{oracle.second} can be implemented using only function estimates via finite differences. 

Consider the following first-order oracle: given $x\in\mathbb{R}^n$, choose $\sigma>0$ and compute $f(y)$ for all $y$ in the set ${\cal Y}=\{x\}\cup\{x+\sigma u_i\}_{i=1}^n$ using Oracle~\ref{oracle.zero}, where $u_i$, $i=1,\dots,n$, denotes the unit vector along the $i$-th coordinate. 
Compute $g(x)$ as follows
\begin{equation}\label{eq:ffd1}
g(x) = \sum_{i=1}^n \frac{f(x+\sigma u_i)  -f(x)}{\sigma} u_i. 
\end{equation}

The following proposition holds. 
 \begin{proposition}[\cite{berahas2021theoretical} Theorem 2.1]\label{prop.finite_diff1}
 Assume that $|f(y)-\phi(y)|\leq \hat \epsilon_f$ for all $y\in\cal Y$. Then, under Assumption~\ref{assum:lip_cont} 
\begin{equation}\label{eq:ffd_bound1} 
\|g(x) - \nabla\phi(x)\| \le \frac{\sqrt{n} L_1 \sigma}{2}+ \frac{\sqrt{n} \hat \epsilon_f}{\sigma}. 
\end{equation}
\end{proposition}

Now consider the following second-order oracle: given $x\in\mathbb{R}^n$, choose $\sigma>0$ and compute $f(y)$ for all $y$ in the set ${\cal Y}=\{x\}\cup\{x+\sigma u_i\}_{i=1}^n \cup \{x+\sigma u_i+\sigma u_j\}_{i,j=1}^n$ using Oracle~\ref{oracle.zero}. Compute $H(x)$ as follows 
\begin{equation}\label{eq:ffd2}
H(x) = \sum_{i=1}^n \sum_{j=1}^n \frac{f(x+\sigma u_i+\sigma u_j) - f(x+\sigma u_i) - f(x+\sigma u_j) + f(x)}{\sigma^2} u_i u_j^\intercal,
\end{equation}
The following proposition holds. 
\begin{proposition}\label{prop.finite_diff2}
 Assume that $|f(y)-\phi(y)|\leq \hat \epsilon_f$ for all $y\in\cal Y$. Then, under Assumption~\ref{ass:lip_cont2} 
\begin{equation}\label{eq:ffd_bound2} 
\|H(x) - \nabla^2\phi(x)\| \le \frac{(\sqrt{2}+1) n L_2 \sigma}{3} + \frac{4n\hat \epsilon_f}{\sigma^2}. 
\end{equation}
\end{proposition}

First let us assume that Oracle~\ref{oracle.zero}.\ref{ass:bounded noise} is used. Then, Propositions \ref{prop.finite_diff1} and \ref{prop.finite_diff2} apply with $\hat \epsilon_f=\epsilon_f$ with probability $1$. 
Hence, by selecting $\sigma$ that minimizes the right-hand sides of \eqref{eq:ffd_bound1}, the finite difference formula \eqref{eq:ffd1} gives us Oracle~\ref{oracle.first} with $\epsilon_g=\sqrt{\frac{nL_1\epsilon_f}{2}}$, $\kappa_{\rm eg}=0$ for any $\delta^{(1)} > 0$ and for $p_1=1$. 
Similarly, by choosing $\sigma$ to minimize the right-hand side of \eqref{eq:ffd_bound2}, we obtain from formula \eqref{eq:ffd2} an Oracle~\ref{oracle.second} with $\epsilon_H=\sqrt[3]{3(\sqrt{2}-1)^2\epsilon_f L_2^2}$, $\kappa_{\rm eh}=0$ for any $\delta^{(2)} > 0$ and for $p_2=1$.
 
Next, let us consider the case of Oracle~\ref{oracle.zero}.\ref{ass:subexponential noise}. It is not guaranteed that  $|f(y)-\phi(y)|\leq \hat \epsilon_f$ for all 
$y\in {\cal Y}$. However, for any $\hat \epsilon_f >\epsilon_f$, 
we have that  for any $y$, $|f(y)-\phi(y)|\leq \hat\epsilon_f$ with probability at least $1-e^{a(\epsilon_f-\hat \epsilon_f)}$. 
Thus, with probability at least $(1-e^{a(\epsilon_f-\hat\epsilon_f)})^{n+1}$ it holds that $|f(y)-\phi(y)|\leq \hat \epsilon_f$ for all 
$y \in {\cal Y}$ defined in the first-order oracle. 
Proposition~\ref{prop.finite_diff1} implies that the first-order oracle defined above delivers Oracle~\ref{oracle.first} with $\epsilon_g=\sqrt{\frac{nL_1\hat \epsilon_f}{2}}$, $\kappa_{\rm eg}=0$ for any $\delta^{(1)}>0$ and for $p_1=(1-e^{a(\epsilon_f-\hat\epsilon_f)})^{n+1}$. 
Similarly,  $|f(y)-\phi(y)|\leq \hat\epsilon_f$ for all $y\in {\cal Y}$, defined by the second-order oracle,  
with probability at least $(1-e^{a(\epsilon_f-\hat\epsilon_f)})^{(n+1)(n+2)/2}$, and thus we have Oracle~\ref{oracle.second} 
  with $\epsilon_H=\sqrt[3]{3(\sqrt{2}-1)^2\hat \epsilon_f L_2^2}$, $\kappa_{\rm eh}=0$ for any
 $\delta^{(2)} > 0 $ and for $p_2=(1-e^{a(\epsilon_f-\hat\epsilon_f)})^{(n+1)(n+2)/2}$.

Now, let us consider a first-order oracle based on polynomial interpolation. 
Specifically, for a given $x$ choose  $\sigma>0$ and a linearly independent set of vectors $\{u_i\}_{i=1}^n$,  such that $\|u_i\|\le 1$.  
Compute $f(y)$ for all $y$ in the set ${\cal Y}=\{x\}\cup\{x+\sigma u_i\}_{i=1}^n$ using Oracle~\ref{oracle.zero}. Let  ${\cal U}\in\mathbb{R}^{n\times n}$ denote matrix whose rows are $\{u_i^\intercal\}_{i=1}^n$ and let $\tilde u_i$ denote the columns of  ${\cal U}^{-1}$. 
 Compute 
 \[
 g(x)= \sum_{i=1}^n \frac{f(x+\sigma u_i)  -f(x)}{\sigma} \tilde u_i. 
 \]

It is shown in \cite{berahas2021theoretical} that if $|f(y)-\phi(y)|\leq \hat \epsilon_f$, for all $y\in {\cal Y}$ and under  Assumption~\ref{assum:lip_cont} the  following bound holds. 
\[ 
 \|g(x) - \nabla\phi(x)\|\le \|\mathcal{U}^{-1}\| \left[ \frac{\sqrt{n}L_1\sigma}{2} + \frac{2\sqrt{n}\epsilon_f}{\sigma} \right], 
\]
Using arguments similar to those used above for finite differences, one can easily show  that the interpolation oracle provides Oracle~\ref{oracle.first} with appropriately
chosen $\epsilon_g$ and $\kappa_{\rm eg}=0$ and for any $\delta^{(1)} > 0$. The additional nuance of this case is the choice of ${\cal U}$, so that  $\|{\cal U}^{-1}\|$ is bounded from above either deterministically or probabilistically, depending on an algorithm employed, thus $p^{(1)}$ of this Oracle~\ref{oracle.first} is defined according to the choice of  ${\cal U}$, the instance of Oracle~\ref{oracle.zero} employed.

 Similarly, Oracle~\ref{oracle.second} can be implemented using techniques such as quadratic interpolation by choosing an appropriate sample set. 
How to choose such a sample set involves the convoluted concept of {\it poisedness} which is out of the scope of this paper. 
We refer interested readers to \cite{DFObook,UOBYQA,powell2001lagrange}. 
In \cite{bandeira2014convergence} a stochastic second-order oracle is generated by using quadratic interpolation a  randomly sampled set, which allows
for a more efficient second-order oracle, when $\nabla \phi(x)$ is  approximately sparse.
A further discussion of stochastic oracles in \cite{bandeira2014convergence} and other settings that fit our generic oracle definition is a worthwhile topic, but is beyond the scope of this paper. 
 
\section{Trust-region algorithms for noisy optimization}
\label{sec:algorithm}
In this section, we propose first- and second-order modified TR algorithms which utilize the stochastic oracles discussed in Section~\ref{sec:oracles} to produce models of the objective function.  We also define the requirements on these models and  derive some key properties of both algorithms under these requirements. We finish the section by describing the algorithms as stochastic processes which we then analyze in subsequent sections. 

\subsection{Algorithms}
In every iteration $k \in \{0,1,\dots\}$ of our modified first- and second-order TR algorithms, a quadratic model 
\begin{equation}\label{eq:model_def}
	m_k(x_k+s) = \phi(x_k) + \langle g_k, s \rangle + \frac{1}{2} \langle H_k s, s \rangle
\end{equation}
is constructed to approximate the objective function near the iterate $x_k$. 
The constant term $\phi(x_k)$ appears in \eqref{eq:model_def} for clarity, but is not needed in the algorithm, since only changes in the model value 
$m_k(x_k)-m_k(x_k+s)$  are computed.  
The model gradient $g_k = \nabla m_k(x_k)$ is computed as a (random) approximation of $\nabla \phi(x_k)$ by the first-order oracle (Oracle \ref{oracle.first}) with a specified accuracy and reliability for the iterate $x_k$. 
The model Hessian $H_k = \nabla^2 m_k(x_k)$ can be a quasi-Newton matrix or other (not necessarily random or accurate) approximation of $\nabla^2 \phi(x_k)$, that satisfies the following standard assumption \cite{TRbook}. 
\begin{assumption} \label{assum:bhm}
For all $k=0,1,2,\dots$, $\| H_k\| \leq \kappa_{\rm bhm}$, where $\kappa_{\rm bhm}$ is a positive constant.
\end{assumption}

Our modified first-order TR method is stated in Algorithm~\ref{alg:tr}. 
In the execution of Algorithm~\ref{alg:tr} (Line~\ref{step.s}) it is required that the TR subproblem, defined as 
\begin{equation}\label{eq:tr_sub}
	\min_{s \in B(x_k,\delta_k)} \;\; m_k(x_k + s),
\end{equation}
where $B(x_k,\delta_k)$ is a Euclidean ball with center $x_k$ and radius $\delta_k>0$, is consistently solved accurately enough so that the step $s_k$ satisfies 
\begin{equation}\label{eq:Cauchy decrease}
    m_k(x_k) - m_k(x_k + s_k) \ge \frac{\kappa_{\rm fcd}}{2} \|g_k\| \min\left\{ \frac{\|g_k\|}{\|H_k\|}, \delta_k\right\} ,
\end{equation}
for some (chosen) constant $\kappa_{\rm fcd} \in (0,1]$\footnote{``fcd" stands for ``\underline{f}raction of \underline{C}auchy \underline{d}ecrease''. }. 
Condition \eqref{eq:Cauchy decrease} is commonly used in the literature and is satisfied by the Cauchy point with $\kappa_{\rm fcd} = 1$. 
See \cite[Section 6.3.2]{TRbook} for more details. 

\begin{algorithm}[ht]
\caption{~\textbf{Modified First-Order Trust-Region Algorithm}}
\label{alg:tr}
{\bf Inputs:} starting point $x_0$; initial TR radius $\delta_0 > 0$;  
hyperparameters for controlling the TR radius $\eta_1 > 0,\eta_2 > 0$, $\gamma \in (0,1)$; tolerance parameter $r>0$; and probability parameter $p_1$.\\
\For{$k=0,1,2,\dots$} {
	\nl Compute vector $g_k$ using stochastic Oracle \ref{oracle.first}  with input $(\delta_k, p_1, x_k)$ and matrix $H_k$ that satisfies Assumption~\ref{assum:bhm}. \\
	\nl Compute $s_k$ by approximately minimizing $m_k$ in $B(x_k,\delta_k)$ so that it satisfies \eqref{eq:Cauchy decrease}. \label{step.s}\\
	\nl Compute $f_k$ using Oracle~\ref{oracle.zero}, with input $x_k$, and $f_k^{+}$  using Oracle~\ref{oracle.zero}, with input $x_k+s_k$, and then compute
	\[ \rho_k = \frac{f_k - f_k^+ + r}{m_k(x_k) - m_k(x_k + s_k)}.
	\] \\
	\nl \If{$\rho_k \geq \eta_1$}{Set $x_{k+1} = x_k + s_k$ and 
	\[  \delta_{k+1} = \begin{cases}
	    \gamma^{-1} \delta_k \quad & \text{if } \|g_k\| \geq \eta_2 \delta_k \\
		\gamma \delta_k, \quad & \text{if } \|g_k\| < \eta_2 \delta_k.	
	\end{cases} \] }
	\nl \Else{ Set $x_{k+1} = x_k $ and $\delta_{k+1} =  \gamma \delta_k$.}
}
\end{algorithm}

Algorithm~\ref{alg:tr2} is our modified second-order TR algorithm. Similar to Algorithm~\ref{alg:tr}, in the execution of Algorithm~\ref{alg:tr2} (Line~\ref{step.s2}) the TR subproblem \eqref{eq:tr_sub} needs to be solved sufficiently accurately, and the step $s_k$ computed needs to satisfy the following stronger condition
\begin{align}\label{eq:Eigen decrease}
    m_k(x_k) - m_k(x_k + s_k) \geq \frac{\kappa_{\rm fod}}{2} \max \left\{ \|g_k\| \min\left\{\frac{\|g_k\|}{\|H_k\|}, \delta_k\right\}, -\lambda_{\rm min} (H_k) \delta_k^2 \right\}, 
\end{align}
for some (chosen) constant $\kappa_{\rm fod} \in (0,1]$\footnote{``fod" stands for ``\underline{f}raction \underline{o}f \underline{d}ecrease''. }. 
Contrary to Algorithm~\ref{alg:tr}, the Hessian approximations $H_k$ in Algorithm~\ref{alg:tr2} are required to be sufficiently accurate and not just bounded in norm. 
Furthermore, instead of comparing $\|g_k\|$ to $\eta_2 \delta_k$ to determine the adjustment of the TR radius, the following value is used: 
\begin{equation}\label{eq:beta_k_m}
    \beta^m_k \stackrel{\rm def}{=} \max \left\{ \|g_k\|, -\lambda_{\rm min} (H_k) \right\}. 
\end{equation}

\begin{algorithm}[ht]
\caption{~\textbf{Modified Second-Order Trust-Region Algorithm}}
\label{alg:tr2}
{\bf Inputs:} starting point $x_0$; initial TR radius $\delta_0 > 0$;  
hyperparameters for controlling the TR radius $\eta_1 > 0,\eta_2 > 0$, $\gamma \in (0,1)$; tolerance parameter $r>0$; and probability parameters $p_1$ and $p_2$.\\
\For{$k=0,1,2,\dots$} {
	\nl  Compute vector $g_k$ using stochastic Oracle \ref{oracle.first}   with input $(\delta_k^2, p_1, x_k)$  and matrix $H_k$ using stochastic Oracle \ref{oracle.second}  with input $(\delta_k, p_2, x_k)$. \\
	\nl Compute $s_k$ by approximately minimizing $m_k$ in $B(x_k,\delta_k)$ so that it satisfies \eqref{eq:Eigen decrease}. \label{step.s2}\\
	\nl Compute $f_k$ using Oracle~\ref{oracle.zero}, with input $x_k$, and $f_k^{+}$  using Oracle~\ref{oracle.zero}, with input $x_k+s_k$, and then compute
	\[ \rho_k = \frac{f_k - f_k^+ + r}{m_k(x_k) - m_k(x_k + s_k)}.\] \\
	\nl \If{$\rho_k \geq \eta_1$}{Set $x_{k+1} = x_k + s_k$ and using  $\beta^m_k$  defined in \eqref{eq:beta_k_m}
	\[  \delta_{k+1} = \begin{cases} 
	    \gamma^{-1} \delta_k \quad & \text{if } \beta^m_k \ge \eta_2 \delta_k \\
		\gamma \delta_k, \quad & \text{if } \beta^m_k < \eta_2 \delta_k.
	\end{cases} \] }
	\nl \Else{ Set $x_{k+1} = x_k$ and $\delta_{k+1} = \gamma \delta_k$.}
}
\end{algorithm}

We also note that if by any chance $m_k(x_k) = m_k(x_k + s_k)$ in any iteration of either of these two algorithms, the step $s_k$ is automatically rejected, i.e., $x_{k+1}=x_k$ and $\delta_{k+1}=\gamma\delta_k$.

\begin{remark} Algorithms~\ref{alg:tr} and \ref{alg:tr2} are very similar to classical TR algorithms \cite{TRbook}. The major difference pertains to the fact that the step acceptance criterion is relaxed. 
The relaxation is similar to that in \cite{berahas2021global,berahas2019derivative,jin2021high} for line/step search methods. A user-defined tolerance parameter is added to the numerator in order to account for the noise in the zeroth-order oracle. 
The value of $r$ in Algorithm~\ref{alg:tr} is set to $2\epsilon_f$ if $\epsilon_f$ is known (for example when the zeroth-order oracle satisfies Oracle~\ref{oracle.zero}.\ref{ass:bounded noise} with a known noise bound); otherwise, we simply let $r$ be any value large enough to be no less than $2\epsilon_f$. 
Similarly, the value of $r$ in Algorithm~\ref{alg:tr2} is set to $2\epsilon_f + \epsilon_g^{3/2}$ if both $\epsilon_f$ and $\epsilon_g$ are known, and set to any large enough value otherwise. 
The effect of choosing particular values of $r$ will be discussed later in the paper. 
\end{remark}

\subsection{Approximation model accuracy}
\label{sec:fully}

Let us introduce a  definition of a {\em sufficiently accurate}  model. 
\begin{definition}
An approximation model of the form \eqref{eq:model_def} is said to be \textbf{first-order sufficiently accurate} if there are nonnegative constants $\kappa_{\rm eg}$ and $\epsilon_g$ such that 
\begin{equation} \label{eq:sufficiently accurate g}
    \| \nabla \phi (x_k) - g_k \| \le \kappa_{\rm eg} \delta_k + \epsilon_g, 
\end{equation}
and, is said to be \textbf{second-order sufficiently accurate} 
if there are nonnegative constants $\kappa_{\rm eg}, \kappa_{\rm eh}, \epsilon_g$ and $\epsilon_H$ such that

\begin{subequations} \label{eq:sufficiently accurate H g} \begin{align}
    \| \nabla^2\phi(x_k) - H_k \| &\le \kappa_{\rm eh} \delta_k + \epsilon_H \label{eq:sufficiently accurate H g: H} \\
    \|\nabla\phi(x_k) - g_k \| &\le \kappa_{\rm eg} \delta_k^2 + \epsilon_g. \label{eq:sufficiently accurate H g: g}
\end{align} \end{subequations}
\end{definition}

Note that \eqref{eq:sufficiently accurate g} is satisfied with probability $p_1$ by Oracle~\ref{oracle.first} with input $(\delta_k,p_1,x_k)$, \eqref{eq:sufficiently accurate H g: g} is satisfied with probability $p_1$ by Oracle~\ref{oracle.first} with input $(\delta^2_k,p_1,x_k)$ and \eqref{eq:sufficiently accurate H g: H} is satisfied with probability $p_2$ by Oracle~\ref{oracle.second} with input $(\delta_k,p_2,x_k)$.

Under \eqref{eq:sufficiently accurate g} (resp.,  \eqref{eq:sufficiently accurate H g}), error bounds on the model accuracy can be derived. 
The first lemma below provides a bound on the approximation error of the model in $B(x_k,\delta_k)$ under \eqref{eq:sufficiently accurate g}.
\begin{lemma} \label{lem:phi-m}
Under Assumptions~\ref{assum:lip_cont} and \ref{assum:bhm}, if \eqref{eq:sufficiently accurate g} holds, it follows
\begin{equation} \label{eq:phi-m}
    |\phi(x_k+s) - m_k(x_k+s)| \le (L_1 + \kappa_{\rm bhm} + 2\kappa_{\rm eg }) \delta_k^2 /2 + \epsilon_g \delta_k
\end{equation}
for all $x_k + s \in B(x_k,\delta_k)$. 
\end{lemma}

\begin{proof}
By the triangle inequality, Assumptions~\ref{assum:lip_cont} and \ref{assum:bhm} and \eqref{eq:sufficiently accurate g}, it follows that,
\[ \begin{aligned}
    &|\phi(x_k+s) - m_k(x_k+s)| \\
    &= \left| \phi(x_k + s)  - \phi(x_k) - \langle g_k, s \rangle - \langle H_k s, s \rangle /2 \right| \\
    &\le \left| \phi(x_k + s)  - \phi(x_k) - \langle \nabla \phi(x_k), s \rangle \right| + \left| \langle \nabla \phi(x_k), s \rangle - \langle g_k, s \rangle \right| + \left|\langle H_k s, s \rangle /2 \right| \\
    &\le L_1 \|s\|^2 /2 + \|\nabla\phi(x_k) - g_k\| \|s\| + \kappa_{\rm bhm} \|s\|^2 /2 \\
    &\le L_1 \delta_k^2 /2 + (\kappa_{\rm eg} \delta_k  + \epsilon_g) \delta_k + \kappa_{\rm bhm} \delta_k^2 /2 \\
    &= (L_1 + \kappa_{\rm bhm} + 2\kappa_{\rm eg }) \delta_k^2 /2 + \epsilon_g \delta_k. 
\end{aligned} \]
\end{proof}

The next  result provides a bound on the approximation error of the model in $B(x_k,\delta_k)$ under \eqref{eq:sufficiently accurate H g}.
\begin{lemma} \label{lem:phi-m2}
Under Assumption~\ref{ass:lip_cont2}, if \eqref{eq:sufficiently accurate H g} holds, it follows
\begin{equation} \label{eq:phi-m2}
    |\phi(x_k+s) - m_k(x_k+s)| \le (L_2/6 + \kappa_{\rm eg}+1 + \kappa_{\rm eh}/2) \delta_k^2\|s\| + \epsilon_H \delta_k\|s\|/ 2 + \epsilon_g^{3/2} 
\end{equation}
for all $x_k+s \in B(x_k,\delta_k)$. 
\end{lemma}
\begin{proof}
First let us show that if \eqref{eq:sufficiently accurate H g: g} holds, then 
\begin{equation}
    \|\nabla\phi(x_k) - g_k \| \le (\kappa_{\rm eg} +1)\delta_k^2 +\min \left\{\epsilon_g^{3/2}/\delta_k,\epsilon_g\right\}. \label{eq:sufficiently accurate H g: g1}
\end{equation}
Clearly if $\epsilon_g^{3/2} / \delta_k \geq \epsilon_g$, then \eqref{eq:sufficiently accurate H g: g} trivially implies \eqref{eq:sufficiently accurate H g: g1}.
On the other hand, if $\epsilon_g^{3/2} / \delta_k \leq \epsilon_g$, then $\delta_k\geq \epsilon_g^{1/2}$ and thus $\kappa_{eg}\delta_k^2+ \epsilon_g\leq 
(\kappa_{eg}+1)\delta_k^2$, hence again  \eqref{eq:sufficiently accurate H g: g} implies \eqref{eq:sufficiently accurate H g: g1}.

Thus, by the triangle inequality, Assumption~\ref{ass:lip_cont2} and \eqref{eq:sufficiently accurate H g}, it follows that,
\[ \begin{aligned}
    &|\phi(x_k+s) - m_k(x_k+s)| \\
    &= \left| \phi(x_k + s)  - \phi(x_k) - \langle g_k, s \rangle - 0.5 \langle H_k s, s \rangle \right| \\
    &\le \left| \phi(x_k + s)  - \phi(x_k) - \langle \nabla \phi(x_k), s \rangle  - \langle \nabla^2 \phi(x_k) s, s \rangle / 2\right| \\
    &\qquad + \left| \langle \nabla \phi(x_k), s \rangle - \langle g_k, s \rangle \right| + \left| \langle \nabla^2 \phi(x_k) s, s \rangle - \langle H_k s, s \rangle \right| / 2 \\
    &\le L_2 \|s\|^3 / 6  + \|\nabla\phi(x_k) - g_k\| \|s\| + \| \nabla^2 \phi(x_k) - H_k \| \|s\|^2 / 2 \\
    &\le L_2\|s\|^3/6 + ((\kappa_{\rm eg}+1)\delta_k^2 +  \epsilon_g^{3/2}/\delta_k) \|s\| + (\kappa_{\rm eh}\delta_k + \epsilon_H) \|s\|^2/ 2 \\
    &\le (L_2/6 + \kappa_{\rm eg}+1 + \kappa_{\rm eh}/2) \delta_k^2\|s\| + \epsilon_H \delta_k\|s\|/ 2 + \epsilon_g^{3/2}. 
\end{aligned} \]
\end{proof}

\subsection{Algorithms viewed as stochastic processes}
\label{sec:stoch_proc}

For the purposes of analyzing the convergence of Algorithms~\ref{alg:tr} and~\ref{alg:tr2}, we view the algorithms as stochastic processes. Here we introduce and explain some useful notation. 

Let $\{X_k\}, \{X_k^+\}$ denote the sequences of random vectors in $\mathbb{R}^n$ whose realizations are $\{x_k\}$ and $\{x_k+s_k\}$, respectively. 
Let $\{\Delta_k\}$ denote the sequence of random positive numbers whose  realizations are $\{\delta_k\}$, and $\{M_k\}$ denote the sequence of random models whose realizations are $\{m_k\}$. 
For the error in the zeroth-order oracle, we denote by $\{\mathcal{E}_k\}$ and $\{\mathcal{E}_k^+\}$ the absolute errors whose realizations are $\{|f_k-\phi(x_k)|\}$ and $\{|f_k^+ - \phi(x_k+s_k)|\}$, respectively. 
Additionally, with a slight abuse of notation, let $\{\rho_k\}$ be the random variables that share the same symbols as their realizations.  

With these random variables, Algorithms~\ref{alg:tr} or \ref{alg:tr2} first generate  in each iteration (random) gradient and Hessian approximations, $g(X_k,\xi^{(1)})$ and $H(X_k,\xi^{(2)})$, using Oracles~\ref{oracle.first} and \ref{oracle.second}, respectively. 
Subsequently, a random model $M_k$ is constructed deterministically around the current iterate $X_k$ using these approximations and then minimized within the trust-region to generate $X_k^+$ (deterministically, given $X_k$ and $M_k$). 
Then, random function value estimates $f(X_k, \xi^{(0)}_k)$ and $f(X_k^+, \xi^{(0+)}_{k})$ are generated using Oracle~\ref{oracle.zero}, which define the random errors $\mathcal{E}_k = |e(X_k, \xi^{(0)}_k)|$ and $\mathcal{E}_k^+ = |e(X_k^+, \xi^{(0+)}_{k})|$. 
And finally, $\rho_k$, $\Delta_{k+1}$ and $X_{k+1}$ are computed in a deterministic manner (given $X_k$, $M_k$ and the function values). 
Thus, given $X_k$ and $\Delta_k$, the randomness at the $k$-th iteration is generated by the variables $\xi^{(0)}_{k}$,  $\xi^{(0+)}_{k}$, $\xi^{(1)}_k$ and $\xi^{(2)}_k$. 

Let $\mathcal{F}_{k-1}$ denote the $\sigma$-algebra
\begin{equation}\label{eq:sig_alg1}
\mathcal{F}_{k-1}=\sigma\left( \left(\xi^{(0)}_0, {\xi^{(0+)}_{0}},\xi^{(1)}_0, \xi^{(2)}_0\right), \dots, \left(\xi^{(0)}_{k-1}, {\xi^{(0+)}_{k-1}},\xi^{(1)}_{k-1}, \xi^{(2)}_{k-1}\right)\right). 
\end{equation}

Similarly, let 
\begin{equation}\label{eq:sig_alg2}
\mathcal{F}^\prime_{k-1}=\sigma\left( \left(\xi^{(0)}_0, {\xi^{(0+)}_{0}},\xi^{(1)}_0, \xi^{(2)}_0\right), \dots, \left(\xi^{(0)}_{k-1}, {\xi^{(0+)}_{k-1}},\xi^{(1)}_{k-1}, \xi^{(2)}_{k-1}\right), \left(\xi^{(1)}_k, \xi^{(2)}_k\right) \right) 
\end{equation}

We note that the random variables $\{X_k\}$ and $\Delta_k$ are defined by $\mathcal{F}_{k-1}$, the random variables $\{X_k^+\}$ and $M_k$ are defined by  $\mathcal{F}^\prime_{k-1}$, and the random variables $\{\mathcal{E}_k\}$, $\{\mathcal{E}_k^+\}$ and $\rho_k$ are defined by $\mathcal{F}_{k}$.

\section{First-order stochastic convergence analysis} 
\label{sec:analysis1}

In this section, we analyze the first-order convergence of Algorithm~\ref{alg:tr}.  
The goal is to derive a probabilistic result of the form 
\[ \mathbbm{P} \left\{ \min_{0\le k \le T-1} \|\nabla\phi(X_k)\| < \epsilon \right\} \ge \text{a function of $T$ that converges to 1 as $T$ increase}
\]
for any sufficiently large $\epsilon$. This result cannot hold for arbitrarily small values of $\epsilon>0$ unless $\epsilon_f=\epsilon_g=0$.
The specific lower bounds on $\epsilon$ in terms of $\epsilon_f$ and $\epsilon_g$ will be presented in Theorems~\ref{thm:convergence1_bounded} and \ref{thm:convergence1_subexponential}. 

We begin by stating and proving three key lemmas about the behavior of Algorithm~\ref{alg:tr} when \eqref{eq:sufficiently accurate g} and Assumptions~\ref{assum:lip_cont} and \ref{assum:bhm} hold. 
Let $e_k = f_k - \phi(x_k)$ and $e_k^+ = f_k^+ - \phi(x_k+s_k)$.
The first lemma provides a sufficient condition for accepting a step ($x_{k+1}=x_k+s_k$).

\begin{lemma}[{\bf Sufficient condition for accepting step}] \label{lem:delta accept} 
Under Assumptions~\ref{assum:lip_cont} and \ref{assum:bhm}, if \eqref{eq:sufficiently accurate g} holds, $r \ge e_k^+ - e_k$, and 
\begin{equation} \label{eq:delta accept}
	\delta_k \le \frac{(1-\eta_1)\kappa_{\rm fcd}}{L_1 + \kappa_{\rm bhm} + 2\kappa_{\rm eg}} \|g_k\| - \frac{2}{L_1 + \kappa_{\rm bhm} + 2\kappa_{\rm eg}} \epsilon_g,
\end{equation}
then, $\rho_k \geq \eta_1$ in Algorithm~\ref{alg:tr}.
\end{lemma}
\begin{proof}
Since $\delta_k \le (1-\eta_1)\kappa_{\rm fcd} \|g_k\| / (L_1 + \kappa_{\rm bhm} + 2\kappa_{\rm eg})  \le \|g_k\| / \kappa_{bhm} \le \|g_k\| / \|H_k\|$, \eqref{eq:Cauchy decrease} suggests $m_k(x_k) - m_k(x_k + s_k) \geq \kappa_{\rm fcd} \|g_k\| \delta_k / 2$. 
Combining this inequality with $e_k - e_k^+ + r \ge 0$ and Lemma~\ref{lem:phi-m}, we have 
\[ \begin{aligned}
\rho_k &=\frac{\phi(x_k) + e_k - \phi(x_k + s_k) - e_k^+ + r}{m_k(x_k) - m_k(x_k + s_k)} \\
&\ge \frac{\phi(x_k) - \phi(x_k + s_k)}{m_k(x_k) - m_k(x_k + s_k)} \\
&\leftstackrel{\eqref{eq:phi-m}}{\ge} \frac{\phi(x_k) - m_k(x_k + s_k) - (L_1 + \kappa_{\rm bhm} + 2\kappa_{\rm eg}) \delta_k^2 /2 - \epsilon_g \delta_k}{m_k(x_k) - m_k(x_k + s_k)} \\
&\leftstackrel{\eqref{eq:model_def}}{=} 1 - \frac{(L_1 + \kappa_{\rm bhm} + 2\kappa_{\rm eg}) \delta_k^2 /2 + \epsilon_g \delta_k}{m_k(x_k) - m_k(x_k + s_k)} \\
&\leftstackrel{\eqref{eq:Cauchy decrease}}{\ge} 1 - \frac{(L_1 + \kappa_{\rm bhm} + 2\kappa_{\rm eg}) \delta_k^2 + 2\epsilon_g \delta_k}{\kappa_{\rm fcd} \|g_k\| \delta_k} \\
&\leftstackrel{\eqref{eq:delta accept}}{\ge} 1 - (1 - \eta_1) = \eta_1. 
\end{aligned} \]
\end{proof}

The next lemma provides a sufficient condition for a successful step ($x_{k+1}=x_k+s_k$ and $\delta_{k+1}=\gamma^{-1} \delta_k$). 

\begin{lemma}[\bf Sufficient condition for successful step] \label{lem:delta success}
Under Assumptions~\ref{assum:lip_cont} and \ref{assum:bhm}, if \eqref{eq:sufficiently accurate g} holds, $r \ge e_k^+ - e_k$, and 
\begin{equation} \label{eq:delta success}
    \delta_k \le C_1 \|\nabla \phi(x_k)\| - C_2 \epsilon_g, 
\end{equation} 
where 
\begin{equation} \begin{aligned}\label{eq:c1c2}
    C_1 &\stackrel{\rm def}{=} \min\left\{ 
        \frac{(1-\eta_1)\kappa_{\rm fcd}}{L_1+\kappa_{\rm bhm}+2\kappa_{\rm eg} + (1-\eta_1)\kappa_{\rm fcd}\kappa_{\rm eg}}, 
        \frac{1}{\kappa_{\rm eg} + \eta_2}
    \right\} \\
    C_2 &\stackrel{\rm def}{=} \max\left\{ 
        \frac{(1-\eta_1)\kappa_{\rm fcd} + 2}{L_1+\kappa_{\rm bhm}+2\kappa_{\rm eg} + (1-\eta_1)\kappa_{\rm fcd}\kappa_{\rm eg}},
        \frac{1}{\kappa_{\rm eg} + \eta_2}
    \right\}
\end{aligned} 
\end{equation} 
then, $\rho_k \geq \eta_1$ and $\|g_k\| \ge \eta_2 \delta_k$ in Algorithm~\ref{alg:tr}. 
\end{lemma}
\begin{proof}
By \eqref{eq:sufficiently accurate g} we have
\[\|g_k\| \ge \|\nabla \phi(x_k)\| - \|\nabla \phi(x_k) - g_k\| 
\ge \|\nabla \phi(x_k)\| - \kappa_{\rm eg} \delta_k - \epsilon_g. 
\]
Then  
\[ \begin{aligned}
&\frac{(1-\eta_1)\kappa_{\rm fcd}}{L_1 + \kappa_{\rm bhm} + 2\kappa_{\rm eg}} \|g_k\| - \frac{2}{L_1 + \kappa_{\rm bhm} + 2\kappa_{\rm eg}} \epsilon_g \\
&\ge \frac{(1-\eta_1)\kappa_{\rm fcd}}{L_1 + \kappa_{\rm bhm} + 2\kappa_{\rm eg}} (\|\nabla \phi(x_k)\| - \kappa_{\rm eg} \delta_k - \epsilon_g) - \frac{2}{L_1 + \kappa_{\rm bhm} + 2\kappa_{\rm eg}} \epsilon_g 
\stackrel{\eqref{eq:delta success}}{\ge} \delta_k. 
\end{aligned} \]
The last inequality holds due to \eqref{eq:delta success} with $C_1$ and $C_2$ set to the first term in their corresponding min/maximization operation. 
Then by Lemma~\ref{lem:delta accept}, we have $\rho_k \ge \eta_1$. 
We also have 
\[  \|g_k\| 
    \ge \|\nabla \phi(x_k)\| - \kappa_{\rm eg} \delta_k - \epsilon_g 
    \stackrel{\eqref{eq:delta success}}{\ge} \eta_2 \delta_k. 
\]
\end{proof}

The last lemma provides a lower bound on the progress made in each iteration.
\begin{lemma}[\bf Progress made in each iteration] \label{lem:progress}
In Algorithm~\ref{alg:tr}, if $\rho_k \geq \eta_1$ and $\|g_k\| \ge \eta_2 \delta_k$, then 
\begin{equation*}
    \phi(x_k) - \phi(x_{k+1}) \ge h(\delta_k) - e_k + e_k^+ - r,
\end{equation*}
where
\begin{equation}\label{eq.h}
	h(\delta) = C_3 \delta^2 \;\; \text{ and } \;\;
C_3 \stackrel{\rm def}{=} \frac{1}{2}\eta_1\eta_2 \kappa_{\rm fcd}\min \left\{ \frac{\eta_2}{\kappa_{\rm bhm}},1 \right\}. 
\end{equation}
If $\rho_k \geq \eta_1$ but $\|g_k\| < \eta_2 \delta_k$, then 
\begin{equation}
    \phi(x_k) - \phi(x_{k+1}) \ge - e_k + e_k^+ - r. 
\end{equation}
If $\rho_k < \eta_1$, then $\phi(x_{k+1}) = \phi(x_k)$. 
\end{lemma}

\begin{proof}
Let $\rho_k \ge \eta_1$. We have 
\[ 
\eta_1 \le \rho_k = \frac{\phi(x_k) + e_k - \phi(x_k + s_k) - e_k^+ + r}{m_k(x_k) - m_k(x_k + s_k)}, 
\]
which can be rearranged to $\phi(x_k) - \phi(x_{k+1}) \ge \eta_1[m_k(x_k) - m_k(x_k + s_k)] - e_k + e_k^+ - r$. 
If $\|g_k\| \ge \eta_2 \delta_k$, the first term of this expression satisfies
\[ 
\eta_1[m_k(x_k) - m_k(x_k + s_k)] 
\stackrel{\eqref{eq:Cauchy decrease}}{\ge} \frac{\eta_1\kappa_{\rm fcd}}{2} \|g_k\| \min \left\{ \frac{\|g_k\|}{\| H_k\|},\delta_k\right\} 
\ge \frac{\eta_1\kappa_{\rm fcd}}{2} \eta_2 \delta_k \min \left\{ \frac{\eta_2 \delta_k}{\kappa_{\rm bhm}},\delta_k\right\}
= h(\delta_k);
\]
otherwise 
\[ \eta_1[m_k(x_k) - m_k(x_k + s_k)] 
\stackrel{\eqref{eq:Cauchy decrease}}{\ge} \frac{\eta_1\kappa_{\rm fcd}}{2} \|g_k\| \min \left\{ \frac{\|g_k\|}{\| H_k\|},\delta_k\right\} 
\ge 0. 
\]
If $\rho_k < \eta_1$, we have $x_{k+1} = x_k$, so $\phi(x_{k+1}) = \phi(x_k)$. 
\end{proof}

Next, our analysis relies on categorizing iterations $k=0,1,\dots, T-1$ into different types, where $T$ is any positive integer. 
These types are defined using the following random indicator variables: 
\begin{equation*}
\begin{aligned}
&I_k = \mathbbm{1}\{\|\nabla\phi(X_k) - \nabla M_k(X_k)\| \le \kappa_{\rm eg} \Delta_k + \epsilon_g \} &\text{ (whether the model is first-order sufficiently accurate)} &\\
&J_k = \mathbbm{1}\{r \ge \mathcal{E}_k^+ + \mathcal{E}_k\} &\text{ (whether function evaluation errors are compensated by $r$)} &\\
&\Theta_k = \mathbbm{1}\{\rho_k \ge \eta_1 \text{ and } \|\nabla M_k(X_k)\| \ge \eta_2 \Delta_k \} &\text{ (whether the step is successful)} &\\
&\Theta_k' = \mathbbm{1}\{ \rho_k \ge \eta_1 \} &\text{ (whether the step is accepted)} &\\
&\Lambda_k = \mathbbm{1}\{ \Delta_k > \bar\Delta\}, \quad
\Lambda_k' = \mathbbm{1}\{ \Delta_k \ge \bar\Delta'\}, 
\end{aligned}
\end{equation*}
where $\bar\Delta$ and $\bar\Delta'$ are defined as 
\begin{equation} \label{eq:deltabar} \begin{aligned}
    \bar\Delta &= C_1 \min_{0\le k \le T-1} \|\nabla\phi(X_k)\| - C_2 \epsilon_g, \\
    \bar\Delta' &= \min_l\{\gamma^l \delta_0:~ \gamma^l \delta_0 > \gamma \bar\Delta \text{ and } l \in \mathbb{Z}\},
\end{aligned} \end{equation}
and the positive constants $C_1$ and $C_2$ are defined in \eqref{eq:c1c2}.

Notice that under Oracle~\ref{oracle.zero}.\ref{ass:bounded noise}, the condition $r \ge 2\epsilon_f \ge \mathcal{E}_k^+ + \mathcal{E}_k$ always holds, thus $J_k=1$. 
The random variable $\bar\Delta$ is crucial to our analysis not only because it involves the value $\min_{0\le k \le T-1} \|\nabla\phi(X_k)\|$, but also because of the following corollary to Lemma~\ref{lem:delta success}, which shows that if the errors from Oracles~\ref{oracle.zero} and \ref{oracle.first} are small, and the TR radius is also small, then the iteration is successful. 
\begin{corollary}\label{cor.1st} 
    If $I_k J_k =1$ and $\Lambda_k=0$, then $\Theta_k = 1$. 
\end{corollary}

We also have the following lemma as a direct consequence of the definitions of Oracle~\ref{oracle.first} and $I_k$.
 
\begin{lemma} \label{lem:true} 
The random sequence $\{ M_k\}$ satisfies the submartingale-type condition
\begin{align}	\label{prob-delta-def}
	\mathbbm{P}\{I_k = 1 | \mathcal{F}_{k-1}\} \geq p_1 \text{ for all } k\in\{0,1,\dots\},
\end{align}
where  $\mathcal{F}_{k-1}$ is defined in \eqref{eq:sig_alg1}.
\end{lemma}

\begin{remark}
    To be specific, the random process $\left\{ \sum_{k=0}^{t-1} I_k - p_1 t \right\}_{t=0,1,\dots}$ is a submartingale. 
\end{remark}

Before we state and prove the main results of this section (Theorem~\ref{thm:convergence1_bounded} for Oracle~\ref{oracle.zero}.\ref{ass:bounded noise} and Theorem~\ref{thm:convergence1_subexponential} for Oracle~\ref{oracle.zero}.\ref{ass:subexponential noise}), we state and prove three technical lemmas that are used in the analysis of Algorithm~\ref{alg:tr} under both Oracle~\ref{oracle.zero}.\ref{ass:bounded noise} and Oracle~\ref{oracle.zero}.\ref{ass:subexponential noise}.
The first lemma provides an upper bound on the number of successful iterations with large (defined by $\Lambda_k'$) TR radius. 

\begin{lemma} \label{lem:total_progress}
For any positive integer $T$, we have 
\begin{equation} \label{eq:total_progress}
    h(\gamma \bar\Delta) \sum_{k=0}^{T-1} \Theta_k \Lambda_k' < \phi(x_0) - \hat\phi + \sum_{k=0}^{T-1} \Theta_k' \left( \mathcal{E}_k + \mathcal{E}_k^+ + r \right). 
\end{equation}
\end{lemma}

\begin{proof}
Notice $h(\cdot)$ (defined in \eqref{eq.h}) is a monotonically non-decreasing function, so $h(\Delta_k) \ge h(\bar\Delta')$ if $\Lambda_k' = 1$. 
By Lemma~\ref{lem:progress}, 
\[ \phi(X_k) - \phi(X_{k+1}) \ge \left\{ \begin{array}{ll}
    h(\bar\Delta') - \mathcal{E}_k - \mathcal{E}_k^+ - r &\text{if } \Theta_k \Lambda_k' = 1  \\
    - \mathcal{E}_k - \mathcal{E}_k^+ - r &\text{if } \Theta_k' = 1 \\
    0 &\text{otherwise.}
\end{array} \right. 
\]
Thus, 
\[  \phi(x_0) - \hat\phi \ge \phi(x_0) - \phi(X_T) = \sum_{k=0}^{T-1} \phi(X_k) - \phi(X_{k+1}) 
    \ge \sum_{k=0}^{T-1} \Theta_k \Lambda_k' h(\bar\Delta') - \sum_{k=0}^{T-1} \Theta_k' \left( \mathcal{E}_k + \mathcal{E}_k^+ + r \right).
\]
Since using $\bar\Delta'$ over-complicates later analysis, we derive a slightly weaker inequality in \eqref{eq:total_progress} by using the fact that $h(\bar\Delta') > h(\gamma \bar\Delta)$.
\end{proof}

The next lemma bounds the difference between the number of iterations with $\Theta_k(1 - \Lambda_k')=1$ or $(1-\Theta_k)\Lambda_k=1$, where the TR radius moves towards the interval $[\bar\Delta', \bar\Delta]$, and the number of iterations with $(1-\Theta_k)(1-\Lambda_k)=1$ or $\Theta_k \Lambda_k'=1$, where the TR radius moves away from this interval. 
\begin{lemma}\label{lem:total_succ_unsucc}
For any positive integer $T$, we have 
\begin{equation}\label{eq:total_succ_unsucc}
    \sum_{k=0}^{T-1} \Theta_k(1 - \Lambda_k') - (1-\Theta_k)(1-\Lambda_k) + (1-\Theta_k)\Lambda_k - \Theta_k \Lambda_k' 
    \le \left|\log_\gamma \frac{\bar\Delta'}{\delta_0} \right|
    < \left|\log_\gamma \frac{\bar\Delta}{\delta_0} \right| + 1. 
\end{equation}
\end{lemma} 

\begin{proof}
Consider the sequence
\[\zeta_k = \max\left\{ \log(\Delta_k / \bar\Delta'), 0 \right\}. 
\]
This non-negative value starts at $\zeta_0 = \max\left\{ \log(\delta_0 / \bar\Delta'), 0 \right\}$ and increases by $-\log\gamma$ in iteration $k$ if $\Theta_k \Lambda_k'= 1$ or decreases by $-\log\gamma$ if $(1 -\Theta_k) \Lambda_k = 1$. 
Other types of iterations do not affect this value. 
Thus, 
\[
    \zeta_T = \zeta_0 + \sum_{k=0}^{T-1} - \Theta_k \Lambda_k' \log\gamma + (1 -\Theta_k) \Lambda_k \log\gamma \ge 0,
    \] 
    from which it follows that, 
    \[
    \sum_{k=0}^{T-1} - \Theta_k \Lambda_k' + (1 -\Theta_k) \Lambda_k \le - \frac{\zeta_0}{\log\gamma} = \max \left\{ \log_\gamma \frac{\bar\Delta'}{\delta_0}, 0 \right\}. 
\]
Similarly, consider the sequence
\[ \zeta_k' = \max\left\{ \log(\bar\Delta' / \Delta_k), 0 \right\}. 
\] 
It decreases by $-\log \gamma$ if $\Theta_k(1 - \Lambda_k') = 1$ and increases by $-\log \gamma$ if $(1 -\Theta_k)(1-\Lambda_k) = 1$. 
Thus,
\[ 
    \zeta_T' = \zeta_0' + \sum_{k=0}^{T-1} - (1 -\Theta_k)(1-\Lambda_k) \log\gamma + \Theta_k(1 - \Lambda_k') \log\gamma \ge 0 \]
    from which it follows that, 
    \[    
    \sum_{k=0}^{T-1} - (1 -\Theta_k)(1-\Lambda_k) + \Theta_k(1 - \Lambda_k') \le - \frac{\zeta_0'}{\log\gamma} = \max \left\{ \log_\gamma \frac{\delta_0}{\bar\Delta'}, 0 \right\}. 
 \]
The first inequality in \eqref{eq:total_succ_unsucc} follows by combining the above two results. The second inequality is trivially true. As in the previous lemma, we relax the right-hand side to a function of $\bar\Delta$ instead of $\bar\Delta'$ in order to simplify the subsequent analysis.
\end{proof}

The last lemma uses the Azuma-Hoeffding inequality to establish a probabilistic lower bound on the number of iterations with sufficiently accurate models. 
\begin{lemma} \label{lem:Azuma}
For any positive integer $T$ and any $\hat{p}_1 \in [0,p_1]$, we have 
\begin{equation} \label{eq:Azuma}
    \mathbbm{P} \left\{ \sum_{k=0}^{T-1} I_k > \hat{p}_1T \right\} 
    \ge 1 - \exp \left( - \frac{(1 - \hat{p}_1/p_1)^2}{2} T \right). 
\end{equation} 
\end{lemma}

\begin{proof}
Consider the submartingale $\left\{ \sum_{k=0}^{t-1} I_k - p_1 t \right\}_{t=0,1,\dots}$. 
Since 
\[ \left| \left( \sum_{k=0}^{(t+1)-1} I_k - p_1 (t+1) \right) - \left( \sum_{k=0}^{t-1} I_k - p_1 t \right) \right| = |I_t - p_1| \le \max\{|0-p_1|, |1-p_1|\} = p_1
\] 
for any $t \in \mathbb{N}$, by the Azuma-Hoeffding inequality, we have for any positive integer $T$ and any positive real $c$
\[ \mathbbm{P} \left\{ \sum_{k=0}^{T-1} I_k - Tp_1 \le -c \right\} 
    \le \exp \left( - \frac{c^2}{2Tp_1^2} \right). 
\]
Setting $c = (p_1 - \hat{p}_1) T$ and subtracting 1 from both sides yields the result.
\end{proof}

\subsection{Convergence Analysis: The Bounded Noise Case} \label{sec:convergence1_bounded}
We present in this subsection our result on Algorithm~\ref{alg:tr} with Oracle~\ref{oracle.zero}.\ref{ass:bounded noise}. 
The following lemma combines the inequalities from Lemmas~\ref{lem:total_progress}-\ref{lem:Azuma}. 

\begin{lemma} \label{lem:combine1bounded}
For any positive integer $T$ and any $\hat{p}_1 \in [0,p_1]$, we have 
\begin{equation} \label{eq:combine1bounded}
    \mathbbm{P} \left\{ \left(\hat{p}_1 - \frac{1}{2} - \frac{2\epsilon_f+r}{h(\gamma \bar\Delta)} \right) T < \frac{\phi(x_0) - \hat\phi}{h(\gamma \bar\Delta)} + \frac{1}{2}\left|\log_\gamma \frac{\bar\Delta}{\delta_0}\right| + \frac{1}{2} \right\}
    \ge 1 - \exp \left( - \frac{(1 - \hat{p}_1/p_1)^2}{2} T \right). 
\end{equation}
\end{lemma}
\begin{proof}
Multiply 
\[ \sum_{k=0}^{T-1} \Theta_k(1 - \Lambda_k') + (1-\Theta_k)(1-\Lambda_k) + (1-\Theta_k)\Lambda_k + \Theta_k \Lambda_k' = \sum_{k=0}^{T-1} 1 = T 
\]
by $(2\epsilon_f+r) / h(\gamma \bar\Delta) - 0.5$ and \eqref{eq:total_succ_unsucc} by 0.5, then add the results together to obtain
\begin{equation} \label{eq:temp1}
    \frac{2\epsilon_f+r}{h(\gamma \bar\Delta)} \left[ \sum_{k=0}^{T-1} 1 \right] - \left[ \sum_{k=0}^{T-1} (1 - \Theta_k)(1 - \Lambda_k) + \Theta_k \Lambda_k' \right] 
    < \left( \frac{2\epsilon_f+r}{h(\gamma \bar\Delta)} - \frac{1}{2} \right) T + \frac{1}{2} \left|\log_\gamma\frac{\bar\Delta}{\delta_0}\right| + \frac{1}{2}. 
\end{equation}
As a first step to proving \eqref{eq:combine1bounded}, we derive the following probabilistic bound: 
\[ \begin{aligned}
&\mathbbm{P}\left\{ \sum_{k=0}^{T-1} - \Theta_k \Lambda_k' + \frac{2\epsilon_f+r}{h(\gamma \bar\Delta)} \Theta_k' < \left(\frac{2\epsilon_f+r}{h(\gamma \bar\Delta)} + \frac{1}{2} - \hat{p}_1 \right)T + \frac{1}{2} \left|\log_\gamma \frac{\bar\Delta}{\delta_0}\right| + \frac{1}{2} \right\} \\
&\ge \mathbbm{P}\left\{ \left[ \sum_{k=0}^{T-1} (1 - \Theta_k) (1 - \Lambda_k) \right] - \frac{2\epsilon_f+r}{h(\gamma \bar\Delta)} \left[ \sum_{k=0}^{T-1} 1 - \Theta_k' \right] < (1 - \hat{p}_1) T \text{ and \eqref{eq:temp1} holds. } \right\} \\
&= \mathbbm{P}\left\{ \left[ \sum_{k=0}^{T-1} (1 - \Theta_k) (1 - \Lambda_k) \right] - \frac{2\epsilon_f+r}{h(\gamma \bar\Delta)} \left[ \sum_{k=0}^{T-1} 1 - \Theta_k' \right] < (1 - \hat{p}_1) T \right\} \\
&\ge \mathbbm{P}\left\{ \sum_{k=0}^{T-1} (1 - \Theta_k) (1 - \Lambda_k) + (1-I_k)\Theta_k + (1-I_k)(1-\Theta_k)\Lambda_k < (1 - \hat{p}_1) T \right\} \\
&= \mathbbm{P}\left\{ \sum_{k=0}^{T-1} I_k(1 - \Theta_k) (1 - \Lambda_k) + (1-I_k) < (1 - \hat{p}_1) T \right\} \\
&= \mathbbm{P}\left\{ \sum_{k=0}^{T-1} (1-I_k) < (1 - \hat{p}_1) T \right\} 
= \mathbbm{P}\left\{ \sum_{k=0}^{T-1} I_k > \hat{p}_1 T \right\} \\
&\leftstackrel{\eqref{eq:Azuma}}{\ge} 1 - \exp \left( - \frac{(1 - \hat{p}_1/p_1)^2}{2} T \right). 
\end{aligned} \]
The first inequality holds because the event in the first line is an inequality obtained  by the sum  of \eqref{eq:temp1} with the first inequality in the second line.
The second inequality holds because the left-hand side of the inequality in the fourth line is greater than or equal to the left-hand side of the inequality in the third line. 
The second last equality holds since $J_k = 1$ for all $k$ and then $\sum_{k=0}^{T-1} I_k(1 - \Theta_k) (1 - \Lambda_k) = 0$ due to Corollary~\ref{cor.1st}. 

Meanwhile, by Lemma~\ref{lem:total_progress} and Oracle~\ref{oracle.zero}.\ref{ass:bounded noise}, we have 
\[ \begin{aligned}
h(\gamma \bar\Delta) \sum_{k=0}^{T-1} \Theta_k \Lambda_k' &\le \phi(x_0) - \hat\phi + (2\epsilon_f+r) \sum_{k=0}^{T-1} \Theta_k' \\
\text{or equivalently } \quad\quad
- \frac{\phi(x_0) - \hat{\phi}}{h(\gamma \bar\Delta)} &\le -\sum_{k=0}^{T-1}  \Theta_k \Lambda_k' + \frac{2\epsilon_f+r}{h(\gamma \bar\Delta)} \Theta_k'. 
\end{aligned} \]
Combining the above two results, it follows that
\[ \mathbbm{P}\{\text{A}\} =
\mathbbm{P}\left\{ - \frac{\phi(x_0) - \hat{\phi}}{h(\gamma \bar\Delta)} <  \left(\frac{2\epsilon_f+r}{h(\gamma \bar\Delta)} + \frac{1}{2} - \hat{p}_1 \right)T + \frac{1}{2} \left|\log_\gamma\frac{\bar\Delta}{\delta_0}\right| + \frac{1}{2} \right\}
\ge 1 - \exp \left( - \frac{(1 - \hat{p}_1/p_1)^2}{2} T \right). 
\]
\end{proof}

Our main theorem follows the above lemma. 
\begin{theorem} \label{thm:convergence1_bounded}
Let Assumptions~\ref{assum:lip_cont} and \ref{assum:low_bound} hold for the objective function $\phi$. 
Let Assumptions~\ref{assum:bhm} and $r\ge 2\epsilon_f$ hold for Algorithm~\ref{alg:tr}. 
Given any $\epsilon > \sqrt{\frac{4\epsilon_f+2r}{C_3 \gamma^2 C_1^2 (2p_1-1)}} + \frac{C_2}{C_1} \epsilon_g$, where $C_1$, $C_2$ and $C_3$ are defined in \eqref{eq:c1c2} and \eqref{eq.h}, the sequence of iterates generated by Algorithm~\ref{alg:tr} with Oracle~\ref{oracle.zero}.\ref{ass:bounded noise} satisfies
\begin{equation} \label{eq:convergence1_bounded P}
    \mathbbm{P} \left\{ \min_{0\le k \le T-1} \|\nabla\phi(X_k)\| \le \epsilon \right\} \ge 1 - \exp \left( - \frac{(1 - \hat{p}_1/p_1)^2}{2} T \right) 
\end{equation}
for any $\hat{p}_1 \in \left(\frac{1}{2} + \frac{2\epsilon_f+r}{C_3\gamma^2 (C_1 \epsilon-C_2\epsilon_g)^2}, p_1 \right]$ and any  
\begin{equation} \begin{aligned}\label{eq:convergence1_bounded T} 
    T \ge& \left(\hat{p}_1 - \frac{1}{2} - \frac{2\epsilon_f+r}{C_3 \gamma^2 (C_1 \epsilon - C_2\epsilon_g)^2} \right)^{-1} \\ 
    &\left[\frac{\phi(x_0) - \hat\phi}{C_3 \gamma^2 (C_1 \epsilon-C_2\epsilon_g)^2} + \frac{1}{2}\log_\gamma \left( 
        \min\left\{ \frac{C_1 \epsilon-C_2\epsilon_g}{\delta_0}, \frac{\delta_0}{C_1\|\nabla\phi(x_0)\|-C_2\epsilon_g} \right\} 
    \right) + \frac{1}{2}\right]. 
\end{aligned} \end{equation}
\end{theorem}

\begin{proof}
Recall $\bar\Delta $ is defined as $C_1\min_{0\le k \le T-1} \|\nabla\phi(X_k)\| - C_2 \epsilon_g$. 
If $\bar\Delta\le0$, then $\min_{0\le k \le T-1} \|\nabla\phi(X_k)\| \le C_2 \epsilon_g / C_1$, and the result is trivially true. 
Now assume $\bar\Delta>0$. Consider the univariate function
\[ Q(y) = \left(\hat{p}_1 - \frac{1}{2} - \frac{2\epsilon_f+r}{C_3 \gamma^2 (C_1 y - C_2 \epsilon_g)^2} \right)^{-1}
\left[ \frac{\phi(x_0) - \hat\phi}{C_3 \gamma^2 (C_1 y - C_2 \epsilon_g)^2} + \frac{1}{2}\left|\log_\gamma \frac{C_1 y - C_2 \epsilon_g}{\delta_0}\right| + \frac{1}{2} \right]. 
\]
The probabilistic bound \eqref{eq:combine1bounded} we proved in Lemma~\ref{lem:combine1bounded} is
\[ \mathbbm{P} \left\{ T < Q\left(\min_{0\le k\le T-1} \|\nabla\phi(X_k)\|\right) \right\}
\ge 1 - \exp \left( - \frac{(1 - \hat{p}_1/p_1)^2}{2} T \right), 
\] 
and \eqref{eq:convergence1_bounded T} implies $T \ge Q(\epsilon)$. 
To prove \eqref{eq:convergence1_bounded P}, we only need to show that $T < Q\left(\min_{0\le k\le T-1} \|\nabla\phi(X_k)\|\right)$ implies $\min_{0\le k \le T-1} \|\nabla\phi(X_k)\| \le \epsilon$. 
Suppose, for the sake of contradiction, that $T < Q\left(\min_{0\le k\le T-1} \|\nabla\phi(X_k)\|\right)$ but $\min_{0\le k \le T-1} \|\nabla\phi(X_k)\| > \epsilon$. 

If $\bar\Delta \le \delta_0$, we have $C_1\epsilon - C_2\epsilon_g \le C_1\min_{0\le k \le T-1} \|\nabla\phi(X_k)\| - C_2\epsilon_g = \bar\Delta \le \delta_0$ and the function $Q$ decreases monotonically between $\epsilon$ and $\min_{0\le k \le T-1} \|\nabla\phi(X_k)\|$ (recall $0<\gamma<1$). 
It follows that 
$ Q(\epsilon) > Q(\min_{0\le k \le T-1} \|\nabla\phi(X_k)\|), 
$
which contradicts 
$  Q(\epsilon) \stackrel{\eqref{eq:convergence1_bounded T}}{\le} T < Q(\min_{0\le k \le T-1} \|\nabla\phi(X_k)\|). 
$

Alternatively, if $\bar\Delta > \delta_0$, the condition $T < Q\left(\min_{0\le k\le T-1} \|\nabla\phi(X_k)\|\right)$ implies 
\[ \begin{aligned}
    T &< \left(\hat{p}_1 - \frac{1}{2} - \frac{2\epsilon_f+r}{h(\gamma \bar\Delta)} \right)^{-1} \left[ \frac{\phi(x_0) - \hat\phi}{h(\gamma \bar\Delta)} + \frac{1}{2} \log_\gamma \frac{\delta_0}{\bar\Delta} + \frac{1}{2} \right] \\
    &\le \left(\hat{p}_1 - \frac{1}{2} - \frac{2\epsilon_f+r}{h(\gamma \bar\Delta)} \right)^{-1} \left[ \frac{\phi(x_0) - \hat\phi}{h(\gamma \bar\Delta)} + \frac{1}{2} \log_\gamma \frac{\delta_0}{C_1\|\nabla\phi(x_0)\| - C_2\epsilon_g} + \frac{1}{2} \right] \\
    &< \left(\hat{p}_1 - \frac{1}{2} - \frac{2\epsilon_f+r}{C_3 \gamma^2 (C_1 \epsilon - C_2\epsilon_g)^2} \right)^{-1} \left[ \frac{\phi(x_0) - \hat\phi}{C_3 \gamma^2 (C_1 \epsilon - C_2\epsilon_g)^2} + \frac{1}{2} \log_\gamma \frac{\delta_0}{C_1\|\nabla\phi(x_0)\| - C_2\epsilon_g} + \frac{1}{2} \right], 
\end{aligned} \] 
where the last inequality holds because of the assumption $\min_{0\le k \le T-1} \|\nabla\phi(X_k)\| > \epsilon$ and the fact that the function 
\[ Q'(y) = \left(\hat{p}_1 - \frac{1}{2} - \frac{2\epsilon_f+r}{C_3 \gamma^2 (C_1 y - C_2 \epsilon_g)^2} \right)^{-1}
\left[ \frac{\phi(x_0) - \hat\phi}{C_3 \gamma^2 (C_1 y - C_2 \epsilon_g)^2} + \frac{1}{2}\log_\gamma \frac{\delta_0}{C_1\|\nabla\phi(x_0)\| - C_2\epsilon_g} + \frac{1}{2} \right] 
\]
decreases monotonically on $[\epsilon,+\infty)$. 
However, this contradicts \eqref{eq:convergence1_bounded T}. 
Thus, $T < Q\left(\min_{0\le k\le T-1} \|\nabla\phi(X_k)\|\right)$ implies $\min_{0\le k \le T-1} \|\nabla\phi(X_k)\| \le \epsilon$. 
\end{proof}

\begin{remark}
    The bound \eqref{eq:convergence1_bounded T} is rather complicated. Notice it can also be written as 
    \begin{equation} \begin{aligned} 
        T \ge& \frac{\phi(x_0) - \hat\phi}{(\hat{p}_1 - 0.5)C_3 \gamma^2(C_1 \epsilon - C_2\epsilon_g)^2 - 2\epsilon_f- r} \\
        &+ \frac{\frac{1}{2} \log_\gamma \left( \min\left\{ \frac{C_1 \epsilon-C_2\epsilon_g}{\delta_0}, \frac{\delta_0}{C_1\|\nabla\phi(x_0)\|-C_2\epsilon_g} \right\}\right) + \frac{1}{2}}{(\hat{p}_1 - 0.5)C_3 \gamma^2(C_1 \epsilon - C_2\epsilon_g)^2 - 2\epsilon_f- r} C_3 \gamma^2 (C_1 \epsilon - C_2\epsilon_g)^2.
    \end{aligned} \end{equation}
    The first term on the right-hand side is term that results in $\mathcal{O}(\epsilon^{-2})$ complexity, which is typical (and optimal) for first-order optimization algorithms on smooth nonconvex functions. 
    The second term represents the number of iterations the algorithm takes to adjust the TR radius to a desired level. 
    If the initial radius $\delta_0$ is too big, then $\delta_0/(C_1\epsilon-C_2\epsilon_f)$ is larger and the second term represents the number of iterations needed for the TR to shrink to a level that can achieve a solution with $\epsilon$-accuracy. 
    Alternatively, if the initial radius is too small, then $(C_1\|\nabla\phi(x_0)\|-C_2\epsilon_g) / \delta_0$ is larger and the second term represents the number of iterations needed for the TR to expand to a level so that meaningful steps can be taken. 
    Since $\epsilon$ is always assumed to be small, the second term is typically much smaller than the first term. 
    Additionally, regarding the best achievable accuracy, condition $\epsilon > \sqrt{\frac{4\epsilon_f + 2r}{C_3 \gamma^2 C_1^2 (2p_1-1)}} + \frac{C_2}{C_1}\epsilon_g$ together with $r\ge 2\epsilon_f$ gives a lower bound that can be roughly summarized as $\epsilon\geq{\cal O}(\sqrt{\epsilon_f}) + {\cal O}(\epsilon_g)$. 
\end{remark}

The above theorem has a ``moving component" in $\hat{p}_1$. 
Maximizing the right-hand side of \eqref{eq:convergence1_bounded P} over $\hat{p}_1$ subject to the constraint \eqref{eq:convergence1_bounded T} gives us the optimal value for $\hat{p}_1$, which is 
\[ \frac{1}{2} + \frac{2\epsilon_f+r}{C_3 \gamma^2 (C_1\epsilon-C_2\epsilon_g)^2} + \frac{1}{T}  \left[\frac{\phi(x_0) - \hat\phi}{C_3 \gamma^2 (C_1\epsilon-C_2\epsilon_g)^2} + \frac{1}{2}\log_\gamma \left( 
    \min\left\{ \frac{C_1 \epsilon-C_2\epsilon_g}{\delta_0}, \frac{\delta_0}{C_1\|\nabla\phi(x_0)\|-C_2\epsilon_g} \right\} 
\right) + \frac{1}{2}\right].
\]
By setting $\hat{p}_1$ to this value, we have the following corollary to Theorem~\ref{thm:convergence1_bounded}. 
\begin{corollary}
Under the settings of Theorem~\ref{thm:convergence1_bounded}, given any $\epsilon > \sqrt{\frac{4\epsilon_f+2r}{C_3 \gamma^2 C_1^2 (2p_1-1)}} + \frac{C_2}{C_1} \epsilon_g$, it follows that
{\small 
\begin{multline*} 
\mathbbm{P} \left\{ \min_{0\le k \le T-1} \|\nabla\phi(X_k)\| \le \epsilon \right\} \ge 1 - \exp \left( - \frac{1}{2p_1^2 T} \left\{ \left(p_1 - \frac{1}{2} - \frac{2\epsilon_f+r}{C_3\gamma^2 (C_1\epsilon-C_2\epsilon_g)^2} \right) T \right. \right. \\
\left. \left. - \left[ \frac{\phi(x_0) - \hat\phi}{C_3\gamma^2 (C_1\epsilon-C_2\epsilon_g)^2} + \log_\gamma \left( 
    \min\left\{ \frac{C_1 \epsilon-C_2\epsilon_g}{\delta_0}, \frac{\delta_0}{C_1\|\nabla\phi(x_0)\|-C_2\epsilon_g} \right\} 
\right) + \frac{1}{2} \right] \right\}^2 \right)
\end{multline*}
}
for any $T$ following \eqref{eq:convergence1_bounded T} with the optimal $\hat{p}_1$. 
\end{corollary}
 
\subsection{Convergence Analysis: The Subexponential Noise Case}
\label{sec:convergence1_subexp}
We now extend the analysis for the use of Oracle~\ref{oracle.zero}.\ref{ass:subexponential noise}. 
Since Lemmas~\ref{lem:phi-m}, \ref{lem:delta accept}--\ref{lem:progress}, \ref{lem:true}, and \ref{lem:total_progress}--\ref{lem:Azuma} still hold when Oracle~\ref{oracle.zero}.\ref{ass:subexponential noise} is used instead of Oracle~\ref{oracle.zero}.\ref{ass:bounded noise}, we only need two additional lemmas before we can prove convergence. 
Lemma~\ref{lem:Azuma2} provides a guarantee on the number of iterations with favorable function evaluations ($J_k=1$). 

\begin{lemma} \label{lem:Azuma2}
Let 
\begin{equation} \label{eq:p0 1}
    p_0 = 1 - 2 \exp\left( a(\epsilon_f - r/2)\right), 
\end{equation} 
where $a$ is the positive constant from Oracle~\ref{oracle.zero}.\ref{ass:subexponential noise}.
Then, we have the submartingale condition 
\begin{equation} \label{eq:submartingale0}
	\mathbbm{P}\{J_k = 1 | \mathcal{F}_{k-1}'\} \geq p_0, \;  \text{ for all } k\in\{0,1,\dots\}, 
\end{equation}
where $\Fcal_{k-1}'$ is the $\sigma$-algebra defined in \eqref{eq:sig_alg2}. 
Furthermore, if $r > 2\epsilon_f + \frac{2}{a}\log 4$ (equivalent to $p_0 > 1/2$), then for any positive integer $T$ and any $\hat{p}_0 \in [0, p_0]$, it follows that
\begin{equation} \label{eq:Azuma2}
    \mathbbm{P} \left\{ \sum_{k=0}^{T-1} J_k > \hat{p}_0T \right\} 
    \ge 1 - \exp \left( - \frac{(1 - \hat{p}_0/p_0)^2}{2} T \right). 
\end{equation} 
\end{lemma}

\begin{proof}
By the definition of Oracle~\ref{oracle.zero}.\ref{ass:subexponential noise}, we have 
\[\begin{aligned} 
    &\mathbbm{P}_{\xi^{(0)}_k} \left\{ |e(X_k,\xi^{(0)}_k)| > t \middle| \Fcal_{k-1}' \right\} \le \exp (a(\epsilon_f - t)) \\
    \text{and } &\mathbbm{P}_{\xi^{(0+)}_k} \left\{ |e(X_k^+,\xi^{(0+)}_k)| > t \middle| \Fcal_{k-1}' \right\} \le \exp (a(\epsilon_f - t))
\end{aligned}\]
for all $k\in\{0,1,\dots\}$. 
Consequently, 
\[ \begin{aligned}
\mathbbm{P} \left\{ J_k = 0 \middle| \Fcal_{k-1}' \right\}
&= \mathbbm{P} \left\{ r < \mathcal{E}_k + \mathcal{E}_k^+ \middle| \Fcal_{k-1}' \right\} \\
&\le \mathbbm{P} \left\{ \mathcal{E}_k > r/2 \text{ or } \mathcal{E}_k^+ > r/2 \middle| \Fcal_{k-1}' \right\} \\
&\le \mathbbm{P} \left\{ \mathcal{E}_k > r/2 \middle| \Fcal_{k-1}' \right\} + \mathbbm{P} \left\{ \mathcal{E}_k^+ > r/2 \middle| \Fcal_{k-1}' \right\} \\
&\le \exp\left( a(\epsilon_f - r/2)\right) + \exp\left( a(\epsilon_f - r/2)\right) = 1 - p_0. 
\end{aligned} \]
Thus, \eqref{eq:submartingale0} holds, and the random process $\left\{ \sum_{k=0}^{t-1} J_k - p_0 t \right\}_{t=0,1,\dots}$ 
is a submartingale. 
Since 
\[ \left| \left( \sum_{k=0}^{(t+1)-1} J_k - p_0 (t+1) \right) - \left( \sum_{k=0}^{t-1} J_k - p_0 t \right) \right| = |J_t - p_0| \le \max\{|0-p_0|, |1-p_0|\} = p_0
\] 
for any $t \in \mathbb{N}$, by the Azuma-Hoeffding inequality, we have for any positive integer $T$ and any positive real $c$
\[ \mathbbm{P} \left\{ \sum_{k=0}^{T-1} J_k - Tp_0 \le -c \right\} 
    \le \exp \left( - \frac{c^2}{2Tp_0^2} \right). 
\]
Setting $c = (p_0 - \hat{p}_0) T$ and subtracting 1 from both sides yields \eqref{eq:Azuma2}.
\end{proof}

We largely rely on \cite[Chapter 2]{HDPbook} when it comes to the analysis of subexponential random variables, but the results in that book are not sufficient for our convergence analysis as the subexponential distribution defined there is controlled by one parameter, whereas the one in Oracle~\ref{oracle.zero}.\ref{ass:subexponential noise} is controlled by two parameters, $a$ and $\epsilon_f$. 
A somewhat stronger result is stated in the following proposition and its proof offered in Appendix~\ref{app:proofs}. 
\begin{proposition} \label{prop:subexponential}
Let $X$ be a random variable such that for some $a > 0$ and $b \ge 0$,
\begin{equation} \label{eq:subexponential}
    \mathbbm{P}\{|X| \ge t\} \le \exp (a(b-t)), \quad \text{for all } t>0. 
\end{equation}
Then, it follows that
\begin{equation} \label{eq:E exp lambda X}
    \mathbb{E} \exp(\lambda |X|) \le \frac{1}{1 - \lambda/a} \exp(\lambda b),  \quad \text{for all } \lambda \in \left[0, a\right).
\end{equation}
\end{proposition}

With Proposition~\ref{prop:subexponential}, we establish in Lemma~\ref{lem:"Bernstein"} a probabilistic upper bound on the total increase in the objective function value caused by the noise in the function evaluations. 
\begin{lemma} \label{lem:"Bernstein"}
With Oracle~\ref{oracle.zero}.\ref{ass:subexponential noise}, for any $t \ge 0$, 
\begin{equation} \label{eq:"Bernstein"}
    \mathbbm{P} \left\{\sum_{k=0}^{T-1} \Theta_k' \left( \mathcal{E}_k + \mathcal{E}_k^+ \right) \ge T(4/a + 2\epsilon_f) + t \right\}
    \le \exp\left( - \frac{a}{4} t \right). 
\end{equation}
\end{lemma}

\begin{proof}
With Oracle~\ref{oracle.zero}.\ref{ass:subexponential noise}, $\mathbbm{P} \left\{ \mathcal{E}_k \ge t \middle| \Fcal_{k-1}'\right\} \le \exp (a(\epsilon_f - t))$, and by Proposition~\ref{prop:subexponential} it follows that 
\[ \mathbb{E} \left\{\exp(2\lambda \mathcal{E}_k) \middle| \Fcal_{k-1}'\right\} \le \frac{1}{1 - 2\lambda/a} \exp(2 \lambda \epsilon_f) \quad \text{for all } \lambda \in \left[0, \frac{a}{2} \right). 
\]
A similar result applies to $\mathcal{E}_k^+$. 
Then, by the Cauchy-Schwarz inequality, it follows that for all $\lambda \in [0, a/2)$
\begin{equation} \label{eq:Cauchy-Schwarz} \begin{aligned}
\mathbb{E} \left\{\exp(\lambda \mathcal{E}_k + \lambda \mathcal{E}_k^+) \middle| \Fcal_{k-1}'\right\}
&\le \sqrt{\mathbb{E} \left\{\exp(2\lambda \mathcal{E}_k)\middle| \Fcal_{k-1}'\right\} \cdot \mathbb{E} \left\{\exp(2\lambda \mathcal{E}_k^+) \middle| \Fcal_{k-1}'\right\} } \\
&\le \frac{1}{1 - 2\lambda/a} \exp(2 \lambda \epsilon_f). 
\end{aligned} \end{equation}
By Markov's inequality, for any $\lambda \in [0, a / 2)$, $t \ge 0$, and positive integer $T$, 
\begin{align*}
\mathbbm{P}\left\{ \sum_{k=0}^{T-1} \Theta_k' \left( \mathcal{E}_k + \mathcal{E}_k^+ \right)  \ge t \right\} 
&\le \mathbbm{P}\left\{ \sum_{k=0}^{T-1} \left( \mathcal{E}_k + \mathcal{E}_k^+ \right)  \ge t \right\} 
= \mathbbm{P}\left\{ \exp\left( \lambda \sum_{k=0}^{T-1} \left( \mathcal{E}_k + \mathcal{E}_k^+ \right) \right)  \ge \exp(\lambda t) \right\} \\
&\le e^{-\lambda t} \mathbb{E} \left\{ \exp\left( \lambda \sum_{k=0}^{T-1} \left( \mathcal{E}_k + \mathcal{E}_k^+ \right) \right) \right\} 
\le e^{-\lambda t} \left(\frac{1}{1 - 2\lambda/a} \exp(2 \lambda \epsilon_f)\right)^T, 
\end{align*}
where the last inequality can be proved by induction as follows. 
Firstly, this inequality holds for $T=1$ due to \eqref{eq:Cauchy-Schwarz}. 
Now if it holds for any positive integer $T$, then for $T+1$ 
\begin{align*}
& \;\;\;\; e^{-\lambda t} \mathbb{E} \left\{ \exp\left( \lambda \sum_{k=0}^{(T+1)-1} \left( \mathcal{E}_k + \mathcal{E}_k^+ \right) \right) \right\} \\
&= e^{-\lambda t} \mathbb{E} \left\{ \exp\left( \lambda \sum_{k=0}^{T-1} \left( \mathcal{E}_k + \mathcal{E}_k^+ \right) \right) \E_{\xi^{(0)}_T, \xi^{(0+)}_T} \left[ \exp\left( \lambda\mathcal{E}_{T} + \lambda\mathcal{E}_{T}^+ \right) \middle| \Fcal_{T-1}'\right] \right\} \\
&\leftstackrel{\eqref{eq:Cauchy-Schwarz}}{\le} e^{-\lambda t} \mathbb{E} \left\{ \exp\left( \lambda \sum_{k=0}^{T-1} \left( \mathcal{E}_k + \mathcal{E}_k^+ \right) \right) \frac{1}{1 - 2\lambda/a} \exp(2 \lambda \epsilon_f) \right\} \\
&\le e^{-\lambda t} \left(\frac{1}{1 - 2\lambda/a} \exp(2 \lambda \epsilon_f)\right)^T \frac{1}{1 - 2\lambda/a} \exp(2 \lambda \epsilon_f), 
\end{align*}
which shows this inequality holds for $T+1$ and thus for any positive integer $T$. 
For ease of exposition, we use the fact that $1/(1-x) \le \exp(2x)$ for all $x \in [0, 1/2]$ to simplify the above result 
\[ \mathbbm{P}\left\{ \sum_{k=0}^{T-1} \Theta_k' \left( \mathcal{E}_k + \mathcal{E}_k^+ \right)  \ge t \right\}
\le e^{-\lambda t} \left[ \exp(4\lambda/a)\exp(2 \lambda \epsilon_f) \right]^T
= \exp\left(\lambda \left[T (4/a + 2\epsilon_f) - t\right] \right)
\quad \text{for all } \lambda \in \left[0, \frac{a}{4} \right]. 
\]
Clearly the right-hand side is only less than or equal to 1 when $t \ge T(4/a + 2\epsilon_f)$, which makes the right-hand side a monotonically non-increasing function of $\lambda$. 
We choose $\lambda = a/4$ and apply a change of variable to obtain the final result. 
\end{proof}

Similar to Lemma~\ref{lem:combine1bounded}, we combine the inequalities from the established lemmas. 
\begin{lemma} \label{lem:combine1subexpo}
Under Oracle~\ref{oracle.zero}.\ref{ass:subexponential noise}, it holds for Algorithm~\ref{alg:tr} that 
\begin{equation} \label{eq:combine1subexpo} \begin{aligned} 
    \mathbbm{P} &\left\{ \left(\hat{p}_0 + \hat{p}_1 - \frac{3}{2} - \frac{2\epsilon_f + 4/a + r}{h(\gamma \bar\Delta)} \right) T < \frac{\phi(x_0) - \hat\phi + t}{h(\gamma \bar\Delta)} + \frac{1}{2}\left|\log_\gamma \frac{\bar\Delta}{\delta_0}\right| + \frac{1}{2} \right\} \\
    &\ge 1 - \exp \left( - \frac{(1 - \hat{p}_1/p_1)^2}{2} T \right) - \exp \left( - \frac{(1 - \hat{p}_0/p_0)^2}{2} T \right) - \exp\left( - \frac{a}{4} t \right)
\end{aligned} \end{equation}
for any positive integer $T$, any $\hat{p}_1 \in [0,p_1]$, and any $t \ge 0$. 
\end{lemma}

\begin{proof}
Multiply 
\[ \sum_{k=0}^{T-1} \Theta_k(1 - \Lambda_k') + (1-\Theta_k)(1-\Lambda_k) + (1-\Theta_k)\Lambda_k + \Theta_k \Lambda_k' = \sum_{k=0}^{T-1} 1 = T
\]
with $(2\epsilon_f + 4/a + r) / h(\gamma \bar\Delta) + 0.5$ and \eqref{eq:total_succ_unsucc}  with 0.5 and add the two expressions to obtain 
\begin{equation} \label{eq:temp2}
    \frac{2\epsilon_f + 4/a + r}{h(\gamma \bar\Delta)} T + \left[ \sum_{k=0}^{T-1} \Theta_k(1 - \Lambda_k') + (1-\Theta_k)\Lambda_k \right] 
    < \left( \frac{2\epsilon_f + 4/a + r}{h(\gamma \bar\Delta)} + \frac{1}{2} \right) T + \frac{1}{2} \left|\log_\gamma \frac{\bar\Delta}{\delta_0}\right| + \frac{1}{2}. 
\end{equation}
Then, 
\begin{align*}
&\mathbbm{P}\left\{ \frac{2\epsilon_f + 4/a + r}{h(\gamma \bar\Delta)}T - \sum_{k=0}^{T-1} \Theta_k \Lambda_k' < \left(\frac{2\epsilon_f + 4/a + r}{h(\gamma \bar\Delta)} + \frac{3}{2} - \hat{p}_0 - \hat{p}_1 \right)T + \frac{1}{2} \left|\log_\gamma \frac{\bar\Delta}{\delta_0}\right| + \frac{1}{2} \right\} \\
&\ge \mathbbm{P}\left\{ \sum_{k=0}^{T-1} - \Theta_k - (1 - \Theta_k)\Lambda_k < (1 - \hat{p}_0 - \hat{p}_1) T \text{ and \eqref{eq:temp2} holds. } \right\} \\
&= \mathbbm{P}\left\{ \sum_{k=0}^{T-1} - \Theta_k - (1 - \Theta_k)\Lambda_k < (1 - \hat{p}_0 - \hat{p}_1) T \right\} \\
&\ge \mathbbm{P}\left\{ \sum_{k=0}^{T-1} \left[ - \Theta_k - (1 - \Theta_k) \Lambda_k \right] I_k J_k + (1-I_k)(1-J_k) < (1 - \hat{p}_0 - \hat{p}_1) T \right\} \\
&= \mathbbm{P}\left\{ \sum_{k=0}^{T-1} - I_k J_k + (1-I_k)(1-J_k) < (1 - \hat{p}_0 - \hat{p}_1) T \right\} \\
&\ge \mathbbm{P}\left\{ \sum_{k=0}^{T-1} I_k > \hat{p}_1 T \text{ and } \sum_{k=0}^{T-1} J_k > \hat{p}_0 T \right\} \\
&= \mathbbm{P}\left\{ \sum_{k=0}^{T-1} I_k > \hat{p}_1 T \right\} + \mathbbm{P}\left\{ \sum_{k=0}^{T-1} J_k > \hat{p}_0 T \right\} - \mathbbm{P}\left\{ \sum_{k=0}^{T-1} I_k > \hat{p}_1 T \text{ or } \sum_{k=0}^{T-1} J_k > \hat{p}_0 T \right\} \\
&\ge \left[ 1 - \exp \left( - \frac{(1 - \hat{p}_1/p_1)^2}{2} T \right) \right] + \left[ 1 - \exp \left( - \frac{(1 - \hat{p}_0/p_0)^2}{2} T \right) \right] - 1 \\
&= 1 - \exp \left( - \frac{(1 - \hat{p}_1/p_1)^2}{2} T \right) - \exp \left( - \frac{(1 - \hat{p}_0/p_0)^2}{2} T \right), 
\end{align*}
where the second equality holds due to Corollary~\ref{cor.1st}, and the last inequality holds due to Lemmas~\ref{lem:Azuma} and \ref{lem:Azuma2}. 
Unlike the bounded noise case, Lemmas~\ref{lem:total_progress} and \ref{lem:"Bernstein"} imply
\begin{align*} 
h(\gamma \bar\Delta) \sum_{k=0}^{T-1} \Theta_k \Lambda_k' 
&\le \phi(x_0) - \hat\phi + \sum_{k=0}^{T-1} \Theta_k' \left( \mathcal{E}_k + \mathcal{E}_k^+ + r \right)\\
&\le \phi(x_0) - \hat\phi + (2\epsilon_f + 4/a + r)T + t \\
\text{or equivalently } \quad\quad
- \frac{\phi(x_0) - \hat{\phi} + t}{h(\gamma \bar\Delta)} 
&\le \frac{2\epsilon_f + 4/a + r}{h(\gamma \bar\Delta)}T - \sum_{k=0}^{T-1} \Theta_k \Lambda_k',
\end{align*}
with probability at least $1 - \exp(-at/4)$ for any $t \ge 0$. 
Thus, we can combine the two previous results to obtain \eqref{eq:combine1subexpo}. 
\end{proof}

We present our high probability complexity bound for Algorithm~\ref{alg:tr} with Oracle~\ref{oracle.zero}.\ref{ass:subexponential noise} in the following theorem. 
The proof is analogous to that of Theorem~\ref{thm:convergence1_bounded} and hence omitted. 
\begin{theorem} \label{thm:convergence1_subexponential}
Let Assumptions~\ref{assum:lip_cont} and \ref{assum:low_bound}  hold for the objective function $\phi$. 
Let Assumption~\ref{assum:bhm} and $r\ge 2\epsilon_f$ hold for Algorithm~\ref{alg:tr}.
Let $p_0$ be defined as in \eqref{eq:p0 1}. 
Given any $\epsilon > \sqrt{\frac{4\epsilon_f + 8/a + 2r}{C_3 \gamma^2 C_1^2 (2p_0+2p_1-3)}} + \frac{C_2}{C_1}\epsilon_g$, where $C_1$, $C_2$ and  $C_3$ are defined in \eqref{eq:c1c2} and \eqref{eq.h}, the sequence of iterates generated by Algorithm~\ref{alg:tr} with Oracle \ref{oracle.zero}.\ref{ass:subexponential noise} satisfy  
\begin{equation} \label{eq:convergence1 subexponential P}
    \mathbbm{P} \left\{ \min_{0\le k \le T-1} \|\nabla\phi(X_k)\| \le \epsilon \right\}  \ge 1 - \exp \left( - \frac{(1 - \hat{p}_1/p_1)^2}{2} T \right) - \exp \left( - \frac{(1 - \hat{p}_0/p_0)^2}{2} T \right) - \exp\left( - \frac{a}{4} t \right)
\end{equation}
for any $\hat{p}_0$ and $\hat{p}_1$ such that $\hat{p}_0 + \hat{p}_1 \in \left(\frac{3}{2} + \frac{2\epsilon_f + 4/a + r}{C_3\gamma^2 (C_1\epsilon-C_2\epsilon_g)^2}, p_0 + p_1 \right]$, any $t \ge 0$, and any 
\begin{equation} \label{eq:convergence1 subexponential T} \begin{aligned} 
    T \ge& \left(\hat{p}_0 + \hat{p}_1 - \frac{3}{2} - \frac{2\epsilon_f + 4/a + r}{C_3 \gamma^2 (C_1\epsilon-C_2\epsilon_g)^2} \right)^{-1} \\
    &\left[\frac{\phi(x_0) - \hat\phi + t}{C_3 \gamma^2 (C_1\epsilon-C_2\epsilon_g)^2} + \frac{1}{2}\log_\gamma \left( 
        \min\left\{ \frac{C_1 \epsilon-C_2\epsilon_g}{\delta_0}, \frac{\delta_0}{C_1\|\nabla\phi(x_0)\|-C_2\epsilon_g} \right\} 
    \right) + \frac{1}{2}\right]. 
\end{aligned} \end{equation}
\end{theorem}

Similar to Theorem~\ref{thm:convergence1_bounded}, one can optimize the bounds in Theorem~\ref{thm:convergence1_subexponential} by maximizing the right-hand side of \eqref{eq:convergence1 subexponential P} while satisfying \eqref{eq:convergence1 subexponential T}. However, since there is no closed form solution, we omit this corollary.

\section{Second-order stochastic convergence analysis}
\label{sec:analysis2}

The convergence analysis of Algorithm~\ref{alg:tr2} is analogous to that of Algorithm~\ref{alg:tr} with the difference that the goal is for the algorithm to find an approximate second-order critical point, i.e., a point $x$ such that $\max\{\|\nabla\phi(x)\|, -\lambda_{\min}(\nabla^2\phi(x))\} \le \epsilon$ for some sufficiently large $\epsilon$. 
However, to keep our results concise, we define an optimality measure
\begin{equation*}  
    \beta(x) \stackrel{\rm def}{=}  \max\left\{ 
        C_4\|\nabla\phi(x)\| - C_5\epsilon_g, 
        -C_6 \lambda_{\rm min}(\nabla^2\phi(x)) - C_7 \epsilon_H
    \right\}, 
\end{equation*}
where $C_4,C_5,C_6$, and $C_7$ are positive constants defined in \eqref{eq.constants}. 
The main goal of this section is to derive a probabilistic result of the form
\[ \mathbbm{P} \left\{ \min_{0 \le k \le T -1} \beta(X_k) < \epsilon' \right\} \ge \text{a function of $T$ that converges to 1 as $T$ increases}
\]
for any sufficiently large $\epsilon'$. 
It should be clear that $\epsilon$ is simply a scaled version of $\epsilon'$ plus the errors coming from $\epsilon_g$ and $\epsilon_H$. 

In what follows, we are mainly concerned with iterations where $\delta_k\leq \beta(x_k)$. 
To simplify some of the analysis we find that it is useful to have $\delta_k\leq 1$ whenever $\delta_k\leq \beta(x_k)$. To this end,  
we introduce the  following simple and technical assumption. 
It plays a similar role as \cite[Assumption~6b]{blanchet2019convergence}. 
Such an assumption can be avoided, but at the expense of a significant increase in the complexity of the analysis. 
\begin{assumption}[\textbf{Upper bound on convergence accuracy}] \label{assum:upper_bound_e} 
The optimization problem is properly scaled so that the desired accuracy satisfies $\epsilon' \le 1$ and the irreducible noise constants satisfy $\epsilon_g, \epsilon_H \ll 1$. 
\end{assumption}	
Under this assumption, we can see that we are mainly interested in iterations where $\beta(x_k)\leq 1$, thus, assuming $\delta_k\leq 1$ whenever   $\delta_k\leq \beta(x_k)$ is without loss of generality.

Analogous to Section~\ref{sec:analysis1}, we begin by stating and proving three key lemmas about the behavior of Algorithm~\ref{alg:tr2} under Assumptions~\ref{ass:lip_cont2} and \ref{assum:bhm} and assuming \eqref{eq:sufficiently accurate H g} holds. The first lemma provides a sufficient condition for accepting a step.
\begin{lemma}[\bf Sufficient condition for accepting step] \label{lem:delta accept2}
Under Assumptions~\ref{ass:lip_cont2} and \ref{assum:bhm}, if \eqref{eq:sufficiently accurate H g} holds, $r \ge e_k^+ - e_k  + \epsilon_g^{3/2}$, $\delta_k\leq 1$, and 
\begin{equation} \label{eq:delta accept2} 
    \delta_k \le \max \left\{ \min\left\{ \frac{1}{\kappa_{\rm bhm}}, 
        \frac{(1-\eta_1)\kappa_{\rm fod}}{L_2/3+2\kappa_{eg}+2+\kappa_{\rm eh} + \epsilon_H} 
        \right\}\|g_k\|, 
        \frac{-(1-\eta_1)\kappa_{\rm fod}\lambda_{\rm min}(H_k) - \epsilon_H}{L_2/3 + 2\kappa_{\rm eg} + 2 + \kappa_{\rm eh}} 
    \right\}, 
\end{equation}
then $\rho_k \geq \eta_1$ in Algorithm~\ref{alg:tr2}. 
\end{lemma}
\begin{proof}
First, 
\[ \begin{aligned} 
    \rho_k &= \frac{\phi(x_k) + e_k - \phi(x_k + s_k) - e_k^+ + r}{m_k(x_k) - m_k(x_k + s_k)} \\
    &\ge \frac{\phi(x_k) - \phi(x_k + s_k) + \epsilon_g^{3/2}}{m_k(x_k) - m_k(x_k + s_k)} \\
    &\leftstackrel{\eqref{eq:phi-m2}}{\ge} \frac{\phi(x_k) - m_k(x_k + s_k) - (L_2/6+\kappa_{\rm eg}+1+\kappa_{\rm eh}/2) \delta_k^3 - \epsilon_H \delta_k^2/2}{m_k(x_k) - m_k(x_k + s_k)} \\
&\leftstackrel{\eqref{eq:model_def}}{=} 1 - \frac{(L_2/6+\kappa_{\rm eg}+1+\kappa_{\rm eh}/2) \delta_k^3 + \epsilon_H \delta_k^2/2}{m_k(x_k) - m_k(x_k + s_k)} \\
    &\leftstackrel{\eqref{eq:Eigen decrease}}{\ge} 1 - \frac{(L_2/3+2\kappa_{\rm eg}+2+\kappa_{\rm eh}) \delta_k^3 + \epsilon_H \delta_k^2}{\kappa_{\rm fod} \max\{\|g_k\| \min\{\|g_k\|/\|H_k\|, \delta_k\}, -\lambda_{\rm min}(H_k)\delta_k^2\}}.  
\end{aligned} \]
We consider two cases. If the first term in the maximization operation in \eqref{eq:delta accept2} is larger, it follows that $\delta_k \le \|g_k\| / \kappa_{\rm bhm} \le \|g_k\| / \|H_k\|$ and 
\[ \rho_k 
\ge 1 - \frac{(L_2/3+2\kappa_{\rm eg}+2+\kappa_{\rm eh}) \delta_k^3 + \epsilon_H \delta_k^2}{\kappa_{\rm fod} \|g_k\|\delta_k}
\ge 1 - \frac{(L_2/3+2\kappa_{\rm eg}+2+\kappa_{\rm eh})\delta_k^2 + \epsilon_H \delta_k^2}{\kappa_{\rm fod} \|g_k\|\delta_k}
\stackrel{\eqref{eq:delta accept2}}{\ge} 1 - (1-\eta_1) = \eta_1. 
\]

If the second term in the maximization operation in \eqref{eq:delta accept2} is larger, then 
\[ \rho_k 
\ge 1 - \frac{(L_2/3+2\kappa_{\rm eg}+2+\kappa_{\rm eh}) \delta_k^3 + \epsilon_H \delta_k^2}{- \kappa_{\rm fod} \lambda_{\rm min}(H_k)\delta_k^2} 
\stackrel{\eqref{eq:delta accept2}}{\ge} 1 - (1-\eta_1) = \eta_1. 
\]
\end{proof}

The next result provides a sufficient condition for a successful step. 
\begin{lemma}[\bf Sufficient condition for successful step] \label{lem:delta success2}
Under Assumptions~\ref{ass:lip_cont2} and \ref{assum:bhm}, if \eqref{eq:sufficiently accurate H g} holds, $r \ge e_k^+ - e_k + \epsilon_g^{3/2}$, $\delta_k\leq 1$, and 
\begin{equation} \label{eq:delta success2} 
    \delta_k \le \beta(x_k) \stackrel{\rm def}{=} 
    \max\left\{ 
    C_4\|\nabla\phi(x_k)\| - C_5\epsilon_g, 
    -C_6 \lambda_{\rm min}(\nabla^2\phi(x_k)) - C_7 \epsilon_H
    \right\}, 
\end{equation} 
where 
\begin{equation} \begin{aligned} \label{eq.constants}
    C_4 &\stackrel{\rm def}{=} \min\left\{ 
    \frac{1}{\kappa_{\rm bhm}+\kappa_{\rm eg}}, 
    \frac{(1-\eta_1)\kappa_{\rm fod}}{(L_2/3+2\kappa_{\rm eg}+3+\kappa_{\rm eh}) + (1-\eta_1)\kappa_{\rm fod}\kappa_{\rm eg}}, 
    \frac{1}{\kappa_{\rm eg} + \eta_2}
    \right\}\\ 
    C_5 &\stackrel{\rm def}{=} \max \left\{ 
    \frac{1}{\kappa_{\rm bhm}+\kappa_{\rm eg}}, 
    \frac{(1-\eta_1)\kappa_{\rm fod}}{(L_2/3+2\kappa_{\rm eg}+3+\kappa_{\rm eh}) + (1-\eta_1)\kappa_{\rm fod}\kappa_{\rm eg}}, 
    \frac{1}{\kappa_{\rm eg} + \eta_2}
    \right\} \\
    C_6 &\stackrel{\rm def}{=} \min\left\{
    \frac{(1-\eta_1)\kappa_{\rm fod}}{L_2/3+2\kappa_{\rm eg}+2+\kappa_{\rm eh} + (1-\eta_1)\kappa_{\rm fod}\kappa_{\rm eh}}, 
    \frac{1}{\kappa_{\rm eh} + \eta_2}
    \right\}\\
    C_7 &\stackrel{\rm def}{=} \max\left\{ 
    \frac{(1-\eta_1)\kappa_{\rm fod} + 1}{L_2/3+2\kappa_{\rm eg}+2+\kappa_{\rm eh} + (1-\eta_1)\kappa_{\rm fod}\kappa_{\rm eh}}, 
    \frac{1}{\kappa_{\rm eh} + \eta_2}
    \right\}
\end{aligned} 
\end{equation}
then, $\rho_k \geq \eta_1$ and $\beta^m_k \ge \eta_2 \delta_k$ in Algorithm~\ref{alg:tr2}. 
\end{lemma}
\begin{proof}
By \eqref{eq:sufficiently accurate H g} we have
\begin{align*}
        \|g_k\| &\ge \|\nabla\phi(x_k)\| - \|g_k - \nabla\phi(x_k)\| \qquad &&\text{and} &-\lambda_{\rm min}(H_k) &\ge -\lambda_{\rm min}(\nabla^2\phi(x_k)) - \|H_k - \nabla^2\phi(x_k)\|  \\
        &\ge \|\nabla\phi(x_k)\| - \kappa_{\rm eg} \delta_k^2 - \epsilon_g, &&~ &&\ge -\lambda_{\rm min}(\nabla^2\phi(x_k)) - \kappa_{\rm eh} \delta_k - \epsilon_H.
\end{align*}
We first show $\rho_k \ge \eta_1$. We consider two cases. If the first term in the maximization operation in \eqref{eq:delta success2} is larger, then
\[ \frac{\|g_k\|}{\kappa_{\rm bhm}} 
\ge \frac{\|\nabla\phi(x_k)\| - \kappa_{\rm eg}\delta_k^2 - \epsilon_g}{\kappa_{\rm bhm}}
\ge \frac{\|\nabla\phi(x_k)\| - \kappa_{\rm eg}\delta_k - \epsilon_g}{\kappa_{\rm bhm}}
\stackrel{\eqref{eq:delta success2}}{\ge} \delta_k, 
\]
and 
\[ \frac{(1-\eta_1)\kappa_{\rm fod}}{(L_2/3+2\kappa_{eg}+2+\kappa_{\rm eh})+\epsilon_H} \|g_k\|
\ge \frac{(1-\eta_1)\kappa_{\rm fod}}{(L_2/3+2\kappa_{eg}+2+\kappa_{\rm eh})+1} (\|\nabla\phi(x_k)\| - \kappa_{\rm eg}\delta_k - \epsilon_g)
\stackrel{\eqref{eq:delta success2}}{\ge} \delta_k. 
\]

Alternatively, if the second term in the maximization operation in \eqref{eq:delta success2} is larger, then 
\[  \frac{-(1-\eta_1)\kappa_{\rm fod}\lambda_{\rm min}(H_k) - \epsilon_H}{L_2/3 + 2\kappa_{\rm eg} + 2 + \kappa_{\rm eh}} 
    \ge \frac{(1-\eta_1) \kappa_{\rm fod}(-\lambda_{\rm min}(\nabla^2\phi(x_k)) - \kappa_{\rm eh}\delta_k - \epsilon_H) - \epsilon_H}{L_2/3 + 2\kappa_{\rm eg} + 2 + \kappa_{\rm eh}} 
    \stackrel{\eqref{eq:delta success2}}{\ge} \delta_k. 
\]
Thus, $\rho_k \ge \eta_1$ according to Lemma~\ref{lem:delta accept2}. 

For the condition $\beta^m_k \ge \eta_2 \delta_k$, first consider the case where the first term in the maximization operation in \eqref{eq:delta success2} is larger, then  
\[  \|g_k\| \ge \|\nabla\phi(x_k)\| - \kappa_{\rm eg}\delta_k^2 - \epsilon_g 
\ge \|\nabla\phi(x_k)\| - \kappa_{\rm eg}\delta_k - \epsilon_g
\stackrel{\eqref{eq:delta success2}}{\ge} \eta_2 \delta_k. 
\]
If the second term in the maximization operation in \eqref{eq:delta success2} is larger, then 
\[ -\lambda_{\rm min}(H_k) 
\ge -\lambda_{\rm min}(\nabla^2\phi(x_k)) - \kappa_{\rm eh} \delta_k - \epsilon_H 
\stackrel{\eqref{eq:delta success2}}{\ge} \eta_2 \delta_k. 
\]
\end{proof}

The last result provides a bound on the progress made at each iteration. The proof uses similar arguments as in Lemma~\ref{lem:progress} and \cite[Lemma~3.7]{gratton2018complexity}. 
\begin{lemma}[\bf Progress made in each iteration] \label{lem:progress2}
In Algorithm~\ref{alg:tr2}, if $\rho_k \geq \eta_1$, $\delta_k\leq 1$ and $\beta^m_k \ge \eta_2 \delta_k$, then 
\begin{equation*}
    \phi(x_k) - \phi(x_{k+1}) \ge h(\delta_k) - e_k + e_k^+ - r, 
\end{equation*}
where 
\begin{equation}\label{eq.c_8}
	h(\delta) = C_8 \min\{1, \delta^{3}\} \text{, where }
    C_8 \stackrel{\rm def}{=} \frac{\eta_1 \eta_2 \kappa_{\rm fod}}{2}  \min\left\{ \frac{\eta_2}{\kappa_{\rm bhm}}, 1 \right\}. 
\end{equation}
If $\rho_k \ge \eta_1$ but $\beta^m_k < \eta_2 \delta_k$, then 
\begin{equation}
    \phi(x_k) - \phi(x_{k+1}) \ge - e_k + e_k^+ - r. 
\end{equation}
If $\rho_k < \eta_1$, then $\phi(x_{k+1}) = \phi(x_k)$. 
\end{lemma}

\begin{proof}
Similar to Lemma~\ref{lem:progress}, let $\rho_k \ge \eta_1$ and we have $\phi(x_k) - \phi(x_{k+1}) \ge \eta_1[m_k(x_k) - m_k(x_k + s_k)] - e_k + e_k^+ - r$. 
If $\beta^m_k = -\lambda_{\min}(H_k) \ge \eta_2 \delta_k$, the first term of this expression satisfies 
\[ \begin{aligned} 
     \eta_1[m_k(x_k) - m_k(x_k + s_k)] 
    \stackrel{\eqref{eq:Eigen decrease}}{\ge} \frac{\eta_1 \kappa_{\rm fod}}{2} \left( -\lambda_{\rm min}(H_k) \delta_k^2 \right) 
    \ge \frac{\eta_1 \kappa_{\rm fod}}{2} \eta_2 \delta_k^3; 
\end{aligned} \] 
If $\beta^m_k = \|g_k\| \ge \eta_2 \delta_k$, it satisfies 
\[ \begin{aligned} 
     \eta_1[m_k(x_k) - m_k(x_k + s_k)] ~
    &\leftstackrel{\eqref{eq:Eigen decrease}}{\ge} \frac{\eta_1 \kappa_{\rm fod}}{2} \|g_k\| \min\left\{ \frac{\|g_k\|}{\|H_k\|}, \delta_k \right\} 
    \ge \frac{\eta_1 \kappa_{\rm fod}}{2} \eta_2\delta_k \min\left\{ \frac{\eta_2\delta_k}{\kappa_{\rm bhm}}, \delta_k \right\} \\
         &\ge         \frac{\eta_1 \eta_2 \kappa_{\rm fod}}{2} \min\left\{ \frac{\eta_2}{\kappa_{\rm bhm}}, 1 \right\} \delta_k^3          
\end{aligned} \] 

If $\beta^m_k < \eta_2 \delta_k$, then 
\[  \eta_1[m_k(x_k) - m_k(x_k + s_k)] 
    \stackrel{\eqref{eq:Eigen decrease}}{\ge} \frac{\eta_1 \kappa_{\rm fod}}{2} \max\left\{ \|g_k\| \min\left\{ \frac{\|g_k\|}{\|H_k\|}, \delta_k\right\}, -\lambda_{\rm min}(H_k) \delta_k^2 \right\} 
    \ge 0. 
\]
If $\rho_k < \eta_1$, we have $x_{k+1} = x_k$, so $\phi(x_{k+1}) = \phi(x_k)$. 
\end{proof}

Next, to categorize the iterations $k=0,1,\dots, T-1$ into different types, we redefine the random indicator variables as follows: 
\begin{equation} \label{eq:indicator2} \begin{aligned}
    &I_k = \mathbbm{1}\left\{ \begin{array}{l}
        \|\nabla^2 M_k(X_k) - \nabla^2\phi(X_k)\| \le \kappa_{\rm eh} \Delta_k + \epsilon_H \\
        \text{ and } \|\nabla M_k(X_k) - \nabla\phi(X_k)\| \le \kappa_{\rm eg} \Delta_k^2 + \epsilon_g 
    \end{array} \right\},  \\
    &J_k = \mathbbm{1} \{r \ge \mathcal{E}^+_k + \mathcal{E}_k + \epsilon_g^{3/2} \} \\
    &\Theta_k = \mathbbm{1}\{\rho_k \ge \eta_1 \text{ and } \beta^m_k \ge \eta_2 \Delta_k \}, 
\end{aligned} \end{equation}
where $\beta^m_k$ is defined in \eqref{eq:beta_k_m} and $\bar\Delta$ is now defined as 
\begin{equation} 
    \bar\Delta = \min_{ 0 \le k \le T-1}\beta(X_k). 
\end{equation}
The interpretations of the indicators variables $I_k$, $J_k$ and $\Theta_k$ remain the same as in Section~\ref{sec:analysis1}, as does the definition of $\Theta_k'$. 
The definitions of $\Lambda_k, \Lambda_k'$, and $\bar\Delta'$ also remain the same but with the newly defined $\bar\Delta$. 

Similar to the first-order case, with the redefined random indicator variables, it follows from Lemma~\ref{lem:delta success2} that if $I_k J_k =1$ and $\Lambda_k=0$, then $\Theta_k = 1$ (and the step is successful). We formalize this in the corollary below.
\begin{corollary}[\bf Corollary to Lemma~\ref{lem:delta success2}]\label{cor.2nd} If $I_k J_k =1$ and $\Lambda_k=0$, then $\Theta_k = 1$. 
\end{corollary}
Corollary \ref{cor.2nd} shows that if the gradient and Hessian approximations are accurate and the relaxation parameter is sufficiently large (relative to the noise in the function evaluations), and the TR radius is larger than a threshold, then the iteration is successful.

The following lemma is a direct consequence of the definitions of Oracles~\ref{oracle.first} and \ref{oracle.second} and the new definition of $I_k$ in \eqref{eq:indicator2}. 
 
\begin{lemma} \label{lem:true2} 
The random sequence $\{ M_k\}$ satisfies the submartingale-type condition
\begin{align}	\label{prob-delta-def2}
	\mathbbm{P}\{I_k = 1 | \mathcal{F}_{k-1}\} \geq p_1p_2 \stackrel{\rm def}{=} p_{12},
\end{align}
where $\mathcal{F}_{k-1}$ is defined in \eqref{eq:sig_alg1}. 
\end{lemma}

With the redefined random variables, Lemmas~\ref{lem:total_succ_unsucc} and \ref{lem:Azuma} hold for Algorithm~\ref{alg:tr2}. 
Lemma~\ref{lem:total_progress} also holds for Algorithm~\ref{alg:tr2} with the function $h$ defined in \eqref{eq.c_8}, and 
Lemma~\ref{lem:Azuma2} holds with 
\begin{equation} \label{eq:p0 2}
    p_0 = 1 - 2 \exp\left( \frac{a}{2} (2\epsilon_f + \epsilon_g^{3/2} - r)\right). 
\end{equation}
Furthermore, Lemma~\ref{lem:"Bernstein"} holds without any changes. 
Therefore, with the newly redefined random variables, $h$ and $p_0$, the two main lemmas that combine inequalities, Lemmas~\ref{lem:combine1bounded} and \ref{lem:combine1subexpo}, still hold.  
We can arrive at the following two theorems for both bounded and subexponential noise cases. 
Their proofs are analogous to those of Theorems~\ref{thm:convergence1_bounded} (with Lemma~\ref{lem:combine1bounded}) and \ref{thm:convergence1_subexponential} (with Lemma~\ref{lem:combine1subexpo}), respectively, and, thus, for brevity we provide a sketch of the proof for Theorem~\ref{thm:convergence2_bounded} and omit the others. 

\begin{theorem} \label{thm:convergence2_bounded}
Let Assumptions~\ref{ass:lip_cont2} and \ref{assum:low_bound} hold for the objective function $\phi$. 
Let Assumption~\ref{assum:bhm} and $r \ge 2\epsilon_f+\epsilon_g^{3/2}$ hold for Algorithm~\ref{alg:tr2}. 
Under Assumption~\ref{assum:upper_bound_e}, given any 
$\epsilon' > \sqrt[3]{\frac{4\epsilon_f+2r}{C_8 \gamma^3 (2p_{12}-1)}}$, where $C_8$ is defined in \eqref{eq.c_8}, 
the sequence of iterates generated by  Algorithm~\ref{alg:tr2} with Oracle \ref{oracle.zero}.\ref{ass:bounded noise} satisfies
\begin{equation} \label{eq:convergence2_bounded P}
    \mathbbm{P} \left\{ \min_{0\le k \le T-1}\beta(X_k) \le \epsilon' \right\} \ge 1 - \exp \left( - \frac{(1 - \hat{p}_{12}/p_{12})^2}{2} T \right) 
\end{equation}
for any $\hat{p}_{12} \in \left(\frac{1}{2} + \frac{2\epsilon_f+r}{C_{8}(\gamma \epsilon')^3}, p_{12} \right]$ and any 
\begin{equation} \label{eq:convergence2_bounded T}
T \ge \left(\hat{p}_{12} - \frac{1}{2} - \frac{2\epsilon_f+r}{C_8(\gamma \epsilon')^3} \right)^{-1} \left[\frac{\phi(x_0) - \hat\phi}{C_8 (\gamma \epsilon')^3} + \frac{1}{2}\log_\gamma \left( 
    \min\left\{ \frac{\epsilon'}{\delta_0}, \frac{\delta_0}{\beta(x_0)} \right\} 
\right) + \frac{1}{2}\right]. 
\end{equation}
\end{theorem}

\begin{proof}
Analogous to Lemma~\ref{lem:combine1bounded}, one can prove that for any positive integer $T$ and any $\hat{p}_1 \in [0,p_1]$  
\begin{equation} \label{eq:combine2bounded}  
    \mathbbm{P}\left\{ \left( \hat{p}_{12} - \frac{1}{2} - \frac{2\epsilon_f+r}{h(\gamma\bar\Delta)}\right)T < \frac{\phi(x_0)-\hat\phi}{h(\gamma\bar\Delta)} + \frac{1}{2} \left| \log_\gamma \frac{\bar\Delta}{\delta_0} \right| + \frac{1}{2} \right\} \ge 1 - \exp\left(-\frac{(1-\hat{p}_{12}/p_{12})}{2} T \right). 
\end{equation}
Now recall $h(\gamma\bar\Delta) = C_8 \min\{1,\gamma^3 \bar\Delta^3\} = C_8 \min\{1,\gamma^3 \min_{0\ge k \le T-1} \beta(X_k)\}^3$ and $0<\gamma<1$. Analogous to Theorem~\ref{thm:convergence1_bounded}, when $\bar\Delta \le \delta_0$, since the function $Q$ defined below is a decreasing function of $y$, 
\[  Q(y) = \left( \hat{p}_{12} - \frac{1}{2} - \frac{2\epsilon_f+r}{C_8\gamma^3y^3}\right)^{-1} 
\left[ \frac{\phi(x_0)-\hat\phi}{C_8\gamma^3y^3} + \frac{1}{2} \log_\gamma \frac{y}{\delta_0} + \frac{1}{2} \right], 
\] 
condition \eqref{eq:convergence2_bounded T} implies $T \ge Q(\epsilon')$, and the event in \eqref{eq:combine2bounded} is $T < Q(\min\{\gamma^{-1}, \min_{0\ge k \le T-1} \beta(X_k)\})$, we have $\min\{\gamma^{-1}, \min_{0\ge k \le T-1} \beta(X_k)\} < \epsilon'$. 
Considering $\epsilon'\leq 1 < \gamma^{-1}$, this means $\min_{0\ge k \le T_1} \beta(X_k) < \epsilon'$. 
If $\bar\Delta > \delta_0$, we replace the $0.5\log_\gamma(y/\delta_0)$ in function $Q$ with the constant $0.5\log_\gamma(\delta_0/\beta(x_0))$ and apply the same argument with the decreasing function. 
\end{proof}

\begin{remark}
    Condition $\epsilon' > \sqrt[3]{\frac{4\epsilon_f + 2r}{C_8 \gamma^3 (2p_0+2p_{12}-3)}}$ together with $r \ge 2\epsilon_f+\epsilon_g^{3/2}$
    implies a lower bound on $\epsilon'$ of ${\cal O}(\sqrt[3]{\epsilon_f}) + {\cal O}(\sqrt{\epsilon_g})$. 
    Then, Theorem~\ref{thm:convergence2_bounded} implies the best achievable accuracy $\epsilon$ for an approximate second-order critical point is of the form $\epsilon \ge {\cal O}(\sqrt[3]{\epsilon_f}) + {\cal O}(\sqrt{\epsilon_g})+ {\cal O}(\epsilon_H)$. 
\end{remark}

Additionally, the optimal value for $\hat{p}_{12}$ in Theorem~\ref{thm:convergence2_bounded} is 
\[ \frac{1}{2} + \frac{2\epsilon_f+r}{C_8(\gamma\epsilon')^3} + \frac{1}{T}  \left[\frac{\phi(x_0) - \hat\phi}{C_8 (\gamma\epsilon')^3} + \frac{1}{2}\log_\gamma \left( 
    \min\left\{ \frac{\epsilon'}{\delta_0}, \frac{\delta_0}{\beta(x_0)}\right\} 
\right)  + \frac{1}{2}\right], 
\]
and by setting $\hat{p}_{12}$ to this value, we derive the following corollary to Theorem~\ref{thm:convergence2_bounded}. 

\begin{corollary}
Under the setting of Theorem~\ref{thm:convergence2_bounded}, given any 
$\epsilon' > \sqrt[3]{\frac{4\epsilon_f+2r}{C_8 \gamma^3 (2p_{12}-1)}}$, where $C_8$ is defined in \eqref{eq.c_8}, it follows that 
\begin{align*} 
&\mathbbm{P} \left\{ \min_{0\le k \le T-1}\beta(X_k) \le \epsilon' \right\}  \\
\ge & 1 - \exp \left( - \frac{1}{2p_{12}^2 T} \left\{ \left(p_{12} - \frac{1}{2} - \frac{2\epsilon_f+r}{C_8(\gamma\epsilon')^3} \right) T - \left[ \frac{\phi(x_0) - \hat\phi}{C_8(\gamma\epsilon')^3} + \frac{1}{2}\log_\gamma \left( 
    \min\left\{ \frac{\epsilon'}{\delta_0}, \frac{\delta_0}{\beta(x_0)}\right\} 
\right)  + \frac{1}{2} \right] \right\}^2 \right)
\end{align*}
for any $T$ satisfying \eqref{eq:convergence2_bounded T} with the optimal $\hat{p}_{12}$. 
\end{corollary}

The final theorem is for the setting of Oracle \ref{oracle.zero}.\ref{ass:subexponential noise}. 

\begin{theorem} \label{thm:convergence2_subexponential}
Let Assumptions~\ref{ass:lip_cont2} and \ref{assum:low_bound} hold for the objective function $\phi$. 
Let Assumptions~\ref{assum:bhm} and $r \ge 2\epsilon_f+\epsilon_g^{3/2}$ holds for Algorithm~\ref{alg:tr2}. 
Let $p_0$ be defined as \eqref{eq:p0 2}. Given 
$\epsilon' > \sqrt[3]{\frac{4\epsilon_f + 8/a + 2r}{C_8 \gamma^3 (2p_0+2p_{12}-3)}}$, where $C_8$ is defined in \eqref{eq.c_8},
the sequence of iterates generated by Algorithm~\ref{alg:tr2} with Oracle \ref{oracle.zero}.\ref{ass:subexponential noise} satisfies 
\begin{equation} \label{eq:convergence2 subexponential P}
    \mathbbm{P} \left\{ \min_{0\le k \le T-1}\beta(X_k) \le \epsilon' \right\}  \ge 1 - \exp \left( - \frac{(1 - \hat{p}_{12}/p_{12})^2}{2} T \right) - \exp \left( - \frac{(1 - \hat{p}_0/p_0)^2}{2} T \right)  - \exp\left( - \frac{a}{4} t \right)
\end{equation}
for any $\hat{p}_0$ and $\hat{p}_{12}$ such that $\hat{p}_0 + \hat{p}_{12} \in \left(\frac{3}{2} + \frac{2\epsilon_f + 4/a + r}{C_8(\gamma\epsilon')^3}, p_0 + p_{12} \right]$, any $t \ge 0$, and any 
\begin{equation} \label{eq:convergence2 subexponential T}
T \ge \left(\hat{p}_0 + \hat{p}_{12} - \frac{3}{2} - \frac{2\epsilon_f + 4/a + r}{C_8(\gamma\epsilon')^3} \right)^{-1} \left[\frac{\phi(x_0) - \hat\phi + t}{C_8(\gamma\epsilon')^3} + \frac{1}{2}\log_\gamma \left( 
    \min\left\{ \frac{\epsilon'}{\delta_0}, \frac{\delta_0}{\beta(x_0)}\right\} 
\right)  + \frac{1}{2}\right]. 
\end{equation}
\end{theorem}

\section{Numerical experiments: Adversarial example} 
\label{sec:dual}
In this section, we explore the tightness of our theoretical results through numerical experiments. The goal is to investigate whether over the course of optimization the minimum gradient norm $\|\nabla\phi(x_k)\|$ encountered by Algorithm~\ref{alg:tr} in the presence of noise  is consistent with our theoretical lower bounds on $\epsilon$. Since our analysis is for the worst-case scenario,  we consider synthetic experiments where we injected noise  in an adversarial manner  to make the algorithm perform as poorly as possible at each step. 

First, we choose a simple function, $\phi(x) = L_1 \|x\|^2 /2$, that satisfies Assumptions~\ref{assum:lip_cont} and \ref{assum:low_bound}. We apply Algorithm~\ref{alg:tr} using linear models, thus,  Assumption~\ref{assum:bhm} is satisfied with $\kappa_{\rm bhm} = 0$. 
We do not consider quadratic models in this experiment because developing the adversarial numerical example becomes significantly more complex. 
The solution of the trust-region subproblem for linear models can be expressed as 
\begin{equation}
    s_k= \bcases
    0 &\text{if } g_k=0 \\
    -\frac{\delta_k}{\|g_k\|} g_k &\text{otherwise. }
    \ecases
\end{equation} 
When $s_k=g_k=0$, the step is rejected. In most of the discussion below, $g_k$ is assumed to be nonzero unless stated otherwise. 
Given the trial step $s_k$, the noise in the function value is set as follows to ``encourage'' the algorithm to reject good  steps (those that decrease $\phi$) 
and to accept bad steps (those that increase $\phi$). Specifically, the noise was set as follows, 
\begin{equation} \label{eq:adversarial e}
    (e_k, e_k^+) = \bcases
    (-\epsilon_f, +\epsilon_f) &\text{if } \phi(x_k+s_k) \le \phi(x_k) \\
    (+\epsilon_f, -\epsilon_f) &\text{otherwise.} 
    \ecases
\end{equation} 
The noise setting \eqref{eq:adversarial e} does not ensure the worst-case behavior over all iterations of the algorithm, since accepting a good step may increase the trust-region radius and result in worse performance later on. However, this setting reflects our theoretical analysis, which considers only the worse outcome of each iteration separately. 

To generate gradient approximations in an adversarial manner, we again aim for the algorithm to accept steps that make $\phi$ increase as much as possible and avoid taking  steps when increasing $\phi$ is not possible. 
Since our Oracle~\ref{oracle.first} only offers sufficiently accurate gradient with probability $p_1$, at the beginning of each iteration, we set a random variable $I_k$ to 1 with probability $p_1$ and 0 with probability $1-p_1$. 
During iterations where the gradient is not sufficiently accurate ($I_k=0$), $g_k$ is chosen so that the increase in function value $\phi(x_k+s_k) - \phi(x_k)= -L_1 \delta_k \langle x_k, g_k / \|g_k\| \rangle + L_1\delta_k^2/2$ is maximized under condition that the step is accepted, i.e., $\rho_k \geq \eta_1$. Thus, the following problem is solved.  
\begin{equation} \label{prob:most loss I=0} 
 \begin{aligned}
    &\max_{g_k} &&\phi(x_k+s_k) - \phi(x_k)= -L_1\delta_k \langle x_k, g_k / \|g_k\| \rangle + L_1\delta_k^2/2 \;\; &&\text{Maximize loss.}\\
    &\text{s.t.} &&\rho_k = \frac{L_1\|x_k\|^2/2 - L_1\|x_k-\delta_k g_k / \|g_k\|\|^2/2 + 2\epsilon_f+r}{\|g_k\| \delta_k} \ge \eta_1 \;\; &&\text{Step accepted.} \\
    &~ &&\|g_k\| > 0.  && 
\end{aligned} 
\end{equation}

If \eqref{prob:most loss I=0} is infeasible or its optimal value is less than zero, then $g_k$ is set to zero so that the step is rejected, otherwise $g_k$ is set to its optimal solution. 
We note that expressions for $\rho_k$ depends on how we use \eqref{eq:adversarial e}. Since the goal here is to accept a step with  possible increase in $\phi$, we then set $(e_k, e_k^+) =  (+\epsilon_f, -\epsilon_f) $ and expression for $\rho_k$ is 
\begin{equation} \label{prob:rho} 
\rho_k = \frac{\phi(x_k) - \phi(x_k+s_k) + 2\epsilon_f+r}{m(x_k) - m(x_k+s_k)} 
= \frac{L_1\|x_k\|^2/2 - L_1\|x_k-\delta_k g_k / \|g_k\|\|^2/2 + 2\epsilon_f+r}{\|g_k\| \delta_k}. 
\end{equation}

On the other hand, if $I_k = 1$, problem \eqref{prob:most loss I=0} needs to be solved with the additional constraint that the gradient is sufficiently accurate, i.e.,
\begin{equation}\label{eq:add_constr}
 \begin{aligned}
    &~ &&\|g_k - \nabla\phi(x_k)\| = \|g_k - L_1x_k\| \le \kappa_{\rm eg} \delta_k + \epsilon_g \;\; &&\text{Gradient sufficiently accurate.}
\end{aligned} 
\end{equation}

Either with the additional constraint \eqref{eq:add_constr} or not, problem \eqref{prob:most loss I=0} is not solved as  is, but with the change of variables $y_1 = \langle x_k, g_k/\|g_k\| \rangle$ and $y_2 = \|g_k\|$. 
The reformulation is 
\begin{equation} \label{prob:most loss} 
 \begin{aligned}
    &\min_{y_1,y_2} &&y_1 \;\; &&\text{Maximize loss.}\\
    &\text{s.t.} &&\eta_1y_2 - L_1 y_1 \le (2\epsilon_f + r)/\delta_k - L_1\delta_k/2 \;\; &&\text{Step accepted.} \\
    &~ &&|y_1| \le \|x_k\| \text{ and } y_2 \ge \min\{10^{-6}, 10^{-2}\|L_1x_k\|\},  &&
\end{aligned} 
\end{equation}
where $y_2 = \|g_k\| > 0$ is replaced by $y_2 \ge \min\{10^{-6}, 10^{-2}\|L_1x_k\|\}$ so that the ``minimization'' is well-defined. 
Constraint \eqref{eq:add_constr} is then written as 
\begin{equation} \label{eq:add_constr_y} 
 \begin{aligned}
  &~ &&y_2^2 - 2L_1y_1y_2 + (L_1\|x_k\|)^2 \le (\kappa_{\rm eg} \delta_k + \epsilon_g)^2 \;\; &&\text{Gradient sufficiently accurate.} \\
\end{aligned} 
\end{equation}

Problem \eqref{prob:most loss} (with or without \eqref{eq:add_constr_y}) can be solved analytically, and the  corresponding optimal $g_k$  can be recovered from the optimal $y_1$ and $y_2$. 
The procedures are detailed in Appendix~\ref{app:dual details}. 

When $I_k=1$, if the optimal value of \eqref{prob:most loss}-\eqref{eq:add_constr_y} is less than $\delta_k/2$ (meaning the loss is greater than 0), $g_k$ is set using the optimal values of $y_1$ and $y_2$; 
if this problem is infeasible, then $g_k$ is simply set to $\nabla\phi(x_k)=L_1 x_k$, since an acceptable step does not exist anyway; and, 
if the optimal value of this problem is greater than or equal to $\delta_k/2$, the algorithm cannot be tricked into taking a step that increases $\phi$, so the worst-case scenario is when no step is taken at all. 
In the third case, we try to find a value for $g_k$ to prevent any step from being taken by solving the two optimization problems \eqref{prob:reject1} and \eqref{prob:reject2} described below:
\begin{equation} \label{prob:reject1}  \begin{aligned}
    &\max_{y_1,y_2} &&\eta_1y_2 - L_1 y_1 \;\; &&\text{Try to get the step rejected.} \\
    &\text{s.t.} &&y_2^2 - 2L_1y_1y_2 + (L_1\|x_k\|)^2 \le (\kappa_{\rm eg} \delta_k + \epsilon_g)^2 \;\; &&\text{Gradient  sufficiently accurate.} \\
    &~ &&y_1 < \delta_k/2  \;\; &&\text{Step leads to loss of progress. }
     \\
    &~ &&|y_1| \le \|x_k\| \text{ and } y_2 \ge \min\{10^{-6}, 10^{-2}\|L_1x_k\|\}  &&
\end{aligned} 
\end{equation}
\begin{equation} \label{prob:reject2}  \begin{aligned}
    &\max_{y_1,y_2} &&\eta_1y_2 - L_1 y_1 \;\; &&\text{Try to get the step rejected} \\
    &\text{s.t.} &&y_2^2 - 2L_1y_1y_2 + (L_1\|x_k\|)^2 \le (\kappa_{\rm eg} \delta_k + \epsilon_g)^2 \;\; &&\text{Gradient sufficiently accurate.} \\
    &~ &&y_1 \ge \delta_k/2  \;\;&&\text{Step leads to progress.} 
       \\
    &~ &&|y_1| \le \|x_k\| \text{ and } y_2 \ge \min\{10^{-6}, 10^{-2}\|L_1x_k\|\}.  &&
\end{aligned} 
\end{equation}
There are two problems because of \eqref{eq:adversarial e}. 
In the first problem where the second constraint is $\phi(x_k+s_k) > \phi(x_k)$, $\rho_k$ \eqref{prob:rho} is calculated with a $+2\epsilon_f$ term in the numerator; and, in the second problem where the second constraint is $\phi(x_k+s_k) \le \phi(x_k)$, $\rho_k$ \eqref{prob:rho} is calculated with a $-2\epsilon_f$ term in the numerator. 
If either one of the two optimal values is greater than $(\pm2\epsilon_f+r)/\delta_k - L_1 \delta_k/2$, then we set $g_k$ to the corresponding optimal solution, which will lead to a rejection of the step;
otherwise, a step that decreases $\phi$ would be unavoidable, so we try to inject  noise that can minimize the decrease in $\phi$ by solving the following problem:
\begin{equation} \label{prob:least gain} 
 \begin{aligned}
    &\min_{y_1,y_2} &&y_1 \;\; &&\text{Minimize gain.}\\
    &\text{s.t.} &&y_2^2 - 2L_1y_1y_2 + (L_1\|x_k\|)^2 \le (\kappa_{\rm eg} \delta_k + \epsilon_g)^2 \;\; &&\text{Gradient  sufficiently accurate.} \\
    &~ &&|y_1| \le \|x_k\| \text{ and } y_2 \ge \min\{10^{-6}, 10^{-2}\|L_1x_k\|\}.  &&
\end{aligned} 
\end{equation}

Similar to \eqref{prob:most loss}, the optimization problems  \eqref{prob:reject1},  \eqref{prob:reject2}, \eqref{prob:least gain}, can be solved analytically. 
See  Appendix~\ref{app:dual details} for details. 

For our experiments, we set the parameters of the objective function to $n=20, L_1=1$;
the parameters for the approximation models to $p_1=0.8, \kappa_{\rm eg}=1$;
and the parameters for the algorithm to $\eta_1=0.25, \eta_2=1$,  $\gamma=0.8$, and $r=2\epsilon_f$. 
Then, the theoretical lower bound on $\epsilon$ in Theorem~\ref{thm:convergence1_bounded} is $5\sqrt{30\epsilon_f} + 7/3\epsilon_g \approx 27.39\sqrt{\epsilon_f} + 2.33\epsilon_g$. 
A minor detail here in calculating this theoretical value is that $\kappa_{\rm fcd}$ is set to 2 because the model decrease is $\|g_k\|\delta_k$, even though we assumed $\kappa_{\rm fcd}\in(0,1]$. This is not an issue because the property $\kappa_{\rm fcd} \le 1$ was never used in the analysis. 

We experiment with four noise settings where $(\epsilon_f, \epsilon_g)$ is set to  $(0.2,4), (0,4), (0.2,0)$ and $(0,0)$, respectively. 
We initialize Algorithm~\ref{alg:tr} at $x_0=1.4 \cdot \mathbf{1}$ and with $\delta_0 = 0.5$, and inject adversarial noise as described above at each iteration. 
Figure~\ref{fig:adversarial} shows how $\|\nabla\phi(x_k)\|$ and $\delta_k$ change over the first 250 iterations. 
We note that $\|\nabla\phi(x_k)\|$ stabilizes around 4.8, 4, 1.2, and 0, respectively, for the four noise settings, and executing the same experiment multiple times yields similar results. 
In comparison, their theoretical lower bounds on $\epsilon$ are $21.58(=12.25+9.33 \approx 5\sqrt{30\epsilon_f} + 7\epsilon_g/3)$, $9.33$, $12.25$, and 0, respectively. 
This indicates in the lower bound on $\epsilon$ in Theorem~\ref{thm:convergence1_bounded} the coefficient of $\epsilon_g$ is at most 7/3 times its optimal value, but the coefficient of $\sqrt{\epsilon_f}$ can be up to 10 times as big. 
This indicates the theoretical lower bound on $\epsilon$ is not unreasonably large. 

\begin{figure}[htp]
    \centering
\begin{subfigure}[b]{0.5\linewidth}
    \includegraphics[width=\linewidth]{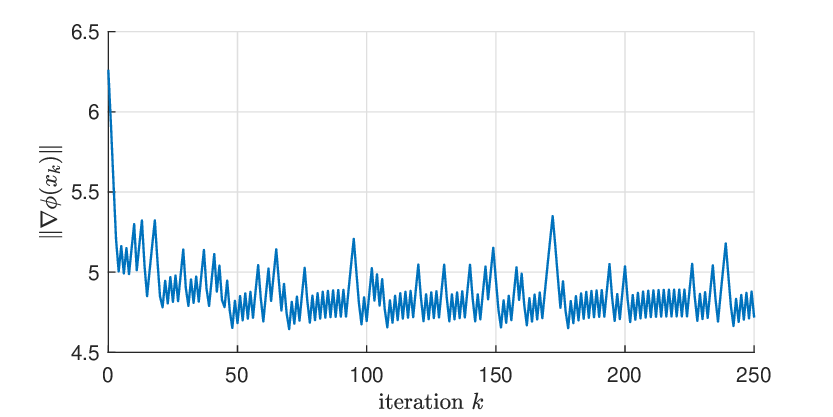}
    \caption{Change of $\|\nabla\phi(x_k)\|$ over $k$ when $\epsilon_f = 0.2$ and $\epsilon_g = 4$}
\end{subfigure}%
\begin{subfigure}[b]{0.5\linewidth}
    \centering
    \includegraphics[width=\linewidth]{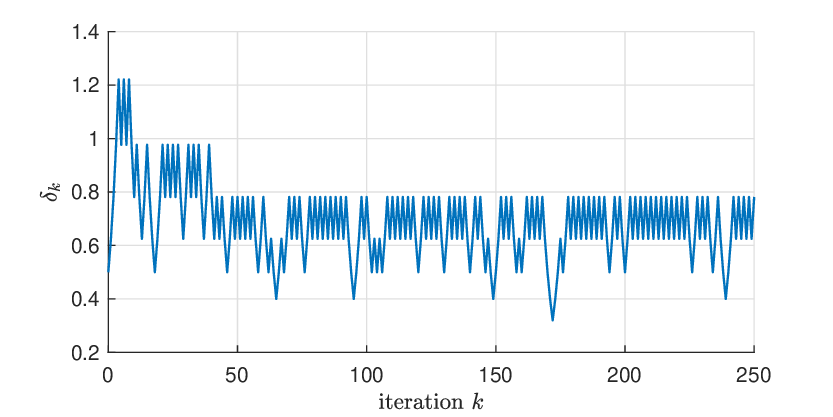}
    \caption{Change of $\delta_k$ over $k$ when $\epsilon_f =0.2$ and $\epsilon_g = 4$}
\end{subfigure}\\
\begin{subfigure}[b]{0.5\linewidth}
    \centering
    \includegraphics[width=\linewidth]{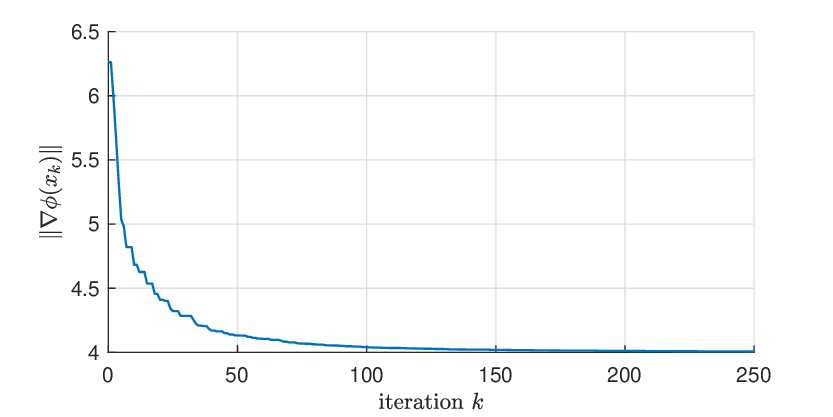}
    \caption{Change of $\|\nabla\phi(x_k)\|$ over $k$ when $\epsilon_f = 0$ and $\epsilon_g = 4$}
\end{subfigure}%
\begin{subfigure}[b]{0.5\linewidth}
    \centering
    \includegraphics[width=\linewidth]{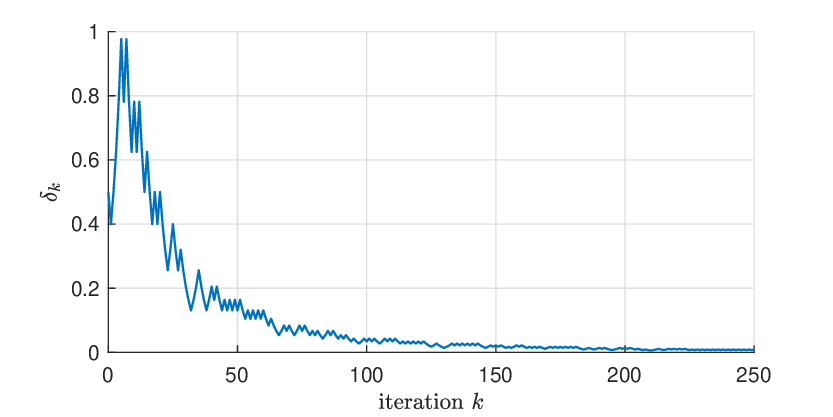}
    \caption{Change of $\delta_k$ over $k$ when $\epsilon_f = 0$ and $\epsilon_g = 4$}
\end{subfigure}\\
\begin{subfigure}[b]{0.5\linewidth}
    \centering
    \includegraphics[width=\linewidth]{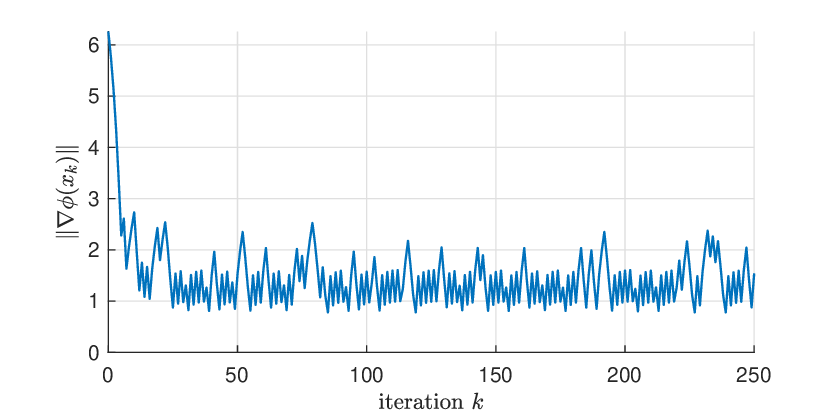}
    \caption{Change of $\|\nabla\phi(x_k)\|$ over $k$ when $\epsilon_f = 0.2$ and $\epsilon_g = 0$}
\end{subfigure}%
\begin{subfigure}[b]{0.5\linewidth}
    \centering
    \includegraphics[width=\linewidth]{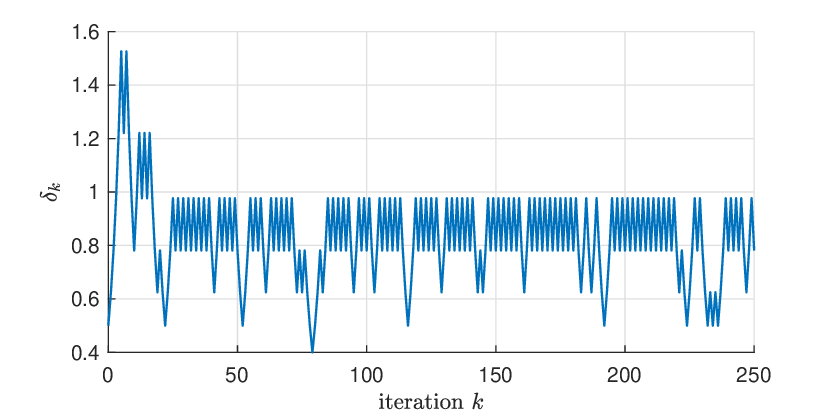}
    \caption{Change of $\delta_k$ over $k$ when $\epsilon_f = 0.2$ and $\epsilon_g = 0$}
\end{subfigure}\\
\begin{subfigure}[b]{0.5\linewidth}
    \centering
    \includegraphics[width=\linewidth]{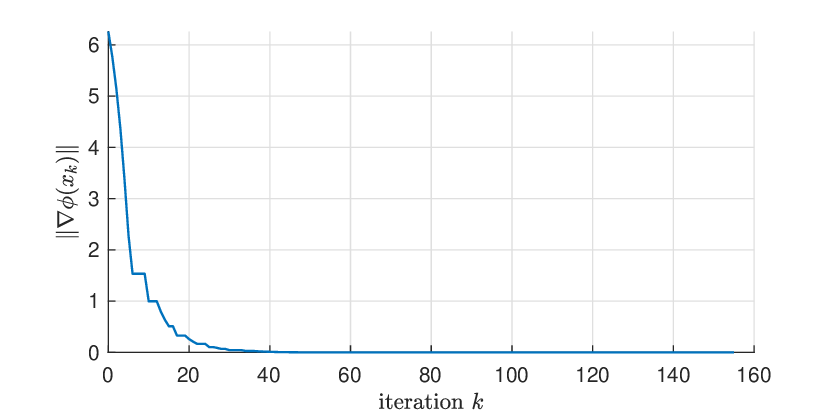}
    \caption{Change of $\|\nabla\phi(x_k)\|$ over $k$ when $\epsilon_f = \epsilon_g = 0$}
\end{subfigure}%
\begin{subfigure}[b]{0.5\linewidth}
    \centering
    \includegraphics[width=\linewidth]{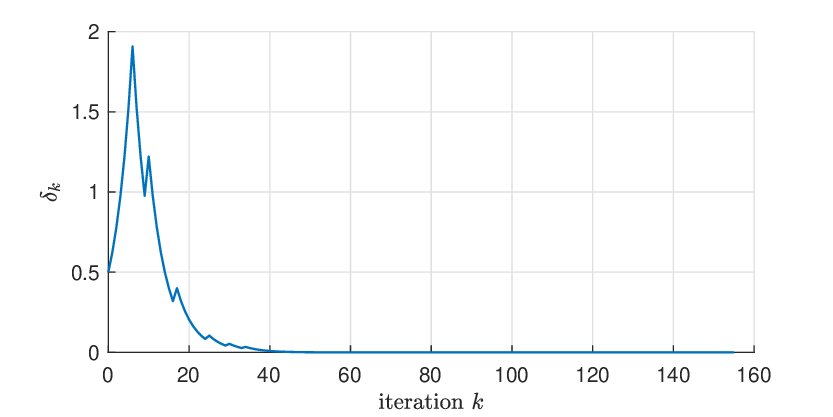}
    \caption{Change of $\delta_k$ over $k$ when $\epsilon_f = \epsilon_g = 0$}
\end{subfigure}
\caption{Performance of Algorithm~\ref{alg:tr} with linear approximation models on $\phi(x) = L_1\|x\|^2/2$ under adversarial noise when $r=2\epsilon_f$.}
    \label{fig:adversarial}
\end{figure}

\section{Numerical experiments: Investigating the effect of $r$}
\label{sec:r}
In this section, we explore numerically the effect of the choice of the hyperparameter $r$ on the performance of the algorithm. Our theory requires that $r \ge 2\epsilon_f$ in Algorithm~\ref{alg:tr} and $r \ge 2\epsilon_f + \epsilon_g^{3/2}$ in Algorithm~\ref{alg:tr2} to offset the function evaluation errors at $x_k$ and $x_k+s_k$. 
If $r$ is set smaller than these values, it is possible, under our particular assumptions on the zeroth-order oracle,  that 
the algorithm will fail to make successful steps due to noise in the function evaluations, even if the gradient  $\|\nabla \phi(\cdot)\|$ is not small. Setting $r$ to be larger than $2\epsilon_f$ or $2\epsilon_f + \epsilon_g^{3/2}$ allows the algorithm to progress until $\|\nabla \phi(\cdot)\|$ reaches the lower bound $\epsilon$, whose value is monotonically increasing with $r$, as can be seen in both Theorem~\ref{thm:convergence1_bounded} and Theorem~\ref{thm:convergence2_bounded}. In other words, the larger the value of $r$ the larger is the best achievable accuracy  $\epsilon$ and the complexity bound on $T$. 
Thus, when Oracle~\ref{oracle.zero}.\ref{ass:bounded noise} is used, it is clearly optimal to set $r=2\epsilon_f$ in Algorithm~\ref{alg:tr} and $r = 2\epsilon_f + \epsilon_g^{3/2}$ in Algorithm~\ref{alg:tr2}.  
However, as $\epsilon_f$ (and $\epsilon_g$) may not be known in practice, here we explore the effect of setting $r$ to a variety of different values with respect to $\epsilon_f$.

In the first set of experiments, we used the same setting as  described in Section~\ref{sec:dual}, with $\epsilon_f=0.2$ and $\epsilon_g=4$ but with $r$ set to different values. 
Figure~\ref{fig:adversarial r0148}, along with  the first line of Figure~\ref{fig:adversarial}, shows the change of $\|\nabla\phi(x_k)\|$ and $\delta_k$ over the iterations when $r = 0, \epsilon_f, 2\epsilon_f, 4\epsilon_f$, and $8\epsilon_f$. 
For experiments with $r\ge 2\epsilon_f$, the level at which $\|\nabla\phi(x_k)\|$ stabilizes get larger with larger values of $r$. 
This phenomenon is already suggested by the theory and is expected. 
However, while the theory suggests that the  algorithm may get ``stuck" if $r < 2\epsilon_f$, this  did not occur in our experiments. 
Instead, we observe the following behavior when $r=0$: in the initial stages of the optimization the algorithm  makes consistent progress because both $\|\nabla\phi(x_k)\|$ and $\delta_k$ are large enough to overcome the noise. 
Moreover, setting $r=0$ prevents increases in $\phi$. 
However, as $\|\nabla\phi(x_k)\|$ decreases, the gradient estimate and the decrease in function value $f(x_k) - f(x_k+s_k)$ become more dominated by noise. 
Without any relaxation in the step acceptance criterion, the noise may cause many successive rejected steps, which would also shrink the trust-region. 
As a result, as $\delta_k$ decreases, function evaluation noise becomes more dominated in $f(x_k) - f(x_k+s_k)$, while the predicted decrease $\|g_k\|\delta_k$ becomes smaller. 
The resulting effect is that it becomes easier for the adversarial noise to be set so that steps for which $\phi$ actually increases get accepted, which explains the monotonic increase of $\phi$ (and $\|\nabla\phi\|$) in the later stage of the experiment. 
As $r$ increases, this effect becomes less prominent and did not appear in the case where $r=\epsilon_f$.

\begin{figure}[htp]
    \centering
\begin{subfigure}[b]{0.5\linewidth}
    \includegraphics[width=\linewidth]{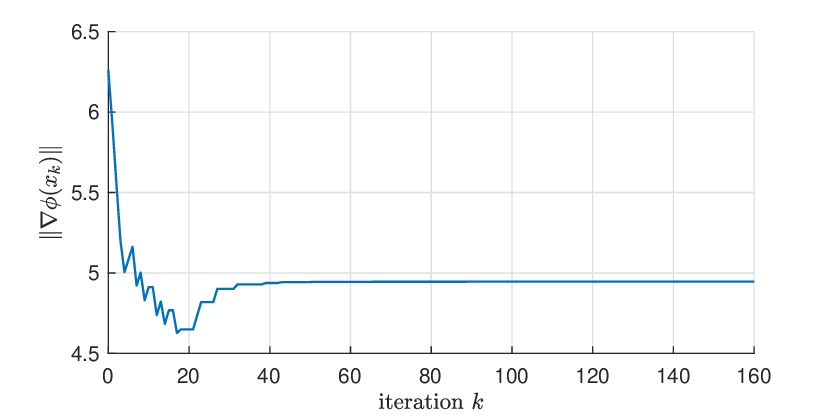}
    \caption{Change of $\|\nabla\phi(x_k)\|$ over $k$ when $r=0$}
\end{subfigure}%
\begin{subfigure}[b]{0.5\linewidth}
    \centering
    \includegraphics[width=\linewidth]{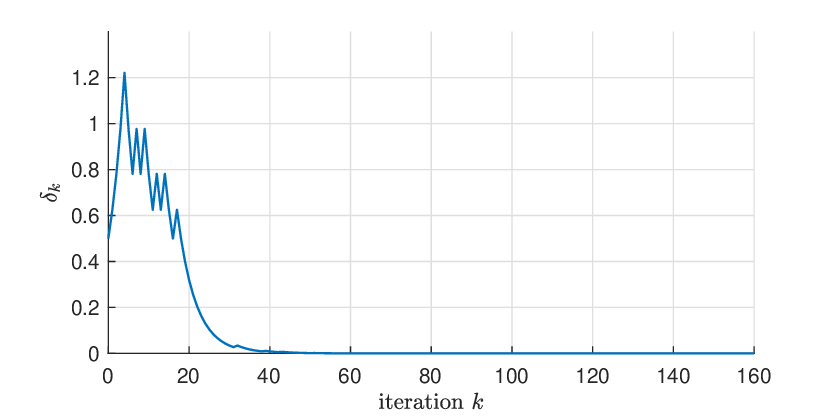}
    \caption{Change of $\delta_k$ over $k$ when $r=0$}
\end{subfigure} \\
\begin{subfigure}[b]{0.5\linewidth}
    \centering
    \includegraphics[width=\linewidth]{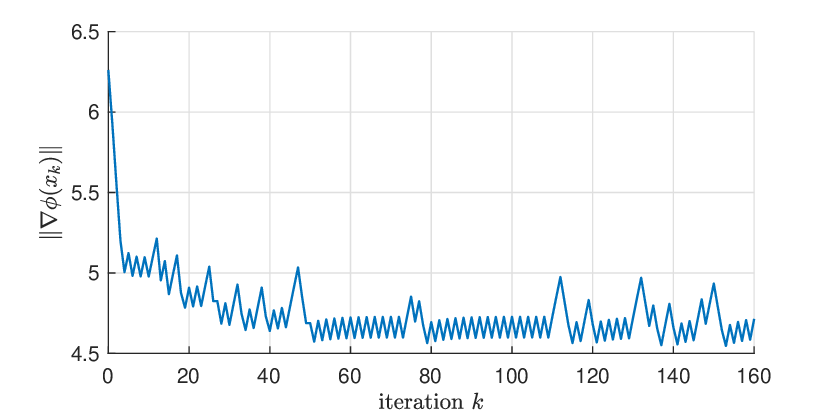}
    \caption{Change of $\|\nabla\phi(x_k)\|$ over $k$ when $r=\epsilon_f$}
\end{subfigure}%
\begin{subfigure}[b]{0.5\linewidth}
    \centering
    \includegraphics[width=\linewidth]{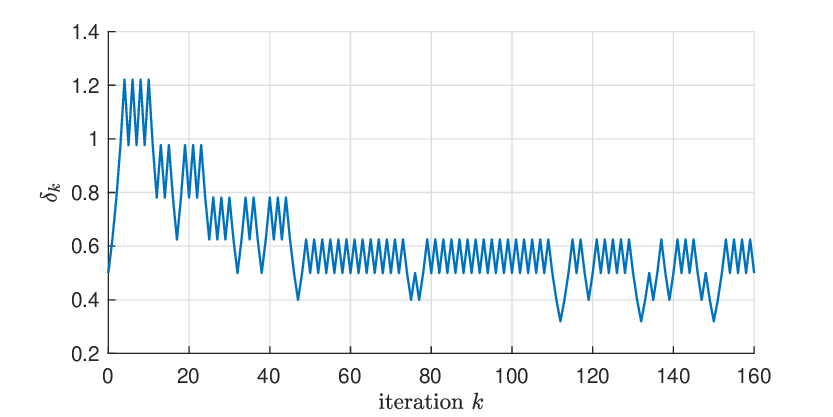}
    \caption{Change of $\delta_k$ over $k$ when $r=\epsilon_f$}
\end{subfigure}\\
\begin{subfigure}[b]{0.5\linewidth}
    \centering
    \includegraphics[width=\linewidth]{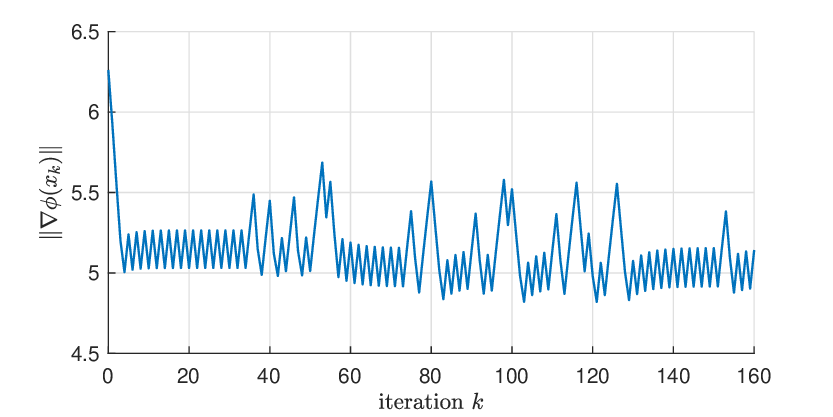}
    \caption{Change of $\|\nabla\phi(x_k)\|$ over $k$ when $r=4\epsilon_f$}
\end{subfigure}%
\begin{subfigure}[b]{0.5\linewidth}
    \centering
    \includegraphics[width=\linewidth]{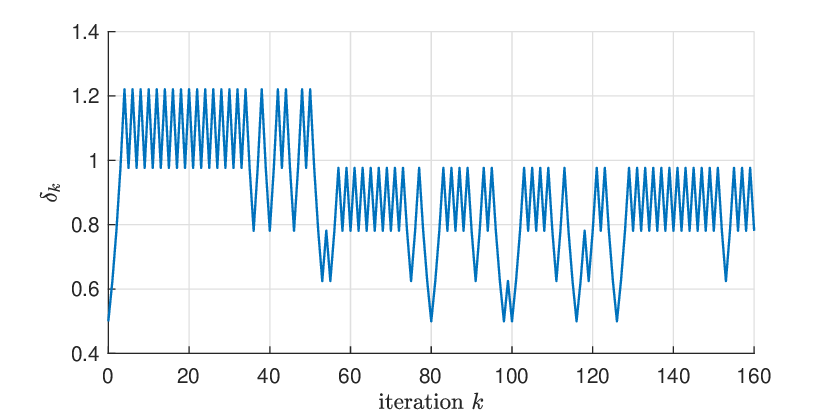}
    \caption{Change of $\delta_k$ over $k$ when $r=4\epsilon_f$}
\end{subfigure} \\ 
\begin{subfigure}[b]{0.5\linewidth}
    \centering
    \includegraphics[width=\linewidth]{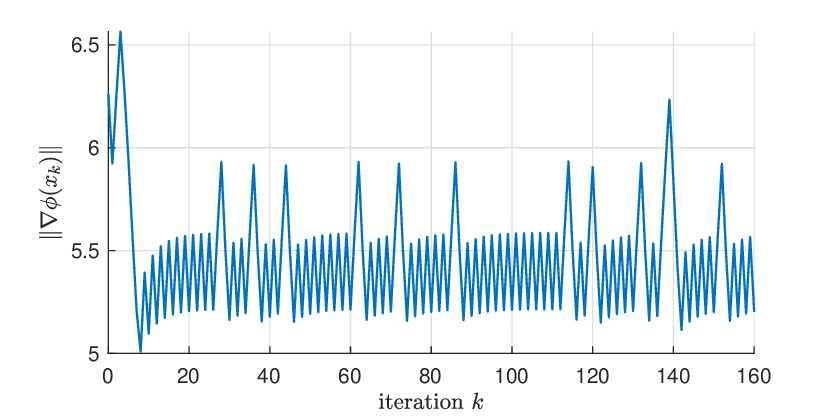}
    \caption{Change of $\|\nabla\phi(x_k)\|$ over $k$ when $r=8\epsilon_f$}
\end{subfigure}%
\begin{subfigure}[b]{0.5\linewidth}
    \centering
    \includegraphics[width=\linewidth]{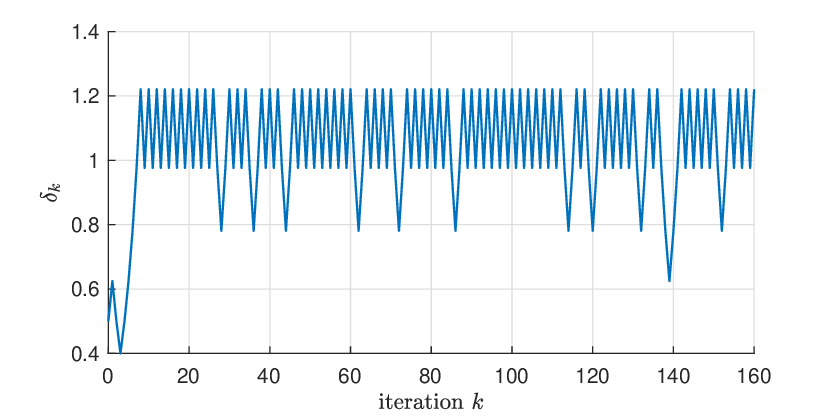}
    \caption{Change of $\delta_k$ over $k$ when $r=8\epsilon_f$}
\end{subfigure}
    \caption{Performance of Algorithm~\ref{alg:tr} with linear approximation models on $\phi(x) = L_1\|x\|^2/2$ under adversarial noise when $r$ is set to various values and $(\epsilon_f, \epsilon_g) = (0.2,4)$.}
    \label{fig:adversarial r0148}
\end{figure}

In the second set of experiments, we examine how the relaxation $r$ affects a practical derivative-free algorithm known as DFO-TR \cite{DFOTRpaper}, which does not always abide by the theory in this paper.  
As a practical algorithm, DFO-TR is different from Algorithms~\ref{alg:tr} and \ref{alg:tr2} and contains many small practical enhancements to improve the numerical performance, but in its essence, employs quadratic interpolation approximation and trust-region method. 
Most importantly, it calculates $\rho_k$ just like Algorithms~\ref{alg:tr} and \ref{alg:tr2} except without the relaxation. 
We add $r$ to the numerator of $\rho_k$ and see how DFO-TR performs with different values for $r$. 

The experiment is conducted on the Mor\'{e} \& Wild benchmarking problem set \cite{more2009benchmarking}. 
We first give DFO-TR infinite budget to solve all the problems and record the best solution for each problem as $\hat\phi$. 
Then the output of the 53 problems are scaled linearly so that $\phi(x_0)=100$ and $\hat\phi=0$ for every problem. 
We replicate each problem 4 times (leading to a total of $53\times5=265$ problems) and then artificially inject noise uniformly distributed on the interval $[-\epsilon_f,+\epsilon_f]$ with $\epsilon_f=0.2$ to function evaluations. 
With the function outputs scaled and the noise added, the problems are solved by DFO-TR subsequently with $r = 0, \epsilon_f, 2\epsilon_f, 4\epsilon_f$, and $8\epsilon_f$. 
Each variant of the algorithm is given a 2000 function evaluation budget for each problem. 
The hyperparameters were tuned to make sure the best of the solutions found by all variants of DFO-TR has a scaled function value close enough to 0. 
The results are compared and presented in performance and data profiles with the (relative) accuracy level $\tau$ set to $10^{-3}$ and $10^{-5}$ (see  \cite{more2009benchmarking} for the detail on how these profiles are created). 
As Figure~\ref{fig:MW uniform} shows, DFO-TR encounters difficulty if $r < 2\epsilon_f$, especially when trying to solve the problems to higher accuracy. 
Having an $r$ too large also affects the performance adversely. 
The best performance comes from the variant with $r=4\epsilon_f$. Considering that DFO-TR uses Hessian approximation and thus resembles Algorithm~\ref{alg:tr2}, whose theory suggests that the optimal $r$ needs to be larger than $2\epsilon_f$, this result is consistent with our theory. 

\begin{figure}[H]
\centering
\begin{subfigure}[b]{0.42\linewidth}
    \includegraphics[width=\linewidth]{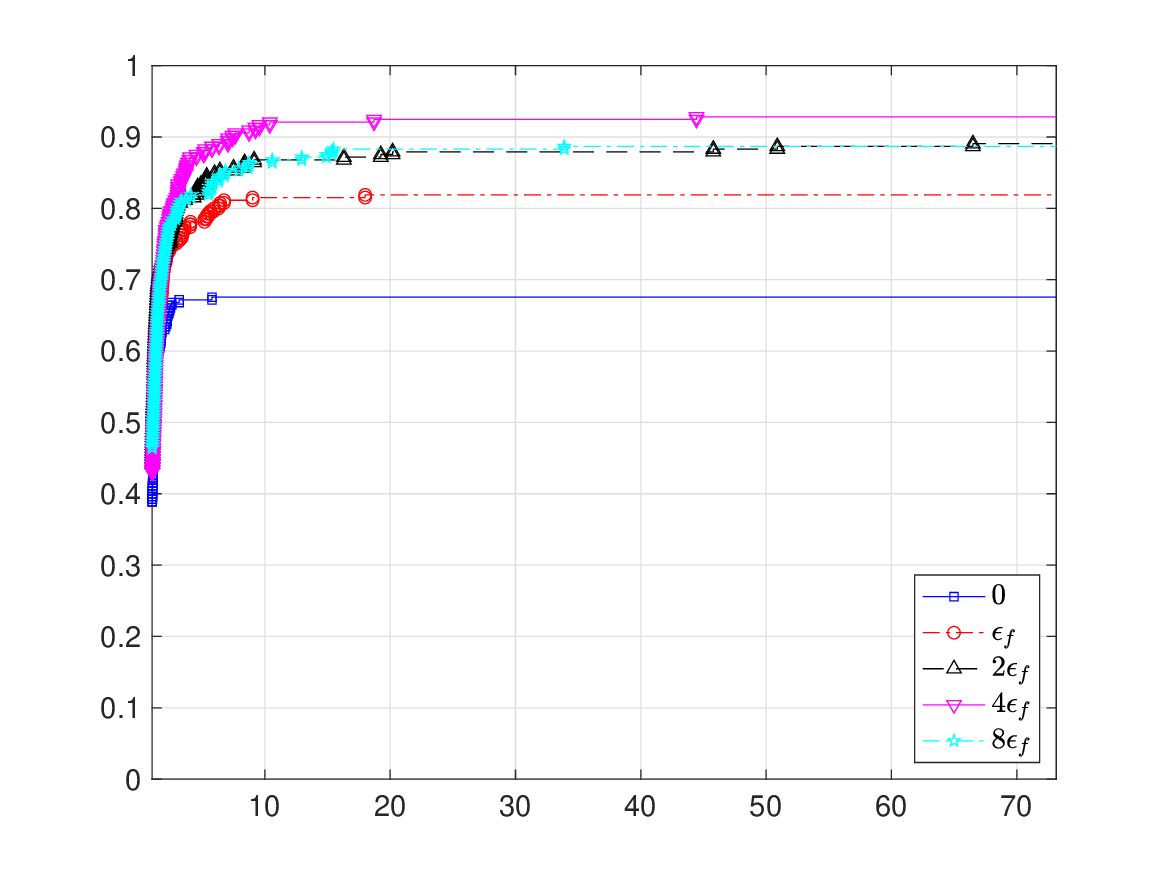}
    \caption{Performance profile with $\tau=10^{-3}$}
\end{subfigure}%
\begin{subfigure}[b]{0.42\linewidth}
    \centering
    \includegraphics[width=\linewidth]{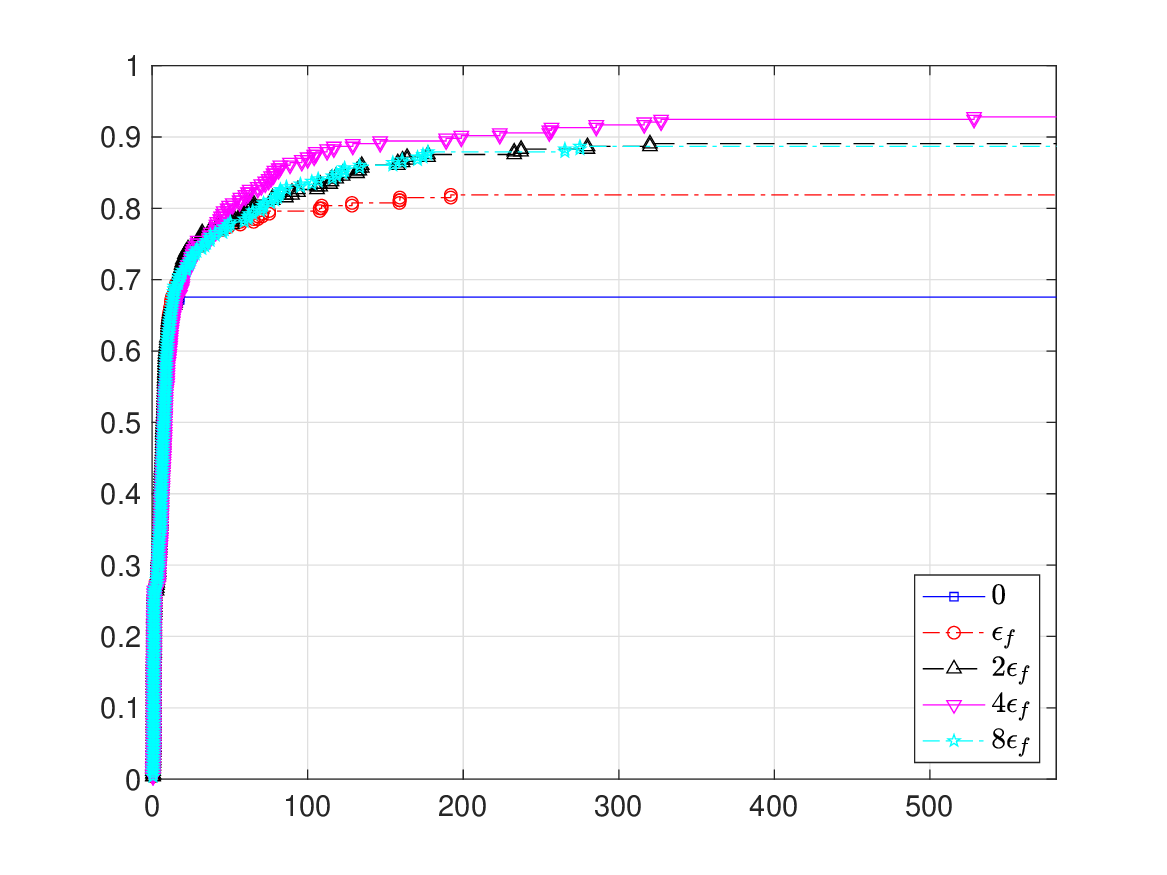}
    \caption{Data profile with $\tau=10^{-3}$}
\end{subfigure}\\
\begin{subfigure}[b]{0.42\linewidth}
    \centering
    \includegraphics[width=\linewidth]{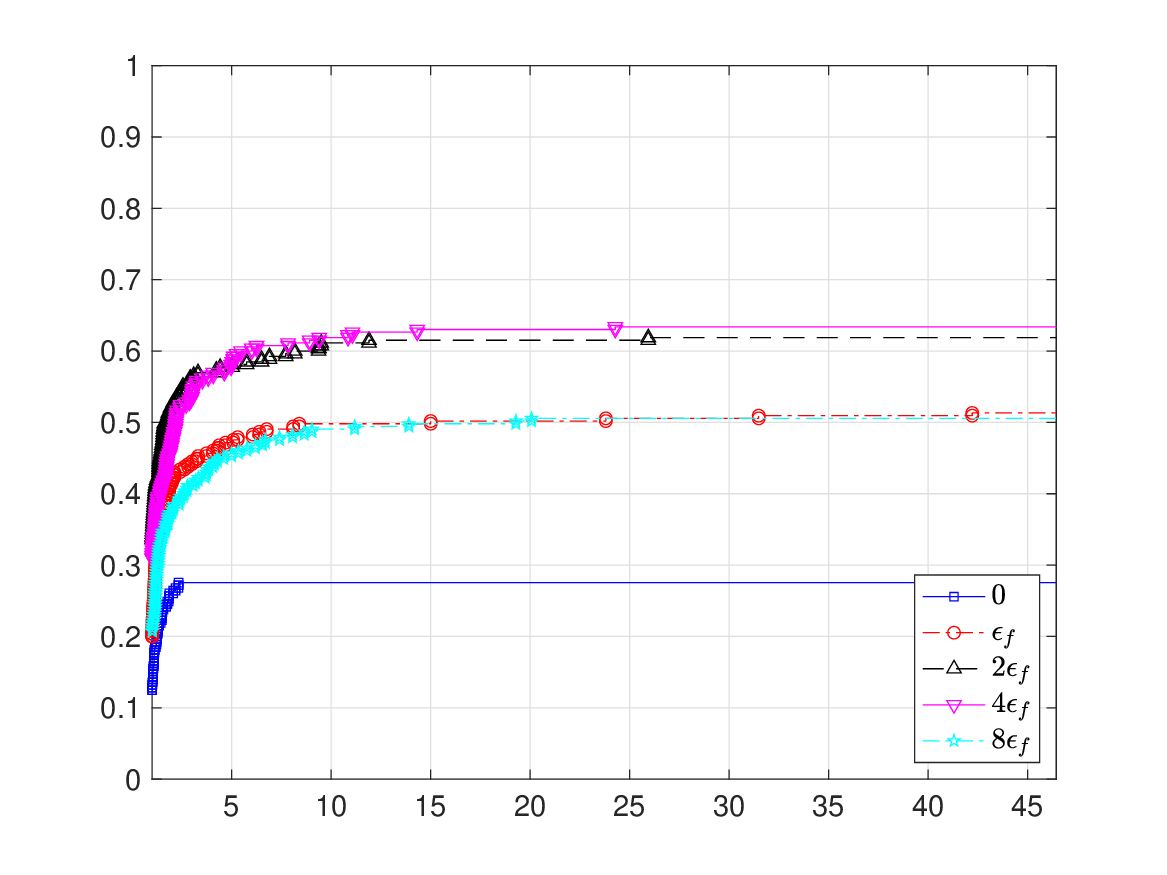}
    \caption{Performance profile with $\tau=10^{-5}$}
\end{subfigure}%
\begin{subfigure}[b]{0.42\linewidth}
    \centering
    \includegraphics[width=\linewidth]{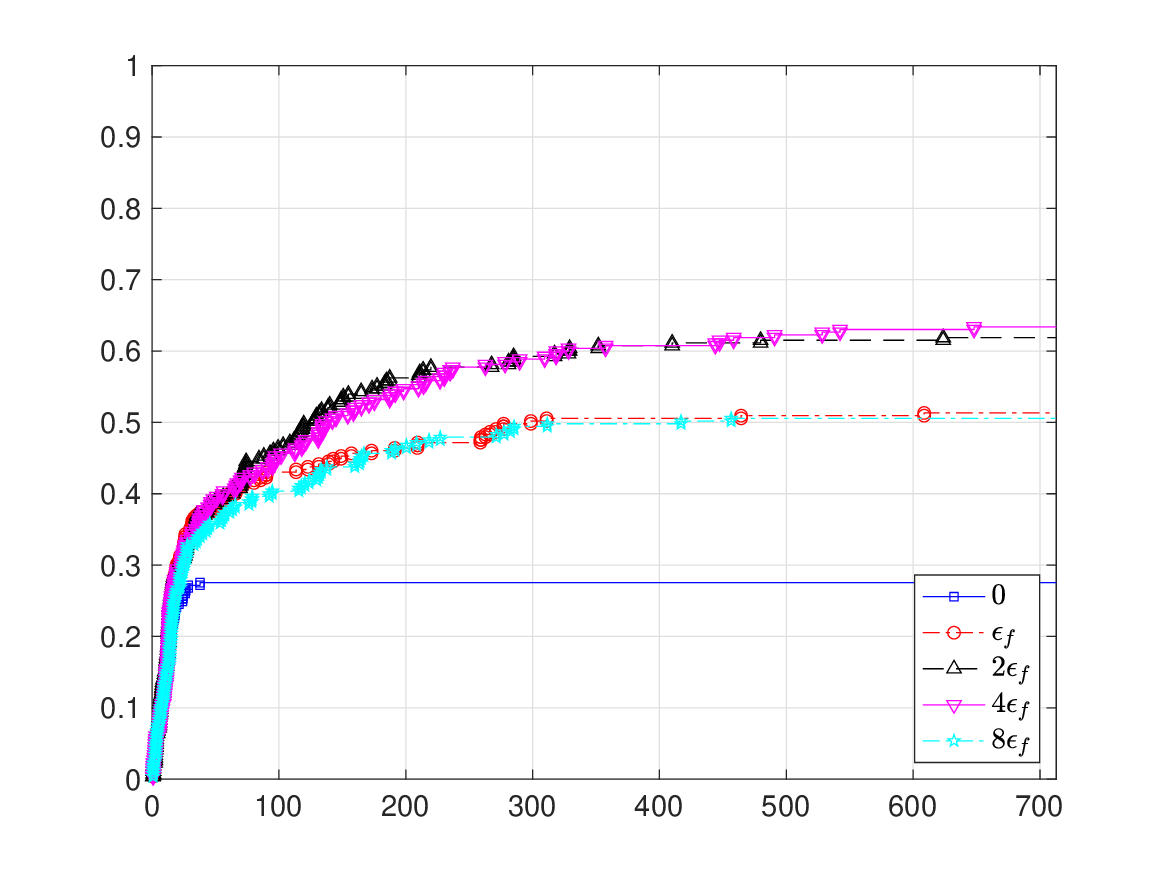}
    \caption{Data profile with $\tau=10^{-5}$}
\end{subfigure}\\
    \caption{Performance of DFO-TR with different relaxed step acceptance criteria ($r = 0, \epsilon_f, 2\epsilon_f, 4\epsilon_f$ and $8\epsilon_f$) on the Mor\'{e} \& Wild problem set with uniformly distributed function evaluation noise.}
    \label{fig:MW uniform}
\end{figure}

Finally, we repeated the  experiment with the uniformly distributed noise replaced by unbounded subexponential noise (Oracle~\ref{oracle.zero}.\ref{ass:subexponential noise}). 
Specifically, each time an objective function is evaluated, two random variables are generated - one uniformly distributed on $[0, \epsilon_f]$ and the other  exponentially distributed with parameter $a$. 
The sum of the two random variables is then multiplied by $-1$ with probability $0.5$ and then added to the true function value. 
We chose  $\epsilon_f=0.1$ and $a=20$ so that the expected magnitude of the noise is $0.1$. 
We test five levels for $r=$ 0, 0.2, 0.5, 1.25, and 3.125. 
The results are presented in Figure~\ref{fig:MW subexponential}. We see that in this case setting $r$ to $0.5$ or $1.25$ is advantageous for the algorithm, which is also consistent with our theory.

\begin{figure}[ht]
    \centering
\begin{subfigure}[b]{0.42\linewidth}
    \includegraphics[width=\linewidth]{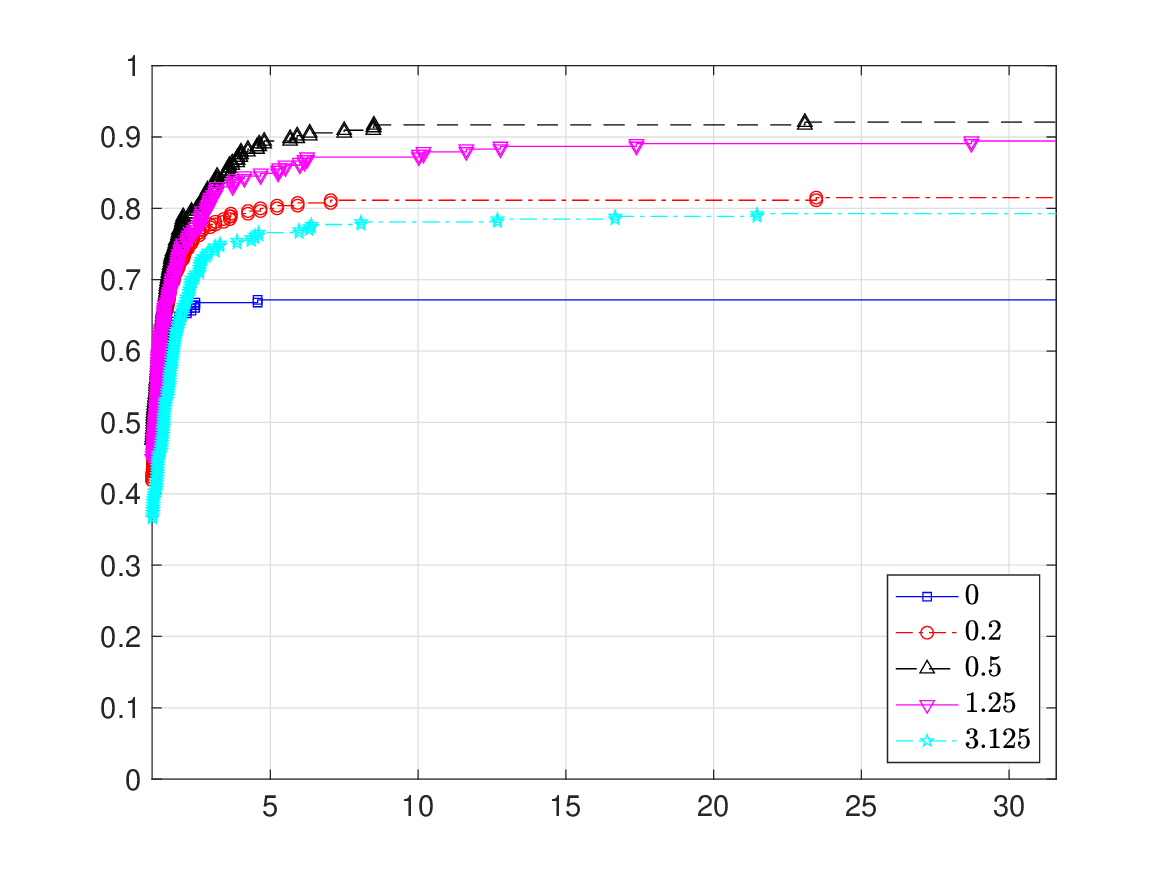}
    \caption{Performance profile with $\tau=10^{-3}$}
\end{subfigure}%
\begin{subfigure}[b]{0.42\linewidth}
    \centering
    \includegraphics[width=\linewidth]{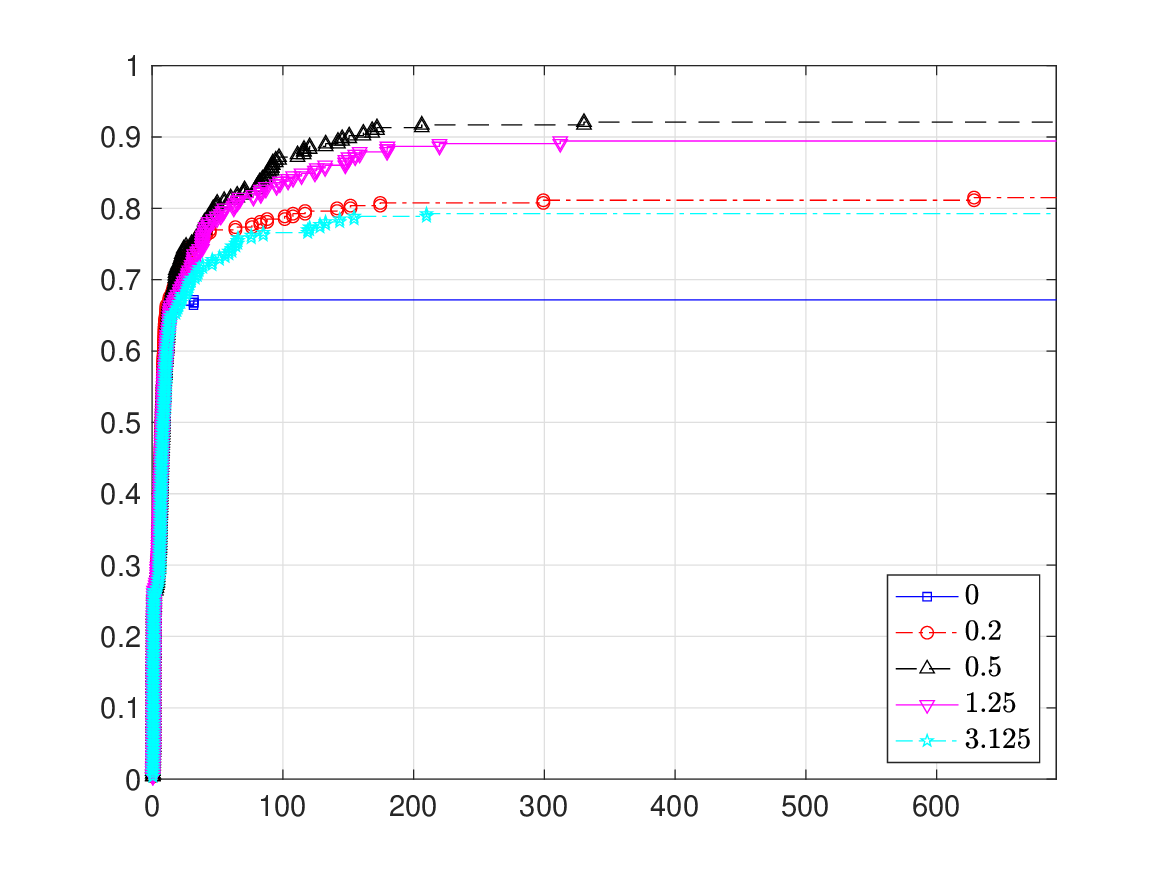}
    \caption{Data profile with $\tau=10^{-3}$}
\end{subfigure}\\
\begin{subfigure}[b]{0.42\linewidth}
    \centering
    \includegraphics[width=\linewidth]{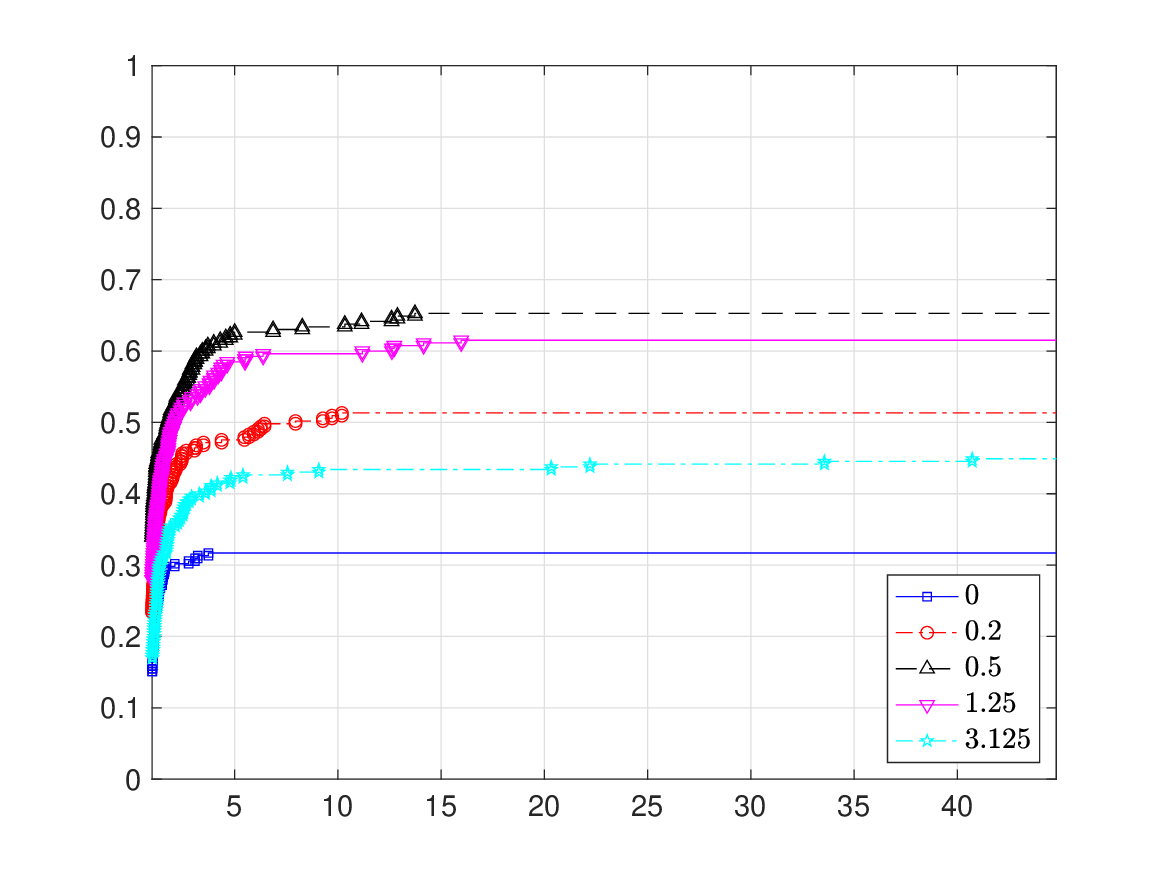}
    \caption{Performance profile with $\tau=10^{-5}$}
\end{subfigure}%
\begin{subfigure}[b]{0.42\linewidth}
    \centering
    \includegraphics[width=\linewidth]{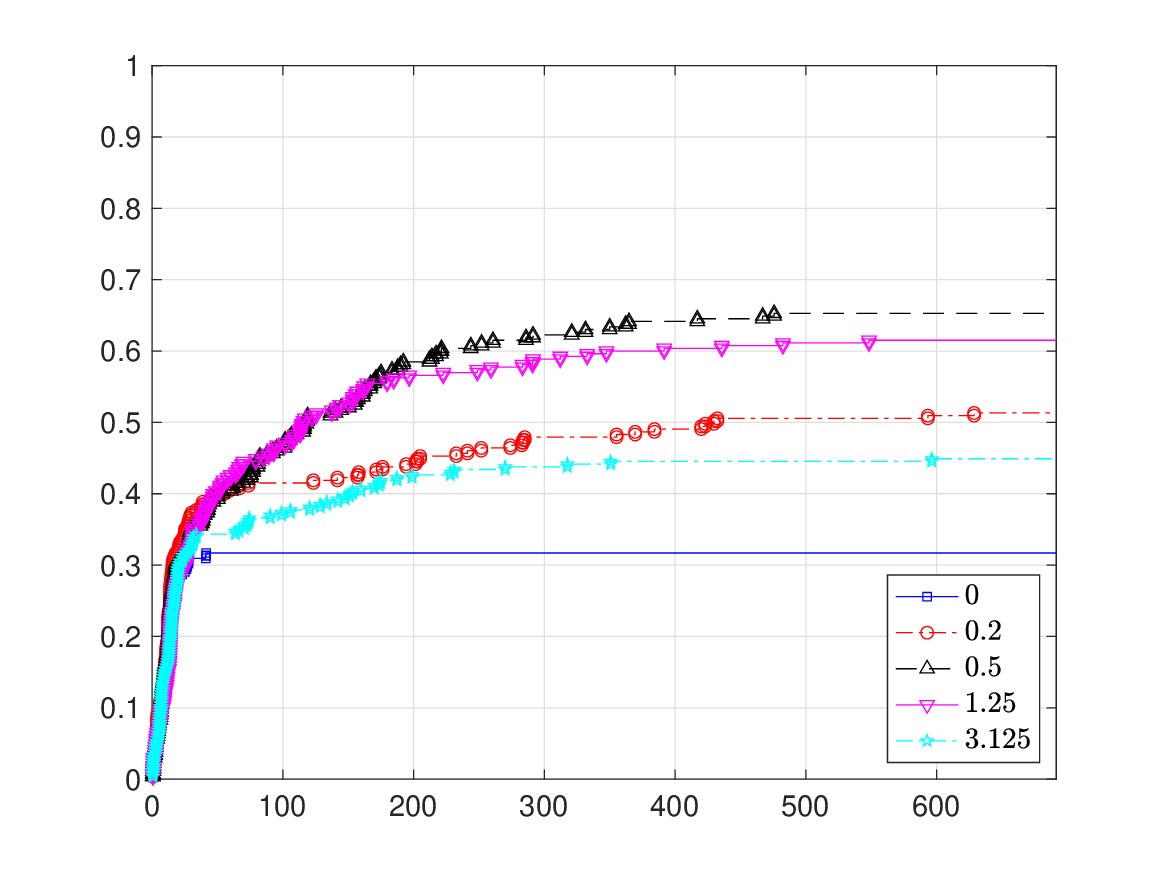}
    \caption{Data profile with $\tau=10^{-5}$}
\end{subfigure}\\
    \caption{Performance of DFO-TR with different relaxed step acceptance criteria ($r = 0, 0.2, 0.5, 1.25,$ and $3.125$) on the Mor\'{e} \& Wild problem set with subexponential function evaluation noise. }
    \label{fig:MW subexponential}
\end{figure}

\section{Conclusions}
\label{sec:conclusion}
We have proposed first- and second-order modified trust-region algorithms for solving noisy (possibly stochastic) unconstrained nonconvex continuous optimization problems. The algorithms utilize estimates of function and derivative information computed via noisy probabilistic zeroth-, first- and second-order oracles. 
The noise in these oracles is not assumed to be smaller than constants $\epsilon_f, \epsilon_g$ and $\epsilon_H$, respectively. 
We show that the first-order method (Algorithm~\ref{alg:tr}) can find an $\epsilon$-first-order stationary point with high probability after ${\cal O}(\epsilon^{-2})$ iterations for any  $\epsilon\geq[{\cal O}(\sqrt{\epsilon_f}) + {\cal O}(\epsilon_g)]$, 
and that the second-order method (Algorithm~\ref{alg:tr2}) can find an $\epsilon$-second-order critical point for any  $\epsilon\geq [{\cal O}(\sqrt[3]{\epsilon_f}) + {\cal O}(\sqrt{\epsilon_g})+ {\cal O}(\epsilon_H)]$ after ${\cal O}(\epsilon^{-3})$ iterations. 
Numerical experiments on standard derivative-free optimization problems and problems with adversarial noise illustrate the performance of the modified trust-region algorithms.

\subsection*{Acknowledgements}
\label{sec:Acknowledgements}
The authors are grateful to the Associate Editor and two anonymous referees for their valuable comments and suggestions. 

\bibliographystyle{plain}
\bibliography{references}

\appendix
\section{Proofs}
\label{app:proofs}
\subsection*{Proof to Proposition~\ref{prop.finite_diff2}}

\begin{proof}
It was shown in \cite{nesterov2006cubic} that under Assumption~\ref{ass:lip_cont2} for any $(x,y) \in \R^n \times \R^n$
\[ \left| \phi(y) - \phi(x) - \langle \nabla\phi(x), y-x\rangle - \frac{1}{2}\langle \nabla^2\phi(x) (y-x), y-x\rangle \right| \le \frac{L_2}{6} \|y-x\|^3. 
\]

It implies the following four inequalities hold for any $(i,j) \in \{1,2,\dots,n\}^2$: 
\[ \begin{aligned} 
\phi(x+\sigma u_i+\sigma u_j) &\le \phi(x) + \langle \nabla \phi(x), \sigma u_i+\sigma u_j \rangle + \frac{1}{2} \langle \nabla^2\phi(x) (\sigma u_i+\sigma u_j), \sigma u_i+\sigma u_j \rangle + \frac{L_2}{6} (\sqrt{2} \sigma )^3 \\
-\phi(x+\sigma u_i) &\le -\phi(x) - \langle\nabla\phi(x), \sigma u_i\rangle - \frac{1}{2} \langle\nabla^2\phi(x) \sigma u_i, \sigma u_i\rangle + \frac{L_2}{6} \sigma ^3 \\
-\phi(x+\sigma u_j) &\le -\phi(x) - \langle\nabla\phi(x), \sigma u_j\rangle - \frac{1}{2} \langle\nabla^2\phi(x) \sigma u_j, \sigma u_j\rangle + \frac{L_2}{6} \sigma ^3 \\
\phi(x) &\le \phi(x), 
\end{aligned} \]
which can be added together as 
\[ \phi(x+\sigma u_i+\sigma u_j) - \phi(x+\sigma u_i) - \phi(x+\sigma u_j) + \phi(x)
\le \langle \nabla^2\phi(x) u_i, u_j \rangle + \frac{(\sqrt{2}+1) L_2}{3} \sigma ^3. 
\]
With bounded noise, we have 
\[ f(x+\sigma u_i+\sigma u_j) - f(x+\sigma u_i) - f(x+\sigma u_j) + f(x) \le \langle \nabla^2\phi(x) \sigma u_i, \sigma u_j \rangle + \frac{(\sqrt{2}+1) L_2}{3} \sigma ^3 + 4\hat\epsilon_f. 
\]
Using the same argument as above, it can also be shown that 
\[ -f(x+\sigma u_i+\sigma u_j) + f(x+\sigma u_i) + f(x+\sigma u_j) - f(x) \le -\langle \nabla^2\phi(x) \sigma u_i, \sigma u_j \rangle + \frac{(\sqrt{2}+1) L_2}{3} \sigma ^3 + 4\hat\epsilon_f.  
\]
Combining the last two inequalities, we obtain
\[ |\langle H(x) u_i, u_j \rangle - \langle \nabla^2\phi(x) u_i, u_j \rangle | 
\le  \frac{(\sqrt{2}+1) L_2}{3} \sigma  + \frac{4\hat\epsilon_f}{\sigma ^2}. 
\]
Then 
\[ \|H(x) - \nabla^2\phi(x)\| \le \|H(x) - \nabla^2\phi(x)\|_F 
\le \sqrt{n^2 \left( \frac{(\sqrt{2}+1) L_2}{3} \sigma  + \frac{4\hat\epsilon_f}{\sigma ^2} \right)^2} 
= \frac{(\sqrt{2}+1) n L_2}{3} \sigma  + \frac{4n\hat\epsilon_f}{\sigma ^2}. 
\]
\end{proof}

\subsection*{Proof to Proposition~\ref{prop:subexponential}}
\begin{proof}
    By the Taylor series of the exponential functions, it follows that
    \[
    \mathbb{E} \exp(\lambda |X|) 
    = \mathbb{E} \sum_{p=0}^\infty \frac{1}{p!} (\lambda |X|)^p 
    = 1 + \sum_{p=1}^\infty \frac{1}{p!} \lambda^p \mathbb{E}|X|^p
    \]
    for any real $\lambda$. 
    Applying the integral identity \cite[Lemma 1.2.1]{HDPbook}),
    \[ \mathbb{E} |X|^p = \int_0^\infty \mathbbm{P}\{|X|^p \ge u\} {\rm d}u 
    = \int_0^\infty \mathbbm{P}\{|X|^p \ge t^p\} {\rm d} t^p 
    \le \int_0^\infty \min\{1, \exp(a(b-t))\} {\rm d} t^p, 
    \]
    which is valid for all $p > 0$. 
    Since $1 \le \exp(a(b-t))$ when $t \le b$, the above result can be written as 
    \[ \mathbb{E} |X|^p \le \int_0^b p t^{p-1} {\rm d} t + \int_b^\infty p t^{p-1} \exp(a(b-t)) {\rm d} t. 
    \]
    Thus, for all $\lambda \in [0, a)$ it follows that
    \[ \begin{aligned} 
    \mathbb{E} \exp(\lambda |X|) 
    &\le 1 + \sum_{p=1}^\infty \frac{1}{p!} \lambda^p \left[ \int_0^b p t^{p-1} {\rm d} t + \int_b^\infty p t^{p-1} \exp(a(b-t)) {\rm d} t \right] \\
    &= 1 + \lambda \int_0^b \sum_{p=1}^\infty \frac{1}{(p-1)!} (\lambda t)^{p-1} {\rm d} t + \lambda e^{ab} \int_b^\infty \sum_{p=1}^\infty \frac{1}{(p-1)!} (\lambda t)^{p-1} e^{-at} {\rm d} t \\
    &= 1 + \lambda \int_0^b \exp(\lambda t) {\rm d} t + \lambda e^{ab} \int_b^\infty \exp((\lambda-a) t) {\rm d} t \\
    &= 1 + \left[ \exp(\lambda b) - 1 \right] + \lambda e^{ab} \frac{1}{a-\lambda} \exp((\lambda-a)b) \\
    &= \frac{a}{a - \lambda} \exp(\lambda b). 
    \end{aligned} \]
\end{proof}

\section{Details on the adversarial gradient estimate}
\label{app:dual details}
When $I_k=1$, problem \eqref{prob:most loss} can be written as 
\begin{equation} \label{prob:most loss reformulated} \begin{aligned}
    &\min_{y_1,y_2} &&y_1 \;\; &&\text{Maximize the loss. } \\
    &\text{s.t.} &y_1 &\ge \frac{\eta_1 y_2}{L_1} + \frac{\delta_k}{2} - \frac{2\epsilon_f+r}{L_1 \delta_k} \;\; &&\text{Step is accepted. } \\
    &~ &y_1 &\ge \frac{y_2}{2L_1} + \frac{(L_1\|x_k\|)^2 - (\kappa_{\rm eg}\delta_k + \epsilon_g)^2}{2L_1 y_2} \;\; &&\text{Gradient is sufficiently accurate.} \\
    &~ &-\|x_k\| &\le y_1 \le \|x_k\| \;\; &&\text{Ensure $|\langle x_k, g_k/\|g_k\| \rangle| \le \|x_k\|$.} \\
    &~ &y_2 &\ge \min\{10^{-6}, 10^{-2}\|L_1x_k\|\} \;\; &&\text{Ensure $\|g_k\| > 0$.}
\end{aligned} \end{equation}
If $L_1\|x_k\| \le \kappa_{\rm eg}\delta_k + \epsilon_g$, all three lower bounds on $y_1$ are monotonically non-decreasing with respect to $y_2$, so we set $y_2$ to $\min\{10^{-6}, 10^{-2}\|L_1x_k\|\}$ and $y_1$ to the largest of the three lower bounds. 
Then if $y_1 \le \|x_k\|$ holds, the problem is solved; otherwise the problem is infeasible. 
Alternatively, if $L_1\|x_k\| > \kappa_{\rm eg}\delta_k + \epsilon_g$, the second lower bound of $y_1$ becomes a positive convex function of $y_2$. 
We set $y_2$ to the minimizer of this convex function $\sqrt{(L_1\|x_k\|)^2 - (\kappa_{\rm eg}\delta_k + \epsilon_g)^2}$, and $y_1$ to its minimum value $\sqrt{\|x_k\|^2 - (\kappa_{\rm eg}\delta_k + \epsilon_g)^2 / L_1^2}$. 
Now if $y_1 \ge \delta_k/2$, the algorithm cannot be tricked into taking a bad step and we move on to problems \eqref{prob:reject1} and \eqref{prob:reject2}; otherwise all constraints are satisfied except maybe the first one, so we check it. 
If it is satisfied, the problem is solved; if not, $y_2$ needs to be reduced until the first and second lower bounds of $y_1$ are equal, which requires solving a quadratic equation. 
We set $y_2$ to the root between 0 and $\sqrt{(L_1\|x_k\|)^2 - (\kappa_{\rm eg}\delta_k + \epsilon_g)^2}$ and $y_1$ to the resulting lower bound. 
If this solution satisfies $y_1 \le \|x_k\|$, the problem is solved, otherwise it is infeasible. 

For problems \eqref{prob:reject1} and \eqref{prob:reject2}, where we try to find a value for $g_k$ that would lead to a step being rejected, if $L_1\|x_k\| \le \kappa_{\rm eg}\delta_k + \epsilon_g$, we can simply set $g_k = 0$. 
Now we assume $L_1\|x_k\| > \kappa_{\rm eg}\delta_k + \epsilon_g$. 
With the additional substitution $y_3 = \eta_1 y_2 - L_1 y_1$, problem \eqref{prob:reject1} is formulated as 
\begin{equation} \label{prob:reject1 reformulated} \begin{aligned}
    &\max_{y_2,y_3} &&y_3 \;\;&&\text{Try to get the step rejected. } \\
    &\text{s.t.} &y_3 &\le \frac{2\eta_1-1}{2}y_2 - \frac{(L_1\|x_k\|)^2 - (\kappa_{\rm eg}\delta_k+\epsilon_g)^2}{2y_2} \;\;&&\text{Gradient is sufficiently accurate.} \\
    &~ &y_3 &\ge \eta_1 y_2 - L_1\|x_k\| \;\;&&\text{Ensure $|\langle x_k, g_k/\|g_k\| \rangle| \le \|x_k\|$.} \\
    &~ &y_2 &\ge \min\{10^{-6}, 10^{-2}\|L_1x_k\|\} \;\;&&\text{Ensure $\|g_k\| > 0$.} \\
    &~ &y_3 &> \eta_1 y_2 - L_1\delta_k/2 \;\;&&\text{Step leads to loss of progress. } 
\end{aligned} \end{equation}
Note the constraint $y_3 \le \eta_1 y_2 + L_1\|x_k\|$, which ensures $y_1 = \langle x_k, g_k/\|g_k\| \rangle \ge -\|x_k\|$, is not present because it is covered by the first constraint in \eqref{prob:reject1 reformulated}.  
If $2\eta_1 < 1$, the upper bound on $y_3$ is a concave function of $y_2$. 
We set $y_2$ to its maximizer $\sqrt{[(L_1\|x_k\|)^2 - (\kappa_{\rm eg}\delta_k+\epsilon_g)^2] / (1-2\eta_1)}$ and $y_3$ to the optimal value $- \sqrt{[(L_1\|x_k\|)^2 - (\kappa_{\rm eg}\delta_k+\epsilon_g)^2](1-2\eta_1)}$. 
If this solution is feasible, the problem is solved; otherwise the upper bound is lower than at least one of the lower bounds, in which case we reduce $y_2$ until the upper and lower bounds of $y_3$ are equal. 
If $2\eta_1 \ge 1$, the upper bound on $y_3$ increases as $y_2$ increases. 
When $y_2$ is sufficiently large, the two lower bounds on $y_3$ increase faster with $y_2$ than the upper bound, so $y_2$ can only be increased until the bounds meet. 
Thus, in either case, we need to solve the quadratic equation 
\[ \begin{aligned} 
&\frac{2\eta_1-1}{2}y_2 - \frac{(L_1\|x_k\|)^2 - (\kappa_{\rm eg}\delta_k+\epsilon_g)^2}{2y_2} 
= \eta_1 y_2 - L_1\min\{ \|x_k\|, \delta_k/2 - 10^{-7}\} \\
&y_2^2 - 2L_1\min\{ \|x_k\|, \delta_k/2 - 10^{-7}\} y_2 + (L_1\|x_k\|)^2 - (\kappa_{\rm eg}\delta_k+\epsilon_g)^2 = 0, 
\end{aligned} \] 
where the $10^{-7}$ is there to deal with the strict inequality. 
The optimal value for $y_2$ should be its larger root, and $y_3$ is the corresponding bound. 
If there is no real root, the problem is infeasible. 

With the substitutions, problem \eqref{prob:reject2} is formulated as 
\begin{equation} \begin{aligned}
    &\max_{y_2,y_3} &&y_3 &&\text{Try to get the step rejected. } \\
    &\text{s.t.} &y_3 &\le \frac{2\eta_1-1}{2}y_2 - \frac{(L_1\|x_k\|)^2 - (\kappa_{\rm eg}\delta_k+\epsilon_g)^2}{2y_2} &&\text{The gradient is sufficiently accurate.} \\
    &~ &y_3 &\ge \eta_1 y_2 - L_1\|x_k\| &&\text{To ensure $|\langle x_k, g_k/\|g_k\| \rangle| \le \|x_k\|$.} \\
    &~ &y_2 &> 0 &&\text{To ensure $\|g_k\| > 0$.} \\
    &~ &y_3 &\le \eta_1 y_2 - L_1\delta_k/2 &&\text{The step leads to a gain of progress.} 
\end{aligned} \end{equation}
Assume $\|x_k\| \ge \delta_k/2$ for feasibility. 
If $2\eta_1 < 1$, we first set $y_2,y_3$ to the maximizer and maximum value of the concave right-hand side of the first constraint. 
Then if this solution violates the last constraint, we set $y_2,y_3$ to the larger one of the two points where the two upper bounds of $y_3$ meet. 
If this solution instead violates the second constraint or $2\eta_1 \ge 1$, we set $y_2,y_3$ to the larger one of the two points where the right-hand sides of the first two constaints are equal. 

Problem \eqref{prob:least gain} can be reformulated as \eqref{prob:most loss reformulated} but without the acceptance constraint. 
Since we have explained how to analytically solve \eqref{prob:most loss reformulated}, it should be clear how to solve the simpler \eqref{prob:least gain}. 
Same goes for problem \eqref{prob:most loss} without the sufficiently accurate gradient constraint.

The optimal $g_k$ needs to be recovered from $y_1,y_2$. 
We let $g_k = \alpha_1 x_k + \alpha_2 v$, where $\alpha_1,\alpha_2$ are real variables, and $v\in\R^n$ is a unit vector with a random direction. 
By solving the system of equations
\begin{equation} \begin{aligned} 
    \langle x_k, \alpha_1 x_k + \alpha_2 v \rangle &= y_1 y_2 \\
    \|\alpha_1 x_k + \alpha_2 v\| &= y_2, 
\end{aligned} \end{equation}
we obtain the values for $\alpha_1, \alpha_2$, and hence can recover $g_k$.

\end{document}